\documentclass[11pt,letterpaper]{amsart}

\usepackage{amsmath,amssymb,amsthm,verbatim, csquotes,mathrsfs,stmaryrd}
\usepackage{hyperref}
\usepackage{xcolor}
\usepackage{esint}
\usepackage{mathtools}
\usepackage{graphicx}
\usepackage{float}
\usepackage{caption}
\mathtoolsset{showonlyrefs}
\usepackage{geometry}

\usepackage{bm}
\usepackage{pdfpages}
\usepackage{import}
\usepackage[shortlabels]{enumitem}

\newtheorem{theorem}{Theorem}[section]%
\newtheorem{lemma}[theorem]{Lemma}%
\newtheorem{proposition}[theorem]{Proposition}%
\newtheorem{definition}[theorem]{Definition}%
\newtheorem{corollary}[theorem]{Corollary}%
\newtheorem{remark}[theorem]{Remark}

\newtheorem{example}[theorem]{Example}

\def\om{\omega}

\def\p{\partial}
\def\ep{\epsilon}
\def\de{\delta}
\def\De{\Delta}

\def\S{{\Sigma}}
\def\<{\langle}
\def\>{\rangle}

\def\na{\nabla}

\def\dist{{\rm dist}}
\def\spt{{\rm spt}}

\providecommand{\abs}[1]{\lvert#1\rvert}
\providecommand{\Abs}[1]{\left\lvert#1\right\rvert}

\providecommand{\norm}[1]{\lVert#1\rVert}

\newcommand{\mbC}{\mathbb{C}}

\newcommand{\mbF}{\mathbb{F}}

\newcommand{\mbH}{\mathbb{H}}

\newcommand{\mbN}{\mathbb{N}}

\newcommand{\mbR}{\mathbb{R}}
\newcommand{\mbS}{\mathbb{S}}

\newcommand{\mcA}{\mathcal{A}}
\newcommand{\mcB}{\mathcal{B}}

\newcommand{\mcH}{\mathcal{H}}

\newcommand{\mcL}{\mathcal{L}}

\newcommand{\mcQ}{\mathcal{Q}}

\newcommand{\mfc}{\mathbf{c}}

\newcommand{\mfC}{\mathbf{C}}

\newcommand{\mfR}{\mathbb{R}}

\newcommand{\msB}{\mathscr{B}}

\newcommand{\rd}{{\rm d}}

\newcommand{\bseta}{\boldsymbol\eta}

\newcommand{\ra}{\rightarrow}

\newcommand{\eq}[1]{\begin{equation}\begin{alignedat}{2} #1 \end{alignedat}\end{equation}}

\numberwithin{equation} {section}

\begin{document}

\title[Regularity of Stable Capillary Minimal Hypersurfaces]{Regularity of stable capillary minimal hypersurfaces}
\date{\today}

\author[Wang]{Gaoming Wang}
\address[G.W]{Beijing Institute of Mathematical Sciences and Applications\\Huairou District\\101408\\Beijing\\China}
\email{wanggaoming@bimsa.cn}

\author[Zhang]{Xuwen Zhang}
	\address[X.Z]{Mathematisches Institut\\Universit\"at Freiburg\\Ernst-Zermelo-Str.1\\79104\\Freiburg\\ Germany}
\email{xuwen.zhang@math.uni-freiburg.de}

\begin{abstract}
We develop a regularity and compactness theory for stable capillary minimal hypersurfaces in the half-space $\mathbb{H}^{n+1}$ with contact angle $\theta \in (0,\pi)$ and dimension $n \geq 2$.
One key analytic ingredient is a capillary differential Schoen inequality, which
allows us to establish a boundary sheeting theorem in the spirit of Bellettini \cite{Bellettini25}.
The other ingredient is a refined classification of stable capillary minimal cones, and we show that for any contact angle $\theta\in(0,\pi)$, the stable capillary minimal hypercone in $\mathbb H^5$ with an isolated singularity must be flat.
As a consequence, for any $n\leq4$ and $\theta\in(0,\pi)$, we obtain the Bernstein theorem for embedded complete stable capillary minimal hypersurfaces in $\mathbb H^{n+1}$ with Euclidean area growth.

\

\end{abstract}

\maketitle
\tableofcontents
\section{Introduction}

Let $n\geq2, \theta\in(0,\pi)$, and let $\mbH^{n+1}\coloneqq\left\{x=(x_1,\cdots,x_{n+1})\in\mbR^{n+1}: x_1>0\right\}$ denote the Euclidean half-space. Given an open Caccioppoli set $E\subset\mbH^{n+1}$, the \emph{capillary free energy} is given by
\eq{\label{defn:A-function}
\mcA_\theta\left(E\right)
=\mcH^n\left(\p^\ast E\cap\mbH^{n+1}\right)-\cos\theta\,\mcH^n\left(\p^\ast E\cap\p\mbH^{n+1}\right),
}
where $\mcH^n$ is the $n$-dimensional Hausdorff measure. This functional models the total interface energy of a liquid droplet $E$ resting on a flat solid wall $\p\mbH^{n+1}$, where the coefficient $\cos\theta$ encodes the relative wettability of the solid surface, a quantity classically governed by Young's law.

Assuming $M\coloneqq\p^\ast E\cap\mbH^{n+1}$ is a smooth hypersurface, the Euler-Lagrange equation for critical points of $\mcA_\theta$ takes the form
\eq{
\begin{cases}
    H=0&\text{ on }M,\\
    \langle\nu,e_1\rangle=\cos\theta&\text{ on }\p M,
\end{cases}
\label{eq:Youngs-law}
}
by the classical Young's law,
where $H$ is the mean curvature and $\nu$ is the outer unit normal with respect to $E$. We call $M$ a \emph{capillary minimal hypersurface} with contact angle $\theta$. If $E$ is moreover a stable critical point, then by classical second variation computations (cf.\ \cite{RS97}), $M$ satisfies the \emph{stability inequality}
\eq{\label{ineq:SS81-(1.17)-smooth-case}
\int_M\abs{A}^2\varphi^2\,\rd\mcH^n
\leq\int_M\abs{\na\varphi}^2\,\rd\mcH^n
+\cot\theta\int_{\p M} A(\eta,\eta)\varphi^2\,\rd\mcH^{n-1},
}
for any $\varphi\in C^1_c(\mbR^{n+1})$, where $A$ is the second fundamental form, $\na$ is the tangential gradient along $M$, and $\eta$ is the outer unit co-normal of $\p M\subset M$.

The classical interior regularity and compactness theory for embedded stable minimal hypersurfaces is by now well understood.
Schoen-Simon-Yau established curvature estimates for embedded stable minimal hypersurfaces in dimensions $n\leq 5$ \cite{Schoen1975curvature}, which imply the compactness of the stable minimal hypersurfaces in the smooth topology. 
Schoen-Simon extended the regularity theory to all dimensions in the embedded setting through their celebrated regularity work \cite{SS81}, and Wickramasekera later provided a definitive treatment \cite{wickramasekera2008regularity2density,Wickramasekera14}.
In contrast, the problem of optimal regularity in the immersed setting remains open.
Recently, Bellettini \cite{Bellettini25} extended Schoen-Simon-Yau's curvature estimates to the case $n\leq 6$.
Subsequently, regularity theories for immersed stable minimal hypersurfaces in all dimensions were established by the first author with Hong and Li under a non-optimal assumption on the size of the singular set \cite{HLW2024deltaStable}, and more recently by Minter--Xiao \cite{MX26} under the assumption that the singular set has vanishing \((n-2)\)-dimensional Hausdorff measure.
Important applications of regularity and compactness theory include curvature estimates and Bernstein-type results for stable minimal hypersurfaces in $\mbR^{n+1}$ under volume growth assumptions.
Recent breakthroughs have focused on removing this constraint, see \cite{Chodosh2021stable,chodosh2023stable,Catino2022stable,ChodoshLi2024Bernstein,Mazet2024Stable,CCMMR26,Stryker26,LX26}.

The existence and regularity of capillary minimal hypersurfaces have been extensively studied through direct minimization in the framework of geometric measure theory; see, for example, \cite{Taylor1977regularity,DePM15,CEL25}.
{These hypersurfaces are substantially used to study geometric problems, see e.g. \cite{Li20,CW23,CW24,EK24,CW25,Wu25,KoYao2024comparison,KY26,EK26}.}
However, the minimizing hypothesis is quite restrictive. In many geometric and analytic applications, particularly in capillary min–max theory developed by Li–Zhou–Zhu \cite{LZZ24} and De Masi–De Philippis \cite{DeMasi2021CapillaryMinMax}, one encounters surfaces that are merely \emph{stable} rather than minimizing. This motivates the development of a regularity theory for more general capillary minimal hypersurfaces.
One important direction is the Allard-type boundary regularity for stationary varifolds in the capillary setting, which was recently established by De Masi–Edelen–Gasparetto–Li \cite{DEGL25} and the first author \cite{Wang2024Allard}.

The boundary regularity of stable capillary minimal hypersurfaces is substantially more subtle than the interior case.
The main difficulty arises from the non-trivial boundary term $\cot\theta\int_{\p M}A(\eta,\eta)\varphi^2$ in the stability inequality \eqref{ineq:SS81-(1.17)-smooth-case}, which prevents direct application of standard interior techniques. Prior to this work, known results were limited to the two-dimensional case $n=2$, due to Hong–Saturnino \cite{Hong2021capillary}, Li–Zhou–Zhu \cite{LZZ24}, and De Masi–De Philippis \cite{DeMasi2021CapillaryMinMax}, as a consequence of their curvature estimate.
For the special case $\theta=\frac{\pi}{2}$ (the \emph{free boundary} case), curvature estimates and a generalized Bernstein theorem for stable free boundary minimal hypersurfaces in $\mbH^{n+1}$ were established by Guang-Li-Zhou \cite{GLZ20} under a volume growth assumption.

In this paper, we develop a complete regularity and compactness theory for stable capillary minimal hypersurfaces for $\theta\in(0,\pi)$ and $n\geq 2$.
We work in the \emph{almost embedded} (or \emph{$\theta$-regular}) sense, introduced in the following definitions.

For a pair of varifolds $(V,W)$, where $V$ is an integral $n$-varifold supported on $\overline{\mbH^{n+1}}$ and $W$ is an integral $n$-varifold supported on $\p\mbH^{n+1}$, we define the capillary energy
\[
\mbF_\theta(V,W)
\coloneqq\norm{V}(\overline{\mbH^{n+1}})-\cos\theta\norm{W}(\p\mbH^{n+1}).
\]
\begin{definition}\label{defn:stationary-pair}
\normalfont
Given a relatively open subset $U\subset\overline{\mbH^{n+1}}$, we say that any pair $(V,W)$ as above is {\em stationary for $\mbF_\theta$ in $U$} if
\eq{
\de_{\mbF_\theta}(V,W)(\psi)=0,
}
for any $\psi\in C^1(\mbR^{n+1};\mbR^{n+1})$ tangential to $\p\mbH^{n+1}$, compactly supported in $U$.
\end{definition}

\begin{example}\label{example:capillary-cone-planes}
For any $\gamma\in(-\pi,\pi)$, define the constant vector field
\eq{\label{defn:nu_theta}
\nu_\gamma
\coloneqq\cos\gamma e_1+\sin \gamma e_{n+1}.
}

Let $\theta\in(0,\pi)$. We define the unit vectors $\nu_\theta, \nu_{-\theta}$ by \eqref{defn:nu_theta},
and let $P_\theta, P_{\pi-\theta}\in G(n,n+1)$ denote, respectively, the $n$-planes perpendicular to $\nu_\theta$ and $\nu_{-\theta}$.
The half-$n$-planes truncated by $\mbH^{n+1}$ are denoted by $H_\theta=P_\theta\cap\mbH^{n+1}$, $H_{\pi-\theta}=P_{\pi-\theta}\cap\mbH^{n+1}$.
The model for our regularity theorem is the multiplicity-one $n$-varifold given by
\eq{
\mathbf{C}
\coloneqq\abs{H_{\theta}}+\abs{H_{\pi-\theta}}.
}
\end{example}

\begin{remark}
\normalfont
Let $V$ be an integral $n$-varifold supported on $\overline{\mbH^{n+1}}$.
We recall that the regular set of $\spt\norm{V}$, in the classical sense, is defined to be the collection of all $X\in\spt\norm{V}$ such that there exists $r_X>0$ with $\spt\norm{V}\cap B_{r_X}(X)$ a connected, embedded, $C^2$-hypersurface, possibly with boundary contained in $\p\mbH^{n+1}$.
In particular, when the boundary exists, we call it the {\em regular boundary part}.
We point out the following two facts:
\begin{itemize}
    \item For the cone $\mathbf{C}$ considered in Example \ref{example:capillary-cone-planes}, its cone spine is not a regular boundary part. 
    \item Given a $\mbF_\theta$-stationary pair $(V,W)$, we cannot conclude from stationarity that {\em Young's law} \eqref{eq:Youngs-law} holds along the regular boundary part of $V$, since $V$ and $W$ could have multiplicity $\geq 2$.
    For an example indicating this fact, cf. \cite[Appendix 1]{Zhang24}.
\end{itemize}

\end{remark}

This motivates the following definition.

\begin{definition}[$\theta$-regular point]\label{Defn:theta-regular-points-manifold}
\normalfont
Let $n\geq2, \theta\in(0,\pi)$.
Let $V$ be a rectifiable $n$-varifold supported on $\overline{\mbH^{n+1}}$.
A point $X\in\spt\norm{V}$ is called a $\theta$-\emph{regular point}, denoted as $X\in {\rm Reg}_\theta V$,
if there exists $\rho>0$ such that one of the following holds:
\begin{enumerate}[label=\textup{(\roman*)}]
    \item
    $\spt\norm{V}$ is a $C^2$, embedded hypersurface without boundary in $B_\rho(X)$;
    \item \label{it:defThetaBoundary}
    $X\in\partial \mathbb{H}^{n+1}$, and for some $N=N(X)\in\mbN$,
    \eq{
        \spt\norm{V}\cap B_\rho(X)
        =
        \bigcup_{j=1}^N \Sigma_j \cap B_\rho(X),
    }
    where each $\Sigma_j$ is an embedded 
    $C^2$-hypersurface such that the following conditions hold:
    \begin{enumerate}[label=\textup{(\alph*)}]
        \item $\partial \Sigma_j \cap B_\rho(X)\subset \partial \mathbb{H}^{n+1}$ for each $j$; \label{it:reg:onboundary}
        \item there exists a unit normal vector field $\nu_j$ of $\Sigma_j$ such that $\langle \nu_j, e_1 \rangle=\cos \theta$ on $\p\S_j$; \label{it:reg:normal}
        \item the interiors of $\Sigma_i$ and $\Sigma_j$ are disjoint in $B_\rho(X)$ for $i\neq j$;
        \item any intersection between distinct components may occur only along $\partial \mathbb{H}^{n+1}$;
        \item \label{it:defNonemptyIntersection} if $\Sigma_i\cap\Sigma_j\neq\emptyset$, then their boundaries are
        \begin{enumerate}
            \item [(e1)] either identical, and with the same induced unit normal in $\p\mbH^{n+1}$, which implies that $\Sigma_i$ and $\Sigma_j$ are identical;
            \item [(e2)] or mutually tangent within $\partial\mathbb{H}^{n+1}$, with opposite induced unit normals in $\partial \mathbb{H}^{n+1}$.
        \end{enumerate}
    \end{enumerate}
\end{enumerate}
We denote the {\em $\theta$-singular set} as ${\rm Sing}_\theta V\coloneqq\spt\norm{V}\setminus{{\rm Reg}_\theta V}$, which is relatively closed in $\spt\norm{V}$.
\end{definition}

\begin{definition}\label{defn:stable-capillary-minimal-hypersurface}
\normalfont
Let $n\geq2, \theta\in(0,\pi)$,
let $U$ be a relatively open subset in $\overline{\mbH^{n+1}}$.
Let $\iota:M \rightarrow \overline{\mbH^{n+1}}$ be a properly immersed two-sided $C^2$-hypersurface in $\overline{\mbH^{n+1}}$, with boundary $\iota(\p M)\subset\p\mbH^{n+1}$.
Let $\nu$ be the unit normal field of $M$.
We say that $M$ is a {\em capillary minimal hypersurface in $U$}, if
\eq{
    \begin{cases}
H=0 & \text{on } M\cap U, \\
\langle\nu,e_1\rangle=\cos\theta & \text{on } \p M\cap U.
\end{cases}
}

\end{definition}

Throughout the paper, we identify $M$ with its image under the immersion $\iota$, and omit the immersion map $\iota$ for simplicity.

\begin{definition}\label{defn:stability}
\normalfont
We say that $M$ is a {\em stable capillary minimal hypersurface in $U$}, if in addition, (cf. \eqref{ineq:SS81-(1.17)-smooth-case})
\eq{\label{ineq:SS81-(1.17)}
\int_{M}\abs{A}^2\varphi^2\rd\mcH^n
\leq\int_{M}\abs{\na\varphi}^2\rd\mcH^n+\cot\theta\int_{\p M} A(\eta,\eta)\varphi^2\rd\mcH^{n-1},
}
for any $\varphi\in C^1(M)$ with compact support in $U$.
\end{definition}
With these definitions, in this paper we consider the following class of capillary varifolds.
\begin{definition}\label{defn:varifold-class}
\normalfont
Let $n\geq2$, $\Lambda\in[1,\infty)$, and $\theta\in(0,\pi)$.
We define $\mathscr{V}(\theta,\Lambda)$ to be the set of all varifolds
$\overline{V}=V+\abs{\cos\theta}\,W$ in
$\overline{\mbH^{n+1}}\cap B_2(0)$ such that:
\begin{enumerate}[label=(\roman*)]
    \item $M$ is a stable capillary minimal hypersurface with contact angle $\theta$, for a suitable choice of unit normal, in $\overline{\mbH^{n+1}}\cap B_2(0)$;%
    \item $V = |M|$ is the multiplicity-$1$ varifold induced by $M$;
    \item $W$ is an integral $n$-varifold supported on $\p\mbH^{n+1}$, and
    \[
    \de V(\psi)+\abs{\cos\theta}\,\de W(\psi)=0
    \]
    for every $\psi\in C^1(\mbR^{n+1};\mbR^{n+1})$ tangential to $\p\mbH^{n+1}$ and compactly supported in $\overline{\mbH^{n+1}}\cap B_2(0)$;
    \item The energy bound holds: $(\|V\| + \abs{\cos\theta}\,\|W\|)(B_2(0)) \leq \Lambda$;
    \item The $\theta$-singular set of $V$ satisfies $\mcH^{n-2}(\mathrm{Sing}_\theta V) = 0$.
\end{enumerate}
Thus $\mathscr{V}(\theta,\Lambda)=\mathscr{V}(\pi-\theta,\Lambda)$.
We denote by
$\overline{\mathscr{V}}(\theta,\Lambda)$ the closure of
$\mathscr{V}(\theta,\Lambda)$ in the varifold topology.

\end{definition}

Our main regularity and compactness theorem is the following.
\begin{theorem}
    \label{thm:mainCompact}
    Let $n\geq2, \theta\in(0,\pi)$.  Any $\overline{V}\in\overline{\mathscr{V}}(\theta,\Lambda)$ can be represented as $\overline{V} = V + \abs{\cos\theta}\, W$, where $V$ is an integral $n$-varifold such that $\spt\norm{V}$ is a stable capillary minimal hypersurface with contact angle $\theta$, for a suitable choice of unit normal, in $\overline{\mbH^{n+1}}\cap B_2(0)$, and $W$ is an integral $n$-varifold supported on $\p\mbH^{n+1}$, such that $(V,W)$ is $\mbF_{\tilde\theta}$-stationary in $\overline{\mbH^{n+1}}\cap B_2(0)$ for $\tilde\theta=\theta$ if $\theta\in[\frac{\pi}{2},\pi)$, $=\pi-\theta$ otherwise.
    The singular set ${\rm Sing}_\theta V$ is empty if $n<n_\theta$, discrete if $n=n_\theta$, and has Hausdorff dimension at most $n-n_\theta$ if $n>n_\theta$, where
    \[
        n_\theta
        \coloneqq
        \begin{cases} 
            7, & \text{if } \left|\theta-90^\circ\right|<4.580^\circ,\\
            6, & \text{if } 4.580^\circ\leq\left|\theta-90^\circ\right|<16.664^\circ,\\
            5, & \text{if } 16.664^\circ\leq\left|\theta-90^\circ\right|<90^\circ.
        \end{cases}
    \]
\end{theorem}

Moreover, the convergence is actually smooth away from the singular set, see Theorem \ref{thm:bdry-regularity} for the precise statement.

Theorem~\ref{thm:mainCompact} also yields the following generalized Bernstein theorem.
\begin{corollary}\label{Cor:Generalized-Bernstein}
    Let $2\leq n\leq4, \theta\in(0,\pi)$.  Any properly embedded complete two-sided stable capillary minimal hypersurface $M^n$ in $\mbH^{n+1}$ with contact angle $\theta$ satisfying the Euclidean area growth condition
    \[
        \mathcal{H}^n(M\cap B_r(0))\leq C r^n,\quad\forall r>0,
    \]
    must be flat.
\end{corollary}
The result can be equivalently formulated as the following curvature estimate.
\begin{corollary}
    Let $2\leq n\leq4$, $\theta\in(0,\pi)$, and let $M^n$ be a properly embedded two-sided stable capillary minimal hypersurface in $\mbH^{n+1}\cap B_1(0)$.
    If $\mcH^n(M\cap B_1(0))\leq \Lambda$, then there exists a constant $C=C(n,\Lambda,\theta)$ such that
    \[
        \sup_{M\cap B_{\frac{1}{2}}(0)}\abs{A}^2\leq C.
    \]
\end{corollary}

\begin{remark}
\normalfont
    For every $\theta\in(0,\pi)$, the above two corollaries extend to general dimensions $2\leq n<n_\theta$.
    In particular, we push the dimension of the Bernstein-type theorem in \cite{LZZ24} up to $n\leq6$.
\end{remark}

\begin{remark}\label{rmk:minmaxExtension}
\normalfont
    Our results also provide the regularity foundation needed to extend the min-max existence theorems of Li--Zhou--Zhu \cite{LZZ24} and De Masi--De Philippis \cite{DeMasi2021CapillaryMinMax} from $3$-dimensional to $5$-dimensional manifolds with boundary.
\end{remark}

The first key ingredient of the paper is the introduction of a capillary tilt function \footnote{The construction of the capillary tilt function below motivates the subsequent work \cite{WZ26}, in which we prove a sheeting theorem for branched stable minimal immersed hypersurfaces near a stationary classical cone or a cone consisting of distinct hyperplanes whose spine has dimension at least $n-2$.}, together with a capillary analogue of the {\it Schoen (differential) inequality} \cite{Schoen77,SS81}, which enables us to use Bellettini's method \cite{Bellettini25} to prove the sheeting theorem for stable capillary minimal hypersurfaces.
Precisely, we introduce the \emph{capillary tilt function}
\eq{\label{defn:g_theta}
g_\theta
\coloneqq
\begin{cases}
    g_{\theta,1},&\text{ for }n=2,\\
    g_{\theta,\frac{1}{n-2}},&\text{ for }n\geq3,
\end{cases}
}
where for $k\in(0,1]$,
\eq{\label{defn:g_theta,k}
    g_{\theta,k}(X)
    \coloneqq\sqrt{1-\nu_1^{2}(X)-\nu_{n+1}^{2}(X)+k(\cos \theta - \nu_1(X))^{2}},\quad X\in M,
}
with $\nu_i=\langle \nu,e_i \rangle$ and $\nu$ is the unit normal vector field of $M$. Note that $g_{\theta,k}\geq 0$, with equality if and only if $\nu_1=\cos\theta$ and $\nu_{n+1}=\pm\sin\theta$, i.e., $\nu$ coincides with one of the canonical capillary unit normals $\nu_{\pm\theta}$.

\begin{theorem}[Capillary Schoen inequality]\label{thm:differential-capillary-Schoen}
Let $n\geq2$ and $\theta\in(0,\pi)$.
Suppose that $M$ is a properly immersed two-sided capillary minimal
hypersurface in a relatively open set
$U\subset\overline{\mbH^{n+1}}$.
Then there exists $\delta_{n,\theta}>0$ such that
\eq{\label{ineq:differential-capillary-Schoen}
g_\theta(\Delta+\abs{A}^2)g_\theta
\geq\delta_{n,\theta}\abs{A}^2
\quad\text{on }M\cap U\cap\{g_\theta>0\}.
}
\end{theorem}

\eqref{ineq:differential-capillary-Schoen} provides an intrinsic PDE on the stable capillary minimal hypersurface, but it leads to a non-trivial boundary condition
\eq{\label{eq:g_theta-bdry-condition}
\partial_\eta g_\theta
+\cot\theta\,A(\eta,\eta)g_\theta=0
\quad\text{on }\partial M\cap\{g_\theta>0\},
}
and causes difficulty in proving a Caccioppoli-type inequality, compared to the case without boundary  \cite[Lemma 4.1]{Bellettini25}.
To overcome this, we consider the ansatz
\eq{\label{defn:h_theta}
h_\theta(X)
\coloneqq 1-\cos\theta\nu_1(X),\quad\forall X\in M,
}
which satisfies
\eq{\label{eq:h_theta}
(\Delta+\abs{A}^2)h_\theta=\abs{A}^2\text{ on }M,\qquad
\p_\eta h_\theta+\cot\theta A(\eta,\eta)h_\theta=0\text{ on }\p M.
}
Testing the capillary stability inequality \eqref{ineq:SS81-(1.17)} with $\varphi=(g_\theta-\ell h_\theta)^+\varphi$, and then use \eqref{eq:g_theta-bdry-condition}, \eqref{eq:h_theta} to eliminate the boundary terms, we arrive at the capillary Caccioppoli-type inequality (see Lemma \ref{lem:capillary-truncated-caccioppoli} below). The capillary sheeting theorem can then be proved in the spirit of \cite{Bellettini25}.
To state it, we introduce the following notation. Let $r>0$,
\begin{itemize}
        \item $B^n_r(0)$ denotes the $n$-dimensional open ball in $\mathbb{R}^n$ with radius $r$ centered at $0$.
        We set $B^{n+}_r(0) \coloneqq B^n_r(0)\cap\{x_1>0\}$.
        We define $B^{n+}_r(x)\coloneqq B^n_r(x)\cap\{x_1>0\}$ for any $x\in\mathbb{R}^{n}$.
        \item For $X=(x,x_{n+1})\in\mbR^{n+1}$, define the region
        \[
        \mbC^\theta_r(X)
        \coloneqq\overline{\mbH^{n+1}}\cap\left(B^{n}_r(x)\times\left(x_{n+1}-(1+|\cot\theta|)r, x_{n+1}+(1+|\cot\theta|)r\right)\right),
        \]
        which is relatively open in $\overline{\mbH^{n+1}}$.
        We use the shorthand $\mbC^\theta_r\coloneqq\mbC^\theta_r(0)$ when $X=0$.
\end{itemize}

\begin{definition}%
\label{Defn:capillary-tilt-excess}
\normalfont
Let $M$ be a properly immersed two-sided $C^2$-hypersurface in $\overline{\mbH^{n+1}}$.
For any $X\in\mbR^{n+1}$, $\sigma\in\mbR$, we define the \textit{capillary tilt-excess} as
    \eq{
    E_\sigma%
    \coloneqq{\frac{1}{\sigma^n}\int_{M\cap \mbC^\theta_\sigma%
    }g_\theta^2\rd\mcH^n}.
    }
\end{definition}

\begin{theorem}\label{thm:sheetingHalf}
Let $n\geq2, \theta\in[\frac{\pi}{2},\pi)$, $\Lambda\in[1,\infty)$.
Let $ \overline{V}\in \mathscr{V}(\theta,\Lambda) $.
Denote by $M,V,W$ the corresponding hypersurface and varifolds as in Definition \ref{defn:varifold-class}.

There exists a positive constant $\ep_0\in(0,1)$, depending only on $n,\theta,\Lambda$, with the following property:
if for some $\sigma\in(0,\frac{1}{2(1+|\cot\theta|)}]$,
    \eq{\label{condi:SheetingThm-height-control}
    E_\sigma<\ep_0,
    }
then
\eq{
	\overline{M}\cap\mbC^\theta_{\frac{\sigma}{8}}
        =\mbC^\theta_{\frac{\sigma}{8}}\cap
        \left(\left(\bigcup_{j\in Q^+}{\mathrm{graph}(u^+_j)}\right)
        \cup\left(\bigcup_{j\in Q^-}{\mathrm{graph}(u^-_j)}\right)\right),
}
	where
    $u^\pm_j : B_{\frac{\sigma}{4}}^{n+}(0) \rightarrow \mathbb{R}$,
$j \in Q^\pm \coloneqq \left\{1,\cdots,q^\pm\right\}$ (note that $Q^\pm$ could be empty, corresponding to the case $q^\pm=0$) 
are smooth functions whose graphs 
\eq{
\left\{\left(x,u_j^\pm(x)\right): x\in B^{n+}_{\frac{\sigma}{4}}(0)\right\},
}
oriented by the unit normal pointing upwards for $u_j^+$ 
and downwards for $u_j^-$, are minimal and satisfy the capillary boundary condition.
If $q^\pm > 1$ then $u_j^\pm \leq u_{j+1}^\pm$ for $j=1,2,\ldots,q^\pm-1$. 
Moreover, for each $j\in Q^\pm$ there is a constant $a_j^\pm\in\mbR$ such that
\eq{\label{esti:C^2-u^pm}
\sigma^{-1}\sup_{B^{n+}_{\frac{\sigma}{8}}(0)}
\abs{u_j^\pm \pm \cot\theta\, x_1-a_j^\pm}
+ \sup_{B^{n+}_{\frac{\sigma}{8}}(0)} \abs{D u_j^\pm \pm \cot\theta\, e_1}
+ \sigma \sup_{B^{n+}_{\frac{\sigma}{8}}(0)} \abs{D^2 u_j^\pm}
\le C \left(E_\sigma\right)^{\frac12},
}
where $C=C(n,\theta,\Lambda)\in(0,\infty)$.
\end{theorem}

The second key ingredient of the paper is a classification of stable capillary minimal hypercones in $\mbH^{n+1}$ with an isolated singularity.
Motivated by recent work on stable and minimizing capillary cones \cite{CEL25,PTV25,FTW26}, together with some techniques used in \cite{HLW2024deltaStable}, we prove the following result, which yields the dimension bound for the singular set in Theorem \ref{thm:mainCompact}.

\begin{theorem}\label{thm:stable-capillary-cone-isolated-singularity}
Let $n\geq3$, and let $M$ be a stable capillary minimal hypercone in $\mbH^{n+1}$ with an isolated singularity at the origin and contact angle $\theta$ in the following range:
\begin{enumerate}[label=\textup{(\roman*)}]
    \item $n=3,4$: $\theta\in(0,\pi)$;
    \item $n=5$: $\theta\in(73.336^\circ,106.664^\circ)$;
    \item $n=6$: $\theta\in(85.420^\circ,94.580^\circ)$.
\end{enumerate}
Then $M$ is flat.
\end{theorem}
In dimension $n=4$, the classification holds for any $\theta\in(0,\pi)$, which removes the angle restriction in the classification result of Chodosh-Edelen-Li \cite[Theorem 1.2 (2)]{CEL25}.
For the proof, we break into two cases: when $\abs{\cos\theta}>\frac{1}{3}$, we show that the mean curvature of $\p M$ in $\p\mbH^{n+1}$ has a sign, so that the classification result of Pacati-Tortone-Velichkov \cite{PTV25} applies. When $\abs{\cos\theta}<\frac{1}{3}$, we use an argument in the spirit of \cite{Schoen1975curvature} (see also \cite[Appendix B]{Simon83}) to prove the result. As a by-product, the angle restriction in the improved bounds for the size of singular set in \cite[Theorem 1.1 (2)]{CEL25} can be removed:
\begin{corollary}\label{cor:minimizing-singular-set}
Let $(X^{n+1},h)$ be a smooth Riemannian manifold with boundary, let
$\beta\in C^1(\p X)$ satisfy $-1<\beta<1$, and let $g\in C^1(X)$.
Suppose that a Caccioppoli set $E\subset X$ locally minimizes
\[
\mcA_{\beta,g}(E)
\coloneqq
\mcH_h^n(\p^*E\cap\mathring X)
-\int_{\p^*E\cap\p X}\beta\,\rd\mcH_h^n
+\int_Eg\,\rd\operatorname{Vol}_h
\]
with respect to compactly supported variations.  If
$M=\overline{\p^*E\cap\mathring X}$, then
\[
\dim_{\mcH}\bigl(\operatorname{Sing}M\cap\p X\bigr)\leq n-5.
\]
\end{corollary}

\noindent{\em The rest of the paper is organized as follows.}
In Section \ref{sec:notations_and_preliminaries}, we introduce the notation and recall basic facts on capillary surfaces and varifolds. In Section \ref{sec:differential-capillary-Schoen}, we prove the capillary Schoen differential inequality. We then develop the capillary first variation formula and its consequences in Section \ref{sec:capillaryFirstVariation}, including the monotonicity formula and Ahlfors regularity. In Section \ref{sec:SheetingTheorem}, we derive the capillary Caccioppoli inequality and prove the Sheeting Theorem. Building on these results, we establish the main regularity and compactness theorem in Section \ref{sec:mainCompactness}. As an application, Section \ref{sec:BernsteinTheorem} is devoted to a generalized Bernstein theorem for stable capillary minimal hypersurfaces in $\mbH^{n+1}$. Finally, in Section \ref{sec:stableCapillaryCone}, we present a classification result for stable capillary cones with an isolated singularity.

\

\noindent{\em Acknowledgments.} 
We are grateful to Otis Chodosh and Chao Li for their interest in this work, and to Xin Zhou for pointing out Remark~\ref{rmk:minmaxExtension} to us. Part of this work was completed while the second author was visiting BIMSA. He would like to thank BIMSA for its hospitality and the first author for the kind invitation.

\noindent{\em On the use of AI.}
Generative AI tools assisted the authors in identifying the argument used to
establish the sign of the boundary mean curvature in the small-angle regime
(see \eqref{eq:n=4-positive-boundary-mean-curvature}), and were also used for
selected algebraic checks and proofreading. The authors independently verified
all arguments and take full responsibility for the completeness and correctness
of the proofs.

\section{Preliminaries}\label{sec:notations_and_preliminaries}

We adopt the following basic notations throughout the paper.
\begin{itemize}
    \item We work with the Euclidean space $\mbR^{n+1}$, with Euclidean scalar product denoted by $\langle\cdot,\cdot\rangle$, and the corresponding Levi-Civita connection denoted by $D$.
When considering the topology of $\mbR^{n+1}$, we denote by $\overline{E}$ the topological closure of a set $E\subset\mbR^{n+1}$.
We denote by $e_i$ ($i=1,\cdots,n+1$) the $i$-th coordinate basis vector of $\mbR^{n+1}$;
    \item $B_r(X)$ is the open ball in $\mbR^{n+1}$, centered at $X$ with radius $r>0$;
    \item $\mcH^k$ is the $k$-dimensional Hausdorff measure on $\mfR^{n+1}$ and $\om_k$ is the $\mcH^k$-measure of $k$-dimensional unit ball;
    \item For two sets $A,B\subset\mbR^{n+1}$, ${\rm dist}_\mcH(A,B)$ denotes the Hausdorff distance of $A,B$ in $\mbR^{n+1}$;
    \item For the definition of Caccioppoli sets (sets of finite perimeter), we refer to \cite[§14]{Simon83}.
\end{itemize}

\subsection{Capillary hypersurfaces}

Let $\nu$ be the unit normal on $M$
, and let $\eta$ be the outer unit co-normal of $\p M$ in $M$.
Then $\{\nu,\eta\}$ spans the normal bundle of $\p M$.
For each $X\in\p M$, let $\bar\nu(X)$ be the unit vector in $T_X\p M$ such that $\{\bar\nu(X),-e_1\}$ spans the same $2$-dimensional plane as $\{\nu(X),\eta(X)\}$ and has the same orientation.
Note that if $M$ is a capillary minimal hypersurface in the sense of Definition \ref{defn:stable-capillary-minimal-hypersurface}, then the boundary condition yields
\eq{\label{eq:nu_and_eta}
	\nu=\cos \theta e_1+\sin\theta \bar\nu,\quad
    \eta=-\sin \theta e_1+\cos\theta \bar\nu.
}

On $M$, we let $\na,{\rm div},\De$ denote the Levi-Civita connection, divergence, and Laplacian induced by the immersion into $\mbR^{n+1}$.
For any vector $e\in\mbR^{n+1}$, we write $e^\top=e-\langle e,\nu\rangle\nu$ for its tangential component along $M$.
Let $A$ denote the second fundamental form of $M$ in $\mbR^{n+1}$, defined by $A(\tau,\xi)=\langle D_{\tau}\xi,\nu\rangle$.

We record some known facts, which will be needed in due course.

\begin{lemma}
	Let $M$ be a capillary minimal hypersurface in the sense of Definition \ref{defn:stable-capillary-minimal-hypersurface}, and let $w$ be a constant vector field on $\mbR^{n+1}$.
	Then
	\begin{align}
		\Delta\langle\nu,w\rangle={}&-\abs{A}^2\langle\nu,w\rangle\quad\text{ on }M,%
        \label{eq:lemLaplacian}\\
		\frac{\p\langle\nu,w\rangle}{\partial \eta}={}&-A(\eta,\eta)\langle\eta,w\rangle\quad\text{ on }\p M.
		\label{eq:lemBoundaryDeri}
	\end{align}
\end{lemma}
\begin{proof}
	\eqref{eq:lemLaplacian} is exactly \cite[(2.3)]{SS81}.
	\eqref{eq:lemBoundaryDeri} follows from the well-known fact that, for a capillary hypersurface in $\mbH^{n+1}$, the outer unit co-normal is a principal direction.
\end{proof}

\begin{lemma}
Let $n\geq2, \theta \in(0,\pi)$, and let $M$ be a capillary minimal hypersurface in the sense of Definition \ref{defn:stable-capillary-minimal-hypersurface}.
Then for any compactly supported $\varphi\in C^1(M)$, there holds
\eq{\label{eq-divf-capillary}
-\sin\theta\int_{\p M}\varphi\rd\mcH^{n-1}
=\int_{M}\langle\na \varphi,e_1\rangle\rd\mcH^n.
}
\end{lemma}
\begin{proof}
The proof can be found in \cite{LZZ24,JZ24}, we include here for completeness.
Since $H=0$ on $M$,%
we have ${\rm div}\left(\varphi e_1^\top\right)=\langle\na\varphi,e_1\rangle$, hence by \eqref{eq:nu_and_eta}
\eq{
-\sin\theta\int_{\p M}\varphi\rd\mcH^{n-1}
=&\int_{\p M}\varphi\langle\eta,e_1\rangle\rd\mcH^{n-1}\\
=&\int_{M}{\rm div}\left(\varphi e_1^\top\right)\rd\mcH^n
=\int_{M}\langle\na \varphi,e_1\rangle\rd\mcH^n.
}
\end{proof}

\begin{lemma}[Michael-Simon-type inequality]\label{Lem:M-S-ineq}
Let $n\geq2, \theta\in[\frac{\pi}{2},\pi)$, $\Lambda\in[1,\infty)$.
Let $\overline{V}\in \mathscr{V}(\theta,\Lambda)$, and let $M,V,W$ be the corresponding hypersurface and varifolds as in Definition \ref{defn:varifold-class}.
Then for any non-negative function $\varphi\in \mathrm{Lip}_c(\overline{M}\cap \mathbb{C}^\theta_2)$,
\eq{\label{ineq-Michael-Simon}
\norm{\varphi}_{L^\frac{n}{n-1}(M)}%
\leq C(n,\theta)\int_M%
\abs{\na \varphi}\rd\mcH^n
}
for some positive constant $C$ depending only on $n,\theta$.
\end{lemma}
\begin{proof}
We modify \cite[Theorem 3.1]{JZ24} to allow for the presence of singularities.
We first consider non-negative functions $\varphi\in C^1(M)$ compactly supported in $\mbC^\theta_2$. 
By the classical Michael-Simon inequality \cite{Allard72,MS73}, 
\eq{
C(n)\norm{\varphi}_{L^\frac{n}{n-1}(M)}
\leq\int_{\p M}%
\varphi\rd\mcH^{n-1}+\int_M%
\abs{\na \varphi}\rd\mcH^n
}
for some positive constant $C=C(n)$.
On the other hand, by \eqref{eq-divf-capillary}
\eq{
\int_{\p M}%
\varphi\rd\mcH^{n-1}
\leq\frac{1}{\sin\theta}\int_M%
\abs{\na\varphi}\rd\mcH^n.
}
Combining, we deduce the required estimate \eqref{ineq-Michael-Simon}.

Finally, since $\mcH^{n-2}({\rm Sing}_\theta V)=0$, and $\mcH^n(M\cap\mbC^\theta_2)\leq\left(\norm{V}-\cos\theta\norm{W}\right)(\mbC^\theta_2)\leq\Lambda$, we can use the standard approximation argument as in \cite{SS81,wickramasekera2008regularity2density}
to show that the required estimate holds for any non-negative Lipschitz function $\varphi$ compactly supported in $\mbC^\theta_2$.
This completes the proof.
\end{proof}

\subsection{Varifolds}
We use the notation and terminology in \cite{Simon83}.
Recall that an $n$-rectifiable varifold $V$ in $U$ is a positive Radon measure on the trivial Grassmannian bundle $U\times G(n,n+1)$ of the form
\eq{
V(\phi(X,P))
=\int_{R_V}\phi(X,T_XR_V)\theta_{V}(X)\rd\mcH^n(X),\quad\forall\phi\in C^0_c(U\times G(n,n+1)),
}
where $R_V$ is an $n$-rectifiable set in $U$, $\theta_V$ is a non-negative $\mcH^n\llcorner R_V$-measurable function.
The weight measure of $V$ is defined as $\norm{V}\coloneqq\pi_\ast V$, where $\pi:U\times G(n,n+1)\ra U$ is the canonical projection, and $\pi_\ast(\cdot)$ denotes the {\em push-forward} of measure through $\pi$.
$V$ is called {\em integral} if in addition, $\theta_V\in\mbN$ at $\norm{V}$-a.e.
If $\S$ is a $k$-dimensional Lipschitz submanifold of $U$, we write $\abs{\S}=\mcH^k\llcorner\S\otimes T_X\S$ for the multiplicity-one varifold naturally induced by $\S$.

For any Borel set $\Omega\subset U$, we denote by $V\llcorner\Omega$ the {restriction} of $V$ to $\Omega\times G(n,n+1)$.
By \textit{support} of $V$ we mean ${\rm spt}\, \norm{V}$, which is the smallest closed subset $B\subset\mfR^{n+1}$ such that $V\llcorner(\mfR^{n+1}\setminus B)=0$.
For any diffeomorphism $f:U\ra\mfR^{n+1}$, the continuous \textit{push-forward map} $f_\#V$ is defined as in \cite[(39.1)]{Simon83}.
Note that this is not the push-forward of Radon measures introduced above,
therefore we adopt different notations.
If $\varphi\in C^1_c(U;\mbR^{n+1})$ generates a one-parameter family of diffeomorphisms $\Phi_t$ of $\mfR^{n+1}$,
then $(\Phi_t)_\#V$ and its \textit{first variation} with respect to $\varphi$ is, see \cite[(4.2), (4.4)]{Allard72},
\eq{\label{defn-1st-variationformula}
    \delta V(\varphi)
    \coloneqq\frac{\rd}{\rd t}_{|_{t=0}}\norm{(\Phi_t)_\# V}(\mfR^{n+1})
    =\int_{U\times G(n,n+1)}{\rm div}_P\varphi(x)\rd V(x,P),
}
where ${\rm div}_P\varphi(x)=\sum_i\langle D_{e_i}\varphi, e_i\rangle$ and $\{e_1,\ldots,e_n\}\subset P$ is any orthonormal basis.
We say that $V$ has \textit{locally bounded first variation in $U$}, if for any compact set $K\subset U$,
\eq{
\sup\{\abs{\de V(\varphi)}:\varphi\in C^1(K;\mbR^{n+1}), \abs{\varphi}\leq1\}\leq C(K)<+\infty.
}

Following \cite[Definition 42.3]{Simon83}, we denote ${\rm VarTan}(V,X)$ to be the set of \textit{varifold tangents} of $V$ at  $X\in{\rm spt}\norm{V}$.
By the compactness of Radon measures,%
${\rm VarTan}(V,X)$ is compact and non-empty provided that the upper density $\Theta^{\ast n}(\norm{V},X)\coloneqq\limsup_{r\searrow0}\frac{\norm{V}(B_r(X))}{\om_nr^n}$ is finite.
Moreover, there exists a non-zero element in ${\rm VarTan}(V,X)$ if and only if $\Theta^{\ast n}(\mu_V,X)>0$.

\subsection{Free boundary varifolds}

\begin{definition}[stationary free boundary varifold]\label{Defn:stationary-free-bdry-vfld}
\normalfont
Let $U$ be a relatively open set of $\overline{\mbH^{n+1}}$.
We call an $n$-rectifiable varifold $V$ stationary with free boundary in $U\subset\overline{\mbH^{n+1}}$, if
\eq{
\de V(\varphi)
=0,
}
for any $\varphi\in C^1(\mbR^{n+1};\mbR^{n+1})$ tangential to $\p\mbH^{n+1}$, compactly supported in $U$.
\end{definition}

By \cite{GJ86,DeMasi21}, these varifolds have locally bounded first variation.

\begin{proposition}\label{Prop:bdd-1st-variation-free-bdry-vfld}
Let $U$ be a relatively open set of $\overline{\mbH^{n+1}}$, and let $V$ be an $n$-rectifiable varifold which is stationary with free boundary in $U$.
Then $V$ has locally bounded first variation in $U$.
Precisely, there exists a Radon measure $\sigma_V\perp\norm{V}$ on $U$ supported in $\p\mbH^{n+1}$, such that for any $\varphi\in C^1_c(U;\mbR^{n+1})$, we have
\eq{
\de V(\varphi)
=-\int_{\p\mbH^{n+1}}\langle\varphi,e_1\rangle\rd\sigma_V,
}
with $\sigma_V$ satisfying the local estimate
\eq{
\sigma_V\left(B_\frac{r}{2}(X)\right)
\leq\frac{C}{r}\norm{V}\left(B_r(X)\right),\quad\text{whenever }B_r(X)\cap\mbH^{n+1}\subset U,
}
where $C>0$ is an absolute constant.
\end{proposition}

\section{The differential capillary Schoen inequality}%
\label{sec:differential-capillary-Schoen}

In this section we prove Theorem \ref{thm:differential-capillary-Schoen}.
Throughout, $M$ denotes a capillary minimal hypersurface satisfying the
assumptions of that theorem.

We shall establish the theorem for $g_{\theta,k}$ with $k$ in a suitable range, and our starting point is as follows.
\begin{lemma}\label{lem:basic_formulas}
    The function $g_{\theta,k}$ satisfies,
    \begin{align}\label{eq:lemBoundaryG}
        \frac{\partial g_{\theta,k}^{2}}{\partial \eta}
        = -2\cot \theta\, A(\eta,\eta)\, g_{\theta,k}^{2},\quad\text{ on }\p M;%
    \end{align}
and on $M$:%
    \begin{align}
        g_{\theta,k} \Delta g_{\theta,k}
        &= |A|^{2}\bigl( k\cos \theta\,\nu_1+(1-k)\nu_1^{2}+\nu_{n+1}^{2} \bigr)\\
        &\quad - (1-k)|\nabla \nu_1|^{2}- |\nabla \nu_{n+1}|^{2} 
        - \frac{\Abs{\nabla \nu \cdot (\nu_1 e_1 + \nu_{n+1} e_{n+1} + k(\cos \theta - \nu_1)e_1)}^{2}}{g_{\theta,k}^{2}}.
        \label{eq:lemDeltaG}
    \end{align}
\end{lemma}
\begin{proof}

Using \eqref{eq:lemBoundaryDeri} and the capillary boundary condition \eqref{eq:nu_and_eta}, we compute
\eq{
\frac{\p g^2_{\theta,k}}{\p\eta}
=&A(\eta,\eta)\left(\left(2(1-k)\nu_1+2k\cos\theta\right)\langle\eta,e_1\rangle+2\nu_{n+1}\langle\eta,e_{n+1}\rangle\right)\\
=&A(\eta,\eta)\left(-2(1-k)\sin\theta\cos\theta-2k\sin\theta\cos\theta+2\sin\theta\cos\theta\langle\bar\nu,e_{n+1}\rangle^2\right)\\
=&-2\sin\theta\cos\theta A(\eta,\eta)(1-\langle\bar\nu,e_{n+1}\rangle^2).
}
On the other hand, by \eqref{eq:nu_and_eta} we find
\eq{
-2\cot\theta A(\eta,\eta)g^2_{\theta,k}
=&-2\cot\theta A(\eta,\eta)\left(1+k\cos^2\theta-(1-k)\cos^2\theta-2k\cos^2\theta-\sin^2\theta\langle\bar\nu,e_{n+1}\rangle^2\right)\\
=&-2\sin\theta\cos\theta A(\eta,\eta)(1-\langle\bar\nu,e_{n+1}\rangle^2)
=\frac{\p g^2_{\theta,k}}{\p\eta},
}
proving \eqref{eq:lemBoundaryG}.

To show \eqref{eq:lemDeltaG}, note that
\eq{
\na(g^2_{\theta,k})
=-2\na\nu\cdot\left(\nu_1e_1+\nu_{n+1}e_{n+1}+k(\cos\theta-\nu_1)e_1\right).
}
Hence
\eq{\label{eq:nabla-g-theta-k}
\abs{\na g_{\theta,k}}^2
=\frac{\abs{\na(g^2_{\theta,k})}^2}{4g^2_{\theta,k}}
=\frac{\Abs{\na\nu\cdot\left(\nu_1e_1+\nu_{n+1}e_{n+1}+k(\cos\theta-\nu_1)e_1\right)}^2}{g^2_{\theta,k}},
}
and by \eqref{eq:lemLaplacian}
\eq{
\frac{1}{2}\De g^2_{\theta,k}
=&-(1-k)\frac{1}{2}\De\nu_1^2-k\cos\theta\De\nu_1-\frac{1}{2}\De\nu^2_{n+1}\\
=&-(1-k)\left(-\abs{A}^2\nu_1^2+\abs{\na\nu\cdot e_1}^2\right)+k\cos\theta\nu_1\abs{A}^2-\left(-\abs{A}^2\nu_{n+1}^2+\abs{\na\nu\cdot e_{n+1}}^2\right)\\
=&\abs{A}^2\left(k\cos\theta\nu_1+(1-k)\nu_1^2+\nu_{n+1}^2\right)-(1-k)\abs{\na\nu_1}^2-\abs{\na\nu_{n+1}}^2.
}
Since $g_{\theta,k}\De g_{\theta,k}=\frac{1}{2}\De g^2_{\theta,k}-\abs{\na g_{\theta,k}}^2$, combining the above we thus obtain \eqref{eq:lemDeltaG}.
\end{proof}

The crucial step is to estimate the gradient terms appearing in~\eqref{eq:lemDeltaG}.
To this end, fix a $2$-plane $\mathscr{P}$ containing $e_1^\top$ and $e_{n+1}^\top$, and choose an orthonormal basis $\{\tau_i\}_{i=1}^{n}$ of $TM$,%
such that $\mathscr{P}$ is spanned by $\tau_1$ and $\tau_2$ with $A(\tau_1,\tau_2)=0$.
For simplicity we write $A_{ij}=A(\tau_i,\tau_j)$.

We parametrize the projections of three vectors onto $\mathscr{P}$ using angles $\xi_i \in [0,2\pi]$:
\eq{
    \vec{a}_1
    \coloneqq{}& \sqrt{1-k}e_1^\top
    =\abs{\vec{a}_1}(\cos \xi_1 \tau_1 + \sin \xi_1 \tau_2),\\
    \vec{a}_2
    \coloneqq{}& e_{n+1}^\top
    =\abs{\vec{a}_2}(\cos \xi_2 \tau_1 + \sin \xi_2 \tau_2),\\
    \vec{a}_3
    \coloneqq{}& \frac{\nu_1 e_1^\top + \nu_{n+1} e_{n+1}^\top + k(\cos \theta - \nu_1)e_1^\top}{g_{\theta,k}}
    =\abs{\vec{a}_3}(\cos \xi_3 \tau_1 + \sin \xi_3 \tau_2).
}
Our goal is to determine the smallest constant $s$ such that
\eq{
    \sum_{i=1}^3\abs{\nabla \nu \cdot \vec{a}_i}^{2}
    \le s\abs{A}^{2}.
}
The following proposition provides an explicit formula for this optimal constant.

\begin{proposition}\label{prop:boundS}
    Define $s(k,n,\theta)$ by
    \eq{
        s(k,n,\theta)
        =\frac{(n-1)(2-k\sin^{2}\theta)+\sqrt{4(1-k\sin^{2}\theta)+(n-1)^{2}k^{2}\sin^4\theta}}{2n}.
    }
    Then, for all contact angles $\theta\in (0,\pi)$ and dimensions $n\ge 2$, we have
    \eq{
        \sum_{i=1}^{3}|\nabla \nu\cdot \vec{a}_i|^{2}\le s(k,n,\theta)|A|^{2}.
    }
\end{proposition}
Here we point out that, when $k=1$ and $\theta=\frac{\pi}{2}$, the estimate reduces to \cite[(2.7)]{SS81}.
We divide the proof of Proposition~\ref{prop:boundS} into the following lemmas:
\begin{lemma}\label{lem:gradient-nu-estimates}
Define the quantities
\eq{
\msB_1
=\sum_{i=1}^{3}|\vec{a}_i|^{2}\cos^{2} \xi_i,\quad \msB_2
=\sum_{i=1}^{3}|\vec{a}_i|^{2}\sin^{2} \xi_i.
}
Then
\eq{
 \sum_{i=1}^{3}|\nabla \nu \cdot \vec{a}_i|^{2}
 \le{}& \msB_1 A_{11}^{2} + \msB_2 A_{22}^{2} + (\msB_1+\msB_2) \sum_{j=3}^{n} (A_{1j}^{2}+A_{2j}^{2}).
}
\end{lemma}

\begin{proof}
By direct computation using the fact that $A_{12}=A_{21}=0$, we obtain
\begin{align*}
    \sum_{i=1}^{3}|\nabla \nu \cdot \vec{a}_i|^{2}={}& \sum_{i=1}^{3}|\vec{a}_i|^{2}\left|A_{11}\cos\xi_i \tau_1+A_{22}\sin \xi_i \tau_2+ \sum_{j=3}^{n}(A_{1j}\cos \xi_i+A_{2j}\sin\xi_i)\tau_j\right|^{2}\\
    ={}& \msB_1 A_{11}^{2} + \msB_2 A_{22}^{2} + \sum_{i=1}^{3}|\vec{a}_i|^{2}\sum_{j=3}^{n} \left( A_{1j}\cos \xi_i + A_{2j}\sin \xi_i \right)^{2}\\
    \le{}& \msB_1 A_{11}^{2} + \msB_2 A_{22}^{2} + (\msB_1+\msB_2) \sum_{j=3}^{n} (A_{1j}^{2}+A_{2j}^{2}).
\end{align*}
\end{proof}

\begin{lemma}\label{lem:quadratic_form}
    For $\tilde{s}$ defined by
    \[
        \tilde{s}=\frac{(n-1)(\msB_1+\msB_2)+\sqrt{(n-1)^{2}(\msB_1+\msB_2)^{2}-4n(n-2)\msB_1 \msB_2}}{2n},
    \]
    we have
    \[
        \msB_1A_{11}^{2}+\msB_2 A_{22}^{2} \le \tilde{s}\left( A_{11}^{2}+A_{22}^{2}+\frac{(A_{11}+A_{22})^{2}}{n-2} \right),
    \]
    where $\frac{(A_{11}+A_{22})^{2}}{n-2}$ is understood as $0$ when $n=2$.
\end{lemma}

\begin{proof}
For $n=2$ we have $\tilde s=\frac{1}{2}(\msB_1+\msB_2)$.
By the minimality condition we have $A_{11}=-A_{22}$, it is then easy to check that the required inequality holds.

For $n\geq3$, the inequality is equivalent to showing that the matrix
\[
    \mathscr{M}
    \coloneqq\begin{pmatrix}
        \frac{n-1}{n-2}\tilde{s}-\msB_1 & \frac{1}{n-2}\tilde{s}\\
        \frac{1}{n-2}\tilde{s} & \frac{n-1}{n-2}\tilde{s} - \msB_2
    \end{pmatrix}
\]
is positive semi-definite.
For $\tilde{s}$ defined above, we have
\[
    {\rm tr}(\mathscr{M})=\frac{2(n-1)}{n-2}\tilde{s} - (\msB_1+\msB_2)\ge \frac{(n-1)^{2}}{n(n-2)}(\msB_1+\msB_2) - (\msB_1+\msB_2)\ge 0,
\]
and $\tilde{s}$ is a root of $\mathrm{det}(\mathscr{M})=0$.
Hence $\mathscr{M}$ is positive semi-definite, which completes the proof.
\end{proof}

\begin{lemma}\label{lem:B1B2_bounds}
    We have
    \[
        \msB_1 +\msB_2=1+(1-k \sin^{2}\theta)w,\quad \msB_1 \msB_2 \ge (1-k \sin^{2}\theta)w,
    \]
    where
    \[
        w=\frac{1-\nu_1^{2}-\nu_{n+1}^{2}}{g_{\theta,k}^{2}}.
    \]
\end{lemma}

\begin{proof}
We begin by noting that $e_i^\top=e_i-\nu_i\nu$, and hence
\eq{\label{eq:tangential-length-square}
\abs{e_i^\top}^2
=1-\nu_i^2,\quad
\langle e_1^\top,e_{n+1}^\top\rangle
=-\nu_1\nu_{n+1}.
}

We first compute $|\vec{a}_3|^2$ explicitly:
\begin{align*}
    |\vec{a}_3|^{2}={}& \frac{\Abs{(\nu_1+k(\cos \theta - \nu_1))e_1^\top+\nu_{n+1} e_{n+1}^\top}^{2}}{g_{\theta,k}^{2}}\\
    ={}& \frac{(\nu_1+k(\cos \theta -\nu_1))^{2}(1-\nu_1^{2})+\nu_{n+1}^{2}(1-\nu_{n+1}^{2})-2(\nu_1+k(\cos \theta -\nu_1))\nu_1\nu_{n+1}^{2}}{g_{\theta,k}^{2}}\\
    ={}& \frac{((1-k)\nu_1+k\cos \theta)^{2}(1-\nu_1^{2})+\nu_{n+1}^{2}(1-\nu_{n+1}^{2}-2\nu_1^{2}-2k(\cos \theta-\nu_1)\nu_1)}{g_{\theta,k}^{2}}\\
    ={}& \frac{\left[ (1-k)^{2}\nu_1^{2}+k^{2} \cos^{2}\theta +2k(1-k)\cos \theta \nu_1 \right](1-\nu_1^{2}) }{g_{\theta,k}^{2}}+\nu_{n+1}^{2}\\
    & - \frac{\nu_{n+1}^{2}\left(\nu_1^{2}+2k(\cos \theta -\nu_1)\nu_1+k(\cos \theta -\nu_1)^{2}\right)}{g_{\theta,k}^{2}}\\
    ={}& \nu_{n+1}^{2}+\frac{\left[ (1-k)^{2}\nu_1^{2}+k^{2}\cos^{2}\theta +2k(1-k)\cos \theta \nu_1 \right] (1-\nu_1^{2})}{g_{\theta,k}^{2}}\\
    & - \frac{\nu_{n+1}^{2}(1-k+k\cos^{2}\theta)-(1-k)\nu_{n+1}^{2}(1-\nu_1^{2})}{g_{\theta,k}^{2}}.
\end{align*}
Collecting the coefficients of $(1-\nu_1^2)$, we get
\begin{align*}
    & (1-k)^{2}\nu_1^{2}+k^{2}\cos^{2}\theta + 2k(1-k)\cos \theta \nu_1 + (1-k)\nu_{n+1}^{2}\\
    ={}& -(1-k)g_{\theta,k}^{2} + (1-k)+(1-k)k\cos^{2}\theta + k^{2} \cos^{2} \theta\\
    ={}& -(1-k)g_{\theta,k}^{2} + 1-k +k \cos^{2}\theta.
\end{align*}
Thus we find
\begin{align*}
    |\vec{a}_3|^{2}={}& \nu_{n+1}^{2} - (1-k)(1-\nu_1^{2})+\frac{(1-k+k\cos^{2}\theta)(1-\nu_1^{2}-\nu_{n+1}^{2})}{g_{\theta,k}^{2}}\\
    ={}& \nu_{n+1}^{2} - (1-k)(1-\nu_1^{2}) + (1-k \sin^{2}\theta) w.
\end{align*}

For $\vec{a}_1, \vec{a}_2$,
it is easy to see
\[
    |\vec{a}_1|^{2}= (1-k)(1-\nu_1^{2}),\quad |\vec{a}_2|^{2}=1-\nu_{n+1}^{2}.
\]
Summing over all three vectors, we obtain
\[
    \msB_1+\msB_2=\sum_{i=1}^{3}|\vec{a}_i|^{2}=1+(1-k \sin^{2}\theta)w.
\]

For the product, consider the symmetric endomorphism
$\mathcal Q\coloneqq\sum_{i=1}^3\vec a_i\otimes\vec a_i$,
where $(\vec a_i\otimes\vec a_i)(z)=\langle\vec a_i,z\rangle\vec a_i$.
With respect to the orthonormal basis $\{\tau_1,\tau_2\}$, its matrix is
\eq{
\mathcal Q=
\begin{pmatrix}
\msB_1&d\\
d&\msB_2
\end{pmatrix},
\qquad
d=\sum_{i=1}^3\abs{\vec a_i}^2\cos\xi_i\sin\xi_i.
}
Consequently,
$\msB_1\msB_2=\det\mathcal Q+d^2\geq\det\mathcal Q$.
Writing $\mcQ=\mcA\mcA^T$ where $\mcA=(\vec a_1,\vec a_2,\vec a_3)$, by the Cauchy--Binet formula
\eq{
\det\mathcal Q
={}&\sum_{i<j}\abs{\vec a_i\wedge\vec a_j}^2.
}
Recall that
$\abs{\vec a\wedge\vec b}^2
=\abs{\vec a}^2\abs{\vec b}^2-\langle\vec a,\vec b\rangle^2$, and hence
\eq{
\abs{e_1^\top\wedge e_{n+1}^\top}^2
=(1-\nu_1^2)(1-\nu_{n+1}^2)-\nu_1^2\nu_{n+1}^2
=1-\nu_1^2-\nu_{n+1}^2.
}
We compute the three terms separately:
\eq{
\abs{\vec a_1\wedge\vec a_2}^2
={}&(1-k)\abs{e_1^\top\wedge e_{n+1}^\top}^2
=(1-k)(1-\nu_1^2-\nu_{n+1}^2),\\
\abs{\vec a_1\wedge\vec a_3}^2
={}&\frac{\abs{\sqrt{1-k}e_1^\top
\wedge\nu_{n+1}e_{n+1}^\top}^2}{g_{\theta,k}^2}
=\frac{(1-k)\nu_{n+1}^2
(1-\nu_1^2-\nu_{n+1}^2)}{g_{\theta,k}^2},\\
\abs{\vec a_2\wedge\vec a_3}^2
={}&\frac{\abs{e_{n+1}^\top\wedge
(\nu_1+k(\cos\theta-\nu_1))e_1^\top}^2}{g_{\theta,k}^2}\\
={}&\frac{(\nu_1+k(\cos\theta-\nu_1))^2
(1-\nu_1^2-\nu_{n+1}^2)}{g_{\theta,k}^2}.
}
Summing these identities gives
\eq{
\det\mathcal Q
={}&\left((1-k)g_{\theta,k}^2+(1-k)\nu_{n+1}^2
+(\nu_1+k(\cos\theta-\nu_1))^2\right)w\\
={}&\left(1-k-(1-k)\nu_1^2
+(1-k)k(\cos\theta-\nu_1)^2+\nu_1^2
+2k\nu_1(\cos\theta-\nu_1)
+k^2(\cos\theta-\nu_1)^2\right)w\\
={}&(1-k+k\cos^2\theta)w
=(1-k\sin^2\theta)w.
}
The required lower bound for $\msB_1\msB_2$ follows.
\end{proof}
\begin{proof}
[Proof of Proposition \ref{prop:boundS}]
We define the function
\eq{
    f(t)=\frac{(n-1)(1+t)+\sqrt{(n-1)^{2}(1+t)^{2}-4n(n-2)t}}{2n},\quad\forall t\in[0,\infty).
}
A direct computation shows when $n\geq2$, $f'(t)>0$ for all $t\in \mathbb{R}_+$, i.e., $f(t)$ is an increasing function in $t$.
Note that $w\in [0,1]$, so by Lemma \ref{lem:quadratic_form} and Lemma \ref{lem:B1B2_bounds}, we have
\eq{\label{ineq:tildes-s}
    \tilde{s}\leq f\left((1-k\sin^2\theta)w\right)
    \le f(1-k \sin^{2}\theta)=s(k,n,\theta).
}
By Lemma \ref{lem:gradient-nu-estimates}, Lemma \ref{lem:quadratic_form}, and \eqref{ineq:tildes-s} we deduce
\eq{
    \sum_{i=1}^{3}|\nabla \nu \cdot \vec{a}_i|^{2}\le{}& s(k,n,\theta)\left( A_{11}^{2}+A_{22}^{2}+\frac{(A_{11}+A_{22})^{2}}{n-2} \right)+ (\msB_1+\msB_2) \sum_{j=3}^{n}(A_{1j}^{2}+A_{2j}^{2})\\
    \le{}& s(k,n,\theta)\left( A_{11}^{2}+A_{22}^{2}+ \frac{(\sum_{j=3}^{n}A_{jj})^{2}}{n-2} \right)+ 2s(k,n,\theta) \sum_{j=3}^{n}(A_{1j}^{2}+A_{2j}^{2})\\
    \le{}& s(k,n,\theta)\left( \sum_{i=1}^{n}A_{ii}^{2} + \sum_{j=3}^{n}(A_{1j}^{2}+A_{2j}^{2}+A_{j1}^{2}+A_{j2}^{2}) \right)\\
    \le{}& s(k,n,\theta)|A|^{2}.
}
Here we have also used the Cauchy-Schwarz inequality, the minimality condition $\sum_{i=1}^{n}A_{ii}=0$, and the fact that
\[
2s(k,n,\theta)
\ge \frac{(n-1)(2-k\sin^2\theta)+\sqrt{4(1-k\sin^2\theta)+k^2\sin^4\theta}}{n}=2-k\sin^2\theta
\ge \msB_1+\msB_2.
\]
\end{proof}
To prove Theorem \ref{thm:differential-capillary-Schoen}, it remains to
bound $s(k,n,\theta)$ from above, which determines the range of $k$.

\begin{lemma}\label{lem:chooseK}
    For $\theta\in (0,\pi)$, and for $0<k\leq1$ when $n=2$; $0<k\le \frac{1}{n-2}$ when $n\ge 3$, we have
    \[
        1+k\cos^{2}\theta - k |\cos \theta |-s(k,n,\theta) >0.
    \]
\end{lemma}

\begin{proof}
By direct computation, it is equivalent to showing
\[
    (2-k\sin^{2}\theta)+nk(1-|\cos \theta|)^{2} > \sqrt{4(1-k\sin^{2}\theta)+(n-1)^{2}k^{2}\sin^{4}\theta}.
\]
Note that
\[
    (1-|\cos \theta|)^{2}>\frac{\sin^4\theta}{2(2-\sin^{2}\theta)}.
\]
So it is sufficient to show
\[
    \left( 2-k\sin^{2}\theta+\frac{nk\sin^4\theta}{2(2-\sin^{2}\theta)} \right)^{2}\ge 4(1-k\sin^{2}\theta)+(n-1)^{2}k^{2}\sin^{4}\theta.
\]
After simplification, this is equivalent to
\[
    \frac{nk\sin^4\theta(2-k\sin^{2}\theta)}{2-\sin^{2}\theta}+k^{2}\sin^{4}\theta + \frac{n^{2}k^{2}\sin^8\theta}{4(2-\sin^{2}\theta)^{2}}\ge (n-1)^{2}k^{2}\sin^{4}\theta.
\]
Note that the range of $k$ yields $nk+k^{2}\ge (n-1)^{2}k^{2}$.
The above inequality then holds, since
\[
    \frac{nk(2-k\sin^{2}\theta)}{2-\sin^{2}\theta}+k^{2}\ge nk+k^{2}\ge (n-1)^{2}k^{2}.
\]
This completes the proof.
\end{proof}

With these ingredients, we can now prove the theorem.

\begin{proof}[Proof of Theorem \ref{thm:differential-capillary-Schoen}]
Put $k_n=1$ if $n=2$, $=\frac1{n-2}$ if $n\geq3$.
By \eqref{eq:lemDeltaG} and Proposition
\ref{prop:boundS}, we find
\eq{
g_\theta(\Delta+\abs{A}^2)g_\theta
\geq\left(1+k_n\cos^2\theta-k_n\cos\theta\,\nu_1-s(k_n,n,\theta)\right)\abs{A}^2
\geq\delta_{n,\theta}\abs{A}^2\text{ on }\{g_\theta>0\},
}
while by Lemma \ref{lem:chooseK},
\eq{
\delta_{n,\theta}
=1+k_n\cos^2\theta-k_n\abs{\cos\theta}-s(k_n,n,\theta)
>0
}
This completes the proof.
\end{proof}

\section{Capillary first variation formula and its consequences}\label{sec:capillaryFirstVariation}

Let $n\geq2$, for any $\overline V\in\mathscr{V}(\theta,\Lambda)$, in view of Definition \ref{defn:varifold-class} (iii), we can always assume WLOG that $(V,W)$ is a $\mbF_\theta$-stationary pair in $B_2(0)$ for $\theta\in[\frac{\pi}2,\pi)$.
By \cite[Lemma 38.4]{Simon83},
\eq{\label{eq:Euc-Riem-capillary-1st-variation}
&{\int{\rm div}_P(\psi)\rd V(X,P)-\cos\theta\int{\rm div}_{\p\mbH^{n+1}}(\psi)\rd\norm{W}(X)}=0
}
for any $\psi\in C^1(\mbR^{n+1};\mbR^{n+1})$ tangential to $\p\mbH^{n+1}$, compactly supported in $B_2(0)$.
Hence the rectifiable $n$-varifold $\overline{V}=V-\cos\theta W$ is stationary with free boundary in $B_2(0)$, in the sense of Definition \ref{Defn:stationary-free-bdry-vfld}.

\subsection{Monotonicity formula}

We record here Kagaya-Tonegawa's monotonicity formula \cite{KT17}.
For any $X=(x_1,x_2,\cdots,x_{n+1})\in\mbH^{n+1}$, we denote by $\tilde X=(-x_1,x_2,\cdots,x_{n+1})$ the reflection point across $\p\mbH^{n+1}$.
It follows immediately for any $Z\in\mbR^{n+1}$,
\eq{\label{eq:relation-reflection-half-space}
\abs{X-\tilde Z}
=\abs{\tilde X-Z}.
}
\begin{lemma}\label{Lem:Monotonicity-formula-doubling}
Let $n\geq2,\theta\in[\frac{\pi}{2},\pi)$.
Let $(V,W)$ be a $\mbF_\theta$-stationary pair in $B_2(0)$, in the sense of Definition \ref{defn:stationary-pair}.
Then for every $0<\sigma<\rho<\frac{1}{2}$ and any $X_0\in B_1(0)$,
\eq{
\frac{1}{\sigma^n}\left(\norm{\overline{V}}(B_\sigma(X_0))+\norm{\overline{V}}(B_\sigma(\tilde X_0))\right)
=\frac{1}{\rho^n}\left(\norm{\overline{V}}(B_\rho(X_0))+\norm{\overline{V}}(B_\rho(\tilde X_0))\right)\\
-\left(\int_{B_\rho(X_0)\setminus{B_\sigma(X_0)}}\frac{\abs{(Y-X_0)^\perp}^2}{\abs{Y-X_0}^{n+2}}\rd\norm{\overline{V}}(Y)+\int_{B_\rho(\tilde X_0)\setminus{B_\sigma(\tilde X_0)}}\frac{\abs{(Y-\tilde X_0)^\perp}^2}{\abs{Y-\tilde X_0}^{n+2}}\rd\norm{\overline{V}}(Y)\right).
}

\end{lemma}
\begin{proof}
Let $\phi$ be any smooth non-increasing function which $\equiv1$ on $(-\infty,\frac{1}{2}]$ and $\equiv0$ on $[1,\infty)$.
For any fixed $X_0\in B_1(0)$,
define the vector fields $\psi_{X_0}(Y)$ and $\psi(Y)$ as
\eq{
\psi_{X_0}(Y)
=\phi\left(\frac{\abs{Y-X_0}}r\right)\left(Y-X_0\right),\quad
\psi(Y)
=\psi_{X_0}(Y)+\psi_{\tilde X_0}(Y).
}
By \eqref{eq:relation-reflection-half-space} one verifies that $\psi$ is tangential to $\p\mbH^{n+1}$.

Testing the first variation formula \eqref{eq:Euc-Riem-capillary-1st-variation} with $\psi$, we get
\eq{
&\int n\phi\left(\frac{\abs{Y-X_0}}r\right)+\phi'\left(\frac{\abs{Y-X_0}}r\right)\frac{\abs{Y-X_0}}r\left(1-\frac{\abs{(Y-X_0)^\perp}^2}{\abs{Y-X_0}^2}\right)\rd\norm{\overline{V}}(Y)\\
+&\int n\phi\left(\frac{\abs{Y-\tilde X_0}}r\right)+\phi'\left(\frac{\abs{Y-\tilde X_0}}r\right)\frac{\abs{Y-\tilde X_0}}r\left(1-\frac{\abs{(Y-\tilde X_0)^\perp}^2}{\abs{Y-\tilde X_0}^2}\right)\rd\norm{\overline{V}}(Y)
=0.
}
Put
\eq{
I(r)
\coloneqq&\int\phi\left(\frac{\abs{Y-X_0}}r\right)\rd\norm{\overline{V}}(Y)+\int\phi\left(\frac{\abs{Y-\tilde X_0}}r\right)\rd\norm{\overline{V}}(Y),\\
J(r)
\coloneqq&\int\phi\left(\frac{\abs{Y-X_0}}r\right)\frac{\abs{(Y-X_0)^\perp}^2}{\abs{Y-X_0}^2}\rd\norm{\overline{V}}(Y)+\int\phi\left(\frac{\abs{Y-\tilde X_0}}r\right)\frac{\abs{(Y-\tilde X_0)^\perp}^2}{\abs{Y-\tilde X_0}^2}\rd\norm{\overline{V}}(Y).
}
We deduce the differential equality
\eq{
\frac{\rd}{\rd r}\left(r^{-n}I(r)\right)
= r^{-n}J'(r).
}
Integrating from $0<\sigma<\rho<\frac{1}{2}$ and taking $\phi\ra\chi_{(-\infty,1)}$ gives the required formula.
\end{proof}

It is then easy to conclude the following.

\begin{corollary}\label{Cor:density-upper-semi-continuity}
Under the assumptions of Lemma \ref{Lem:Monotonicity-formula-doubling},
\begin{enumerate}
    \item the tilde-density, defined as
\eq{
\tilde\Theta^n(\norm{\overline{V}},X)\coloneqq\lim_{\rho\ra0}\frac{\norm{\overline{V}}\left(B_\rho(X)\right)+\norm{\overline{V}}\left(B_\rho(\tilde X)\right)}{\om_n\rho^n},\quad\forall X\in\mbR^{n+1},
}
exists for every $X_0\in B_1(0)$.
Moreover, the function $X\mapsto\tilde\Theta^n(\norm{\overline{V}},X)$ is upper semi-continuous on $B_1(0)$.
    \item ${\rm VarTan}(\overline{V},X)\neq\emptyset$ for every
    $X\in\spt\norm{\overline V}\cap B_1(0)$.
    Every $\mfC\in{\rm VarTan}(\overline{V},X)$ is an $n$-rectifiable cone
    stationary with free boundary in $\mbH^{n+1}$.
\end{enumerate}

\end{corollary}

\begin{lemma}\label{Lem:Ahlfors-regularity}
Let $n\geq2, \theta\in(0,\pi)$.
Let $\overline{V}\in \mathscr{V}(\theta,\Lambda)$, and
denote by $M,V,W$ the corresponding hypersurface and varifolds as in Definition \ref{defn:varifold-class}.
Then there exists a positive constant $C=C(n,\theta,\Lambda)$, with the following property (Ahlfors-regularity):
for any $\sigma\in(0,\frac{1}{2})$ and $X\in\overline{M}%
\cap B_1(0)$,
\eq{
C^{-1}
\leq\frac{\mcH^n(\overline{M}%
\cap B_\sigma(X))}{\sigma^n}
\leq C.
}
\end{lemma}

\begin{proof}
As discussed in the beginning of this section, we assume WLOG that $(V,W)$ is $\mbF_\theta$-stationary for $\theta\in[\frac\pi2,\pi)$.

We first show the upper bound.
By Lemma \ref{Lem:Monotonicity-formula-doubling}, for any $0<\sigma<\rho<\frac{1}{2}$,
\eq{\label{ineq:Monotonicity-tilde-area}
&\frac{\norm{V}\left(B_\sigma(X)\right)-\cos\theta\norm{W}\left(B_\sigma(X)\right)+\norm{V}\left(B_\sigma(\tilde X)\right)-\cos\theta\norm{W}\left(B_\sigma(\tilde X)\right)}{\sigma^n}\\
\leq&\frac{\norm{V}\left(B_\rho(X)\right)-\cos\theta\norm{W}\left(B_\rho(X)\right)+\norm{V}\left(B_\rho(\tilde X)\right)-\cos\theta\norm{W}\left(B_\rho(\tilde X)\right)}{\rho^n}.
}
Note that $\norm{V}=\mcH^n\llcorner\overline{M}$, and that
$\norm{V}(B_\sigma(\tilde X))\leq\norm{V}(B_\sigma(X))$ because $\overline{M}$ is supported in $\overline{\mbH^{n+1}}$.
The desired upper bound then follows immediately, thanks to $\theta\in[\frac{\pi}{2},\pi)$.

Then we derive the lower bound.
For any $X\in\overline{M}%
\cap B_1(0)$ and any $\sigma\in(0,\frac{1}{2})$, we put
\eq{
Q(X,\sigma)
\coloneqq\mcH^n(\overline{M}%
\cap B_\sigma(X)).
}
For simplicity we omit the argument $X$ and write $Q(\sigma)$.
Using the Michael-Simon-type inequality (Lemma \ref{Lem:M-S-ineq}), we can argue as in \cite[Prop 4.7]{JZ24} to show the differential inequality
\eq{
Q(r)^\frac{n-1}{n}
\leq C(n,\theta)Q'(r),\quad\forall r\in(0,\sigma),
}
which gives
\eq{
C(n,\theta)
\leq\frac{Q'(r)}{Q(r)^{1-\frac{1}{n}}}.
}
Since $Q(0)=0$,
integrating the above differential inequality over $(0,\sigma)$, we obtain
\eq{
Q(\sigma)
\geq C(n,\theta)\sigma^n.
}
This gives the required lower bound.

\end{proof}

\begin{lemma}\label{Lem:Ahlfors-regularity-capillary-vfld}

Let $n\geq2, \theta\in[\frac{\pi}{2},\pi)$, $\Lambda\in[1,\infty)$.
Let $M$ be a capillary minimal hypersurface in the sense of Definition \ref{defn:stable-capillary-minimal-hypersurface}, put $V\coloneqq\abs{M}$.
Assume $\mcH^{n-2}\left({\rm Sing}_\theta V\right)=0$, and there exists $W$ such that $(V,W)$ is $\mbF_\theta$-stationary in $\mbC^\theta_2$ with $\left(\norm{V}-\cos\theta\norm{W}\right)(\mbC^\theta_2)\leq\Lambda$.
Then there exists a positive constant $C=C(n,\theta,\Lambda)$, with the following property (Ahlfors-regularity):
for any $\sigma\in(0,\frac{1}{2})$ and $X\in\spt\norm{\overline{V}}\cap B_1(0)$,
\eq{
C^{-1}
\leq\frac{\norm{\overline{V}}(B_\sigma(X))}{\sigma^n}
\leq C.
}
\end{lemma}
\begin{proof}
The upper bound follows by the same argument as in the proof of Lemma \ref{Lem:Ahlfors-regularity}.
The lower bound follows from Lemma \ref{Lem:Monotonicity-formula-doubling} and Corollary \ref{Cor:density-upper-semi-continuity}(1), together with the fact that $\tilde\Theta^n(\norm{\overline{V}},X)\geq1-\cos\theta$ for $X\in M%
$ (recalling Definition \ref{Defn:theta-regular-points-manifold} \ref{it:defThetaBoundary} \ref{it:defNonemptyIntersection}, and the assumption $\mcH^{n-2}\left({\rm Sing}_\theta V\right)=0$.
\end{proof}

\subsection{Tilt-excess controls}

\begin{proposition}\label{Prop:tilt-excess-control-x_1}

Let $n\geq2, \theta\in[\frac{\pi}{2},\pi)$, $\Lambda\in[1,\infty)$.
Let $M$ be a capillary minimal hypersurface in the sense of Definition \ref{defn:stable-capillary-minimal-hypersurface}, put $V\coloneqq\abs{M}$.
Assume $\mcH^{n-2}\left({\rm Sing}_\theta V\right)=0$, there exists $W$ such that $(V,W)$ is $\mbF_\theta$-stationary in $\mbC^\theta_2$ with $\left(\norm{V}-\cos\theta\norm{W}\right)(\mbC^\theta_2)\leq\Lambda$, and for some $\ep\in[0,1)$,
\eq{\label{condi:Lem-tilt-height-excess-small-height-weakly-stable}
2^{-1}\dist_\mcH\left(\overline{M}%
\cap\left(\mbR\times B^{n}_2(0)\right),\{0\}\times B^{n}_2(0)\right)
\leq\ep.
}
Then
\eq{
\int_{M%
\cap B_1(0)}\left(1-\langle\nu,e_1\rangle^2\right)\rd\mcH^n
\leq 8\epsilon^2 \mathcal{H}^n\left(M\cap B_2(0)\right).
}
\end{proposition}

\begin{proof}
Let $\zeta\in C^1_c(B_{2}(0))$ be a cut-off function which is identically $1$ in $B_1(0)$ with $\abs{D\zeta}\leq2$.
Testing the first variation formula \eqref{eq:Euc-Riem-capillary-1st-variation} with $\psi(X)=x_1\zeta^2(X)e_1$ (which vanishes on $\p\mbH^{n+1}$ and hence admissible), we get
\eq{
\int\left(1-\langle\nu,e_1\rangle^2\right)\zeta^2\rd\norm{V}
\leq \int2\abs{x_1}\abs{\zeta}\abs{\langle\na\zeta,e_1\rangle}\rd\norm{V},
}
by Cauchy-Schwarz inequality we find
\eq{\label{ineq:height-tilt-Cauchy-Schwarz}
2\abs{x_1}\abs{\zeta}\abs{\langle\na\zeta,e_1\rangle}
=2\abs{x_1}\abs{\zeta}\abs{\langle\na\zeta,e^\top_1\rangle}
\leq 2\abs{x_1}^2\abs{\na\zeta}^2+\frac{1}{2}\left(1-\langle\nu,e_1\rangle^2\right)\zeta^2.
}
Using the H\"older inequality in conjunction with Lemma \ref{Lem:Ahlfors-regularity}, taking also into account that $V$ is integral, we obtain the required estimate.

\end{proof}

\section{Sheeting theorems}
\label{sec:SheetingTheorem}

\begin{lemma}\label{Lem:normalized-Sobolev}
Let $n\geq2$, $\theta\in(0,\pi)$, and $\Lambda\in[1,\infty)$.
Let $\overline V\in\mathscr V(\theta,\Lambda)$, and denote by $M,V,W$ the
corresponding objects in Definition \ref{defn:varifold-class}.
There is a constant $C=C(n,\theta)$ such that every Lipschitz function $f$
compactly supported in $\mbC_2^\theta$ satisfies the following inequalities.
If $n\geq3$, then
\eq{\label{ineq:normalized-Sobolev}
\left(\int_M\abs{f}^{\frac{2n}{n-2}}\rd\mcH^n\right)^{\frac{n-2}{n}}
\leq C\int_M\abs{\na f}^2\rd\mcH^n.
}
If $n=2$, then
\eq{\label{ineq:normalized-Sobolev-n=2}
\left(\int_M\abs{f}^4\rd\mcH^2\right)^{\frac12}
\leq C\left(\int_M f^2\rd\mcH^2\right)^{\frac12}
\left(\int_M\abs{\na f}^2\rd\mcH^2\right)^{\frac12}.
}
\end{lemma}
\begin{proof}
As discussed in the beginning of Section \ref{sec:capillaryFirstVariation}, we assume WLOG that $(V,W)$ is $\mbF_\theta$-stationary for $\theta\in[\frac\pi2,\pi)$.

For $n\geq3$, apply Lemma \ref{Lem:M-S-ineq} to
$\abs{f}^{2(n-1)/(n-2)}$ and use H\"older's inequality.
For $n=2$, apply the same lemma to $f^2$ and then use Cauchy--Schwarz:
\[
\left(\int_M\abs{f}^4\rd\mcH^2\right)^{\frac12}
\leq C\int_M\abs{f}\abs{\na f}\rd\mcH^2
\leq C\left(\int_M f^2\rd\mcH^2\right)^{\frac12}
\left(\int_M\abs{\na f}^2\rd\mcH^2\right)^{\frac12}.
\]
The approximation across $\operatorname{Sing}_\theta V$ is already included in Lemma \ref{Lem:M-S-ineq}.
\end{proof}

We next combine Theorem \ref{thm:differential-capillary-Schoen} with the stability inequality to prove a Caccioppoli inequality in the spirit of \cite{Bellettini25}. For the ansatz $h_\theta$ defined as in \eqref{defn:h_theta}, we note that the idea of using $h_\theta$ in the capillary stability inequality can be at least dated back to \cite{Hong2021capillary}.
\begin{lemma}
The differential equations \eqref{eq:h_theta} hold.
\end{lemma}
\begin{proof}
The equations follow directly from \eqref{eq:lemLaplacian}, \eqref{eq:lemBoundaryDeri}, and the capillary boundary condition \eqref{eq:nu_and_eta}.
\end{proof}

\begin{lemma}\label{lem:capillary-truncated-caccioppoli}
Under the assumptions of Lemma \ref{Lem:normalized-Sobolev}. There are constants $L_0=L_0(n,\theta)>0$ and $C=C(n,\theta)$ such that, whenever
$0<\ell<L\leq L_0$ and $\varphi\in{\rm Lip}_c(\mbC_2^\theta)$, one has
\eq{\label{eq:two-level-caccioppoli}
\int_M\abs{\na(\varphi (g_\theta-L h_\theta)^+)}^2\rd\mcH^n
\leq\frac{C}{1-\ell/L}
\int_M(g_\theta-\ell h_\theta)^{+^2}\abs{\na\varphi}^2\rd\mcH^n.
}
\end{lemma}
\begin{proof}
By Theorem \ref{thm:differential-capillary-Schoen},
\[
g_\theta(\Delta+\abs{A}^2)g_\theta
\geq\delta_{n,\theta}\abs{A}^2
\quad\text{on }\{g_\theta>0\},
\]
and by \eqref{eq:lemBoundaryG},
\eq{\label{eq:g_theta-bdry}
\partial_\eta g_\theta
+\cot\theta\,A(\eta,\eta)g_\theta=0
\quad\text{on }\partial M\cap\{g_\theta>0\},
}
In view of \eqref{eq:h_theta}, we see that $g_\theta$ and $h_\theta$ satisfy the same boundary condition.

To proceed, put $k_n=1$ if $n=2$, $=\frac1{n-2}$ if $n\geq3$. Set $G_{n,\theta}\coloneqq\max_{\xi\in\mathbb S^n}\sqrt{1-\xi_1^2-\xi_{n+1}^2+k_n(\cos\theta-\xi_1)^2}$ so that $g_\theta\leq G_{n,\theta}$ on $M$, and choose $L_0=\min\left\{1,\frac{\delta_{n,\theta}}{2G_{n,\theta}}\right\}$.
On $\{g_\theta-\ell h_\theta>0\}$, the function $g_\theta$ is bounded away from zero and by \eqref{eq:h_theta},
\eq{\label{eq:g_theta-l-h_theta}
(\Delta+\abs{A}^2)(g_\theta-\ell h_\theta)
=(\Delta+\abs{A}^2)g_\theta-\ell\abs{A}^2.
}
Using $(g_\theta-\ell h_\theta)^+\varphi$ in the stability inequality \eqref{ineq:SS81-(1.17)-smooth-case} then integrating by parts,
and using \eqref{eq:g_theta-bdry}, \eqref{eq:h_theta}, we obtain 
\eq{\label{eq:truncation-before-lower-bound}
\int_{\{g_\theta-\ell h_\theta>0\}}(g_\theta-\ell h_\theta)^+
(\Delta+\abs{A}^2)(g_\theta-\ell h_\theta)
\varphi^2\rd\mcH^n
\leq\int_M(g_\theta-\ell h_\theta)^{+^2}\abs{\na\varphi}^2\rd\mcH^n.
}
By \eqref{eq:g_theta-l-h_theta}, Theorem \ref{thm:differential-capillary-Schoen}, and the choice of $L_0$, we have
\eq{
(g_\theta-\ell h_\theta)^+(\Delta+\abs{A}^2)(g_\theta-\ell h_\theta)
\geq\frac{(g_\theta-\ell h_\theta)^+}{g_\theta}
\left(\delta_{n,\theta}-\ell g_\theta\right)\abs{A}^2
\geq\frac{\delta_{n,\theta}}2
\frac{(g_\theta-\ell h_\theta)^+}{g_\theta}\abs{A}^2.
}
Consequently,
\eq{\label{eq:weighted-A-caccioppoli}
\frac{\delta_{n,\theta}}2
\int_{\{g_\theta-\ell h_\theta>0\}}\frac{(g_\theta-\ell h_\theta)^+}{g_\theta}\abs{A}^2\varphi^2\rd\mcH^n
\leq\int_M(g_\theta-\ell h_\theta)^{+^2}\abs{\na\varphi}^2\rd\mcH^n.
}
Finally, on $\{g_\theta-Lh_\theta>0\}$ we have
\eq{
\frac{(g_\theta-\ell h_\theta)^+}{g_\theta}
=1-\frac{\ell h_\theta}{g_\theta}
\geq1-\frac\ell L,
}
while using the gradient estimate of $g_\theta$ (see \eqref{ineq:nabla-g_theta-estimate}), and also $\abs{\na h_\theta}\leq\abs{\na\nu_1}\leq\abs{A}$, we have
\eq{\label{ineq:nabla-g_theta-Lh_theta}
\abs{\na (g_\theta-L h_\theta)^+}\leq(1+L)\abs{A}.
}
Since $(g_\theta-L h_\theta)\leq (g_\theta-\ell h_\theta)$, \eqref{eq:weighted-A-caccioppoli} can be further written as
\eq{
\frac{1}{(1+L)^2}\int_{\{g_\theta-Lh_\theta>0\}}\abs{\na (g_\theta-L h_\theta)^+}^2\varphi^2\rd\mcH^n
\leq&\int_{\{g_\theta-Lh_\theta>0\}}\abs{A}^2\varphi^2\rd\mcH^n\\
\leq&\frac{C(n,\theta)}{1-\ell/L}
\int_M(g_\theta-\ell h_\theta)^{+^2}\abs{\na\varphi}^2\rd\mcH^n.
}
Expanding $\na(\varphi(g_\theta-L h_\theta)^+)$ and using again $(g_\theta-L h_\theta)\leq (g_\theta-\ell h_\theta)$, yields \eqref{eq:two-level-caccioppoli}.

\end{proof}

\begin{lemma}\label{lem:capillary-De-Giorgi}
Under the assumptions of Lemma \ref{Lem:normalized-Sobolev}, there are constants $\varepsilon=\varepsilon(n,\theta)>0$ and $C=C(n,\theta)>0$ with the following property.
If $0<R\leq\frac{1}{2(1+\abs{\cot\theta})}$ and
\eq{
\mathcal E_R
=R^{-n}\int_{M\cap\mbC_R^\theta}g_\theta^2\rd\mcH^n
\leq\varepsilon,
}
then
\eq{\label{eq:De-Giorgi-pointwise-tilt}
\sup_{M\cap\mbC_{R/2}^\theta}g_\theta
\leq C\mathcal E_R^{\frac12}.
}
\end{lemma}
\begin{proof}
We follow the iteration of Bellettini \cite[(17)--(23)]{Bellettini25},
using Lemma \ref{lem:capillary-truncated-caccioppoli} in place of the interior
Caccioppoli inequality.
Put for simplicity $M_r=M\cap\mbC_r^\theta$, $q_{\ell}=(g_\theta-\ell h_\theta)^+$.
Fix $0<L\leq L_0$, and define
\eq{
r_j=R\left(\frac12+2^{-j-1}\right),
\quad
\ell_j=L\left(1-2^{-j-1}\right),
\quad
Y_j=\int_{M_{r_j}}q_{\ell_j}^2\rd\mcH^n.
}
On $M_{r_{j+1}}\cap\{q_{\ell_{j+1}}>0\}$, we have
\eq{
q_{\ell_j}
=g_\theta-\ell_jh_\theta
\geq(\ell_{j+1}-\ell_j)h_\theta
\geq(1-\abs{\cos\theta})L2^{-j-2}.
}
Set
\eq{
\kappa_n=
\begin{cases}
2,&n=2,\\
\dfrac{n}{n-2},&n\geq3.
\end{cases}
}
Consider a cutoff
function $\zeta_j$ supported in $\mbC_{r_j}^\theta$, equal to one on
$\mbC_{r_{j+1}}^\theta$, and satisfying
$\abs{\na\zeta_j}\leq C(n,\theta)R^{-1}2^j$.
Hence
\begin{align*}
Y_{j+1}\le{}&
\int_{M_{r_{j+1}}\cap\{q_{\ell_{j+1}}>0\}}q_{\ell_{j+1}}^2
\frac{q_{\ell_j}^{2\kappa_n-2}}
{(1-\abs{\cos\theta})^{2\kappa_n-2}
L^{2\kappa_n-2}2^{-(2\kappa_n-2)(j+2)}}\rd\mcH^n\\
\le{}&C^{j+1}L^{-2(\kappa_n-1)}
\int_{M_{r_j}\cap\{q_{\ell_j}>0\}}
(\zeta_jq_{\ell_j})^{2\kappa_n}\rd\mcH^n.
\end{align*}
Here and below, $C=C(n,\theta)>1$ may increase from line to line.
If $n\geq3$, inequality \eqref{ineq:normalized-Sobolev} bounds the last
integral by a constant times
\eq{
\left(\int_{M_{r_j}}
\abs{\nabla(\zeta_jq_{\ell_j})}^2\rd\mcH^n\right)^{\kappa_n}.
}
If $n=2$, inequality \eqref{ineq:normalized-Sobolev-n=2} instead bounds it
by a constant times
\eq{
\underbrace{\left(\int_{M_{r_j}}(\zeta_jq_{\ell_j})^2\rd\mcH^2\right)}_{\leq Y_{j-1}}
\left(\int_{M_{r_j}}
\abs{\nabla(\zeta_jq_{\ell_j})}^2\rd\mcH^2\right).
}
Hence in both cases, we deduce the following estimate by Lemma \ref{lem:capillary-truncated-caccioppoli}:
\eq{
Y_{j+1}
\leq C^{j+1}L^{-2(\kappa_n-1)}
R^{-n(\kappa_n-1)}Y_{j-1}^{\kappa_n}.
}
This is equivalent to the following iteration
\eq{\label{eq:De-Giorgi-preliminary-iteration}
y_{j+1}\leq C^{j+1}y_{j-1}^{\kappa_n},
\qquad j\geq1\quad \text{where }y_j:=\frac{Y_j}{R^nL^2}
}
Iterating along the even subsequence gives, for $l\geq1$,
\eq{
y_{2l}
\leq\left(C^{\sum_{m=1}^l \frac{2m}{\kappa_n^{m}}}y_0\right)^{\kappa_n^l}
\leq\left(C^{\frac{2\kappa_n}{(\kappa_n-1)^2}}y_0\right)^{\kappa_n^l}
}
where we have used the fact that $\sum_{m=1}^l \frac{m}{\kappa_n^m}\le \sum_{m=1}^{+\infty} \frac{m}{\kappa_n^m}=\frac{\kappa_n}{(\kappa_n-1)^2}$.
Thus $y_{2l}\to0$ provided $y_0$ is sufficiently small depending only on
$n$ and $\theta$.
By choosing $L=C(n,\theta)\mathcal E_R^{1/2}$ and decreasing $\varepsilon$ if necessary, we can ensure that $y_0$ is sufficiently small and hence $y_{2l}\to0$, which gives the desired estimate \eqref{eq:De-Giorgi-pointwise-tilt}.

\end{proof}

\begin{proof}[Proof of Theorem \ref{thm:sheetingHalf}]
Applying Lemma \ref{lem:capillary-De-Giorgi} with $R=\sigma$.
After decreasing $\ep_0$, we obtain
\eq{\label{eq:sheeting-uniform-tilt}
\sup_{M\cap\mbC_{\sigma/2}^\theta}g_\theta
\leq C(n,\theta)E_\sigma^{\frac12}
\leq\delta_0(n,\theta).
}

We note that the definition of $g_\theta$ and the fact that $\abs\nu=1$ imply (cf. \eqref{ineq:normal-g-estimate-upward}, \eqref{ineq:nu-n+1-lower-bound}, \eqref{ineq:normal-g-estimate-downward}, \eqref{ineq:nu_n+1+sin-theta})
\eq{
\min\left\{\abs{\nu-\nu_\theta},\abs{\nu-\nu_{-\theta}}\right\}
\leq C(n,\theta)g_\theta,
\quad
\abs{\nu_{n+1}}\geq\frac12\sin\theta.
}

Thus the vertical projection $\pi:\overline{\mbH^{n+1}}\longrightarrow\overline{\mbH^n}$, $(x,x_{n+1})\mapsto x$, is a local diffeomorphism on the immersed hypersurface $\iota:M\to\overline{\mbH^{n+1}}$, including at its boundary $\p M$.
The sign of $\nu_{n+1}$ is constant on each connected component and separates the positive phase ($\nu_{n+1}\geq\frac12\sin\theta$) and negative phase ($\nu_{n+1}\leq-\frac12\sin\theta$).

The gradient estimate in Lemma \ref{Lem:Du+cot-e_1<=g_theta} and
\eqref{eq:sheeting-uniform-tilt} show that every local graph in the positive or negative phase satisfies respectively that
\eq{\label{eq:sheeting-gradient-from-tilt}
\abs{Du+\cot\theta\,e_1}
\leq C(n,\theta)g_\theta,
\quad\text{ or }\quad
\abs{Du-\cot\theta\,e_1}
\leq C(n,\theta)g_\theta.
}
Following the properness and continuation argument in \cite[Proof of Theorem 5]{Bellettini25}, also using $\mcH^{n-2}(\operatorname{Sing}_\theta V)=0$, we can thus deduce the graph decomposition
\eq{
\overline M\cap\mbC_{\sigma/8}^\theta
=\mbC_{\sigma/8}^\theta\cap
\left(\left(\bigcup_{j\in Q^+}\operatorname{graph}(u_j^+)\right)
\cup\left(\bigcup_{j\in Q^-}\operatorname{graph}(u_j^-)\right)\right),
}
where $u_j^\pm$ are smooth functions solving the minimal surface equation and satisfying the capillary boundary condition, such that $u^\pm_j\leq u^\pm_{j+1}$ for $j=1,2,\ldots,q^\pm-1$.

It remains to prove the estimates \eqref{esti:C^2-u^pm}.
Set \footnote{$w_j^\pm$ is the so-called {\em slanted graph function}, recently used by De Masi–Edelen–Gasparetto–Li \cite{DEGL25} in their viscosity approach to the Allard-type boundary regularity.}
\eq{
w_j^\pm=u_j^\pm\pm\cot\theta\,x_1,
\quad
x_*=\left(\frac{\sigma}{16},0,\ldots,0\right),
\quad
a_j^\pm=w_j^\pm(x_*).
}
By \eqref{eq:sheeting-uniform-tilt} and \eqref{eq:sheeting-gradient-from-tilt},
\eq{\label{ineq:w_j-gradient-estimate}
\sup_{B_{\sigma/4}^{n+}(0)}\abs{Dw_j^\pm}
\leq C(n,\theta)E_\sigma^{\frac12}.
}
Integrating this estimate along line segments in the half-ball yields
\eq{
\sigma^{-1}\sup_{B_{\sigma/8}^{n+}(0)}\abs{w_j^\pm-a_j^\pm}
\leq C(n,\theta)E_\sigma^{\frac12}.
}

Let $F(p)=p/\sqrt{1+\abs{p}^2}$ for $p\in\mbR^n$.
Since $u_j^\pm$ solve the minimal surface equation with the capillary boundary condition, the functions $w_j^\pm$ in turn solve
\eq{\label{eq:w_j-MSE}
\operatorname{div}F(\mp\cot\theta e_1+Dw_j^\pm)=0
\quad\text{in }B_{\sigma/4}^{n+}(0),
}
with the constant oblique boundary condition
\eq{
F_1(\mp\cot\theta e_1+Dw_j^\pm)
=F_1(\mp\cot\theta e_1)
\quad\text{on }B_{\sigma/4}^{n}(0)\cap\{x_1=0\}.
}
The linearized boundary operator is uniformly oblique because
\eq{
\frac{\partial F_1}{\partial p_1}(\mp\cot\theta e_1)
=\sin^3\theta>0.
}
If $E_\sigma=0$, then $Dw_j^\pm=0$ and the estimate \eqref{esti:C^2-u^pm} follows.
Otherwise, we consider the rescaled function $v(y)=\frac{w_j^\pm(\sigma y)-a_j^\pm}{\sigma E_\sigma^{1/2}}$, which satisfies $Dv(y)=\frac{Dw_j^\pm(\sigma y)}{E^{1/2}_\sigma}$, so \eqref{eq:w_j-MSE} becomes
\eq{
{\rm div}_yF(\mp\cot\theta e_1+E_\sigma^{1/2}Dv(y))=0
\quad\text{in }B_{1/4}^{n+}(0),
}
and can be further rewritten as (put $\varepsilon=E_\sigma^{1/2}$ for simplicity)
\eq{
{\rm div}_y\left(\mcL_\varepsilon(Dv)\right)=0
\quad\text{in }B_{1/4}^{n+}(0),
}
where $\mcL_\varepsilon: p\mapsto
\frac{F(\mp\cot\theta e_1+\varepsilon p)-F(\mp\cot\theta e_1)}{\varepsilon}$ is the normalized interior operator.
In fact, by Taylor expansion
\eq{
F(\mp\cot\theta e_1+\varepsilon p)
=F(\mp\cot\theta e_1)+\varepsilon DF(\mp\cot\theta e_1)p+O(\varepsilon^2\abs{p}^2),
}
and hence $\mcL_\varepsilon(p)=DF(\mp\cot\theta e_1)p+O(\varepsilon\abs{p}^2)$, which approximates the linearized minimal surface equation at the slope $\mp\cot\theta e_1$ when $\varepsilon\searrow0$.
In particular, after decreasing $\varepsilon_0$ if necessary, depending only on $n,\theta$, one has:
\eq{\label{ineq:mcL_varepsilon-uniform-elliptic}
\lambda_{min}\abs{\xi}^2
\leq\left<D\mcL_\varepsilon(Dv)\xi,\xi\right>
\leq\lambda_{max}\abs{\xi}^2,\quad\forall\xi\in\mbR^n,
}
where $0<\lambda_{min}\leq\lambda_{max}$ are constants depending only on $\theta$.
Here we have used the fact that $D\mcL_\varepsilon(Dv)=DF(\mp\cot\theta e_1+\varepsilon Dv)$, the standard estimates
\eq{
\frac{\abs{\xi}^2}{(1+\abs{p}^2)^{3/2}}
\leq\left<DF(p)\xi,\xi\right>
\leq\frac{\abs{\xi}^2}{\sqrt{1+\abs{p}^2}},
}
as well as $\abs{Dv}\leq C(n,\theta)$, thanks to \eqref{ineq:w_j-gradient-estimate}.
This shows that the normalized interior operator $\mcL_\varepsilon$ is uniformly elliptic.
As a by-product of $\abs{Dv}\leq C(n,\theta)$, we find $\norm{v}_{C^0(B^{n+}_{1/4})}\leq C(n,\theta)$.

Similarly, for the normalized boundary operator $\mcB_\varepsilon:p\mapsto\frac{F_1(\mp\cot\theta e_1+\varepsilon p)-F_1(\mp\cot\theta e_1)}{\varepsilon}$, after decreasing $\ep_0$, one has
\eq{
\frac{\p\mcB_\varepsilon}{\p p_1}(Dv)
=\frac{\partial F_1}{\partial p_1}(\mp\cot\theta e_1+Dw_j^\pm)
\geq\frac{1}{2}\sin^3\theta>0.
}
Thus the rescaled $v$ satisfies a uniformly elliptic PDE, a uniformly oblique boundary condition, and also the uniform bounds $\norm{v}_{C^0}\leq C(n,\theta), \abs{Dv}\leq C(n,\theta)$.
Using the estimates in \cite{LT86},
together with the standard interior estimates (cf. \cite{SS81}), we deduce after rescaling back $v$ that
\eq{
\sigma\sup_{B_{\sigma/8}^{n+}(0)}\abs{D^2u_j^\pm}
\leq C(n,\theta)E_\sigma^{\frac12}.
}
The proof is thus completed.
\end{proof}

Then we consider the case when the capillary hypersurface is close to the supporting hyperplane $\p\mbH^{n+1}$.

\begin{lemma}\label{lem:wall-De-Giorgi}
Under the assumptions of Lemma \ref{Lem:normalized-Sobolev}, set the classical tilt function $\mathbf g=\sqrt{1-\nu_1^2}$.
There are positive constants
$\varepsilon'=\varepsilon'(n,\theta)$ and
$C=C(n,\theta)>0$ with the following property.
If $0<R\leq1$ and
\[
\mathbf E_R
=R^{-n}\int_{M\cap((-R,R)\times B_R^n(0))}\mathbf g^2\rd\mcH^n
\leq\varepsilon',
\]
then
\eq{\label{eq:wall-De-Giorgi}
\sup_{M\cap((-R/2,R/2)\times B_{R/2}^n(0))}\mathbf g
\leq C\mathbf E_R^{\frac12}.
}
In particular, if the right-hand side is smaller than $\frac12\sin\theta$, then
\eq{
\partial M\cap((-R/2,R/2)\times B_{R/2}^n(0))
=\emptyset.
}
\end{lemma}
\begin{proof}
By \cite[(2.7)]{SS81} (see also \cite[(15)]{Bellettini25}), on
$\{\mathbf g>0\}$ we have
\eq{\label{eq:wall-pointwise-Jacobi}
\mathbf g(\Delta+\abs{A}^2)\mathbf g
=\abs{A}^2-\frac{\abs{\na\mathbf g}^2}{1-\mathbf g^2}
\geq\frac1n\abs{A}^2.
}
By \eqref{eq:nu_and_eta} and \eqref{eq:lemBoundaryDeri}, along $\p M$ one has $\mathbf g=\sin\theta$ and
\eq{\label{eq:mfg-bdry}
\partial_\eta\mathbf g=-\cos\theta\,A(\eta,\eta),
\qquad
\partial_\eta\mathbf g
+\cot\theta\,A(\eta,\eta)\mathbf g=0.
}
For the ansatz $h_\theta$ we have \eqref{eq:h_theta}.

For $0<\ell<L\leq(2n)^{-1}$, we use $\bigl(\mathbf g-\ell h_\theta\bigr)^+$ in place of $(g_\theta-\ell h_\theta)^+$ in the proof of Lemma \ref{lem:capillary-truncated-caccioppoli}.
Using \eqref{eq:wall-pointwise-Jacobi}, \eqref{eq:mfg-bdry}, and \eqref{eq:h_theta}, we can repeat the same computations as in Lemma \ref{lem:capillary-truncated-caccioppoli} and deduce
\eq{
\int_M\left|\na\left(\varphi(\mathbf g-Lh_\theta)^+\right)\right|^2\rd\mcH^n
\leq\frac{C(n,\theta)}{1-\ell/L}
\int_M\bigl((\mathbf g-\ell h_\theta)^+\bigr)^2
\abs{\na\varphi}^2\rd\mcH^n.
}
The same iteration argument in the proof Lemma \ref{lem:capillary-De-Giorgi} then yields \eqref{eq:wall-De-Giorgi}.

Finally, since every point of $\partial M$ satisfies $\mathbf g=\sin\theta$, we see that \eqref{eq:wall-De-Giorgi} excludes capillary boundary points from the smaller cylinder, provided that $\varepsilon'=\varepsilon'(n,\theta)$ is sufficiently small. This completes the proof.
\end{proof}

\begin{theorem}[Sheeting theorem: second version]\label{thm:sheeting-2nd}
Let $n\geq2, \theta\in(0,\pi)$, $\Lambda\in[1,\infty)$.
Let $ \overline{V}\in \mathscr{V}(\theta,\Lambda) $.
Denote by $M,V,W$ the corresponding hypersurface and varifolds as in Definition \ref{defn:varifold-class}.

There exists a positive constant $\ep'_0\in(0,1)$, depending only on $n,\theta,\Lambda$, with the following property:
	if for some $\ep\in[0,\ep_0')$,
\eq{\label{condi:hausdoff-distance-x_1=0}
\sup\left\{x_1:(x_1,x')\in\overline M\cap
\left(\mbR\times B_2^n(0)\right)\right\}\leq\ep,
}
then
\eq{
	\overline{M}\cap\left(\mbR\times B^n_{\frac{1}{8}}(0)\right)
        =\bigcup_{j\in Q} {\mathrm{graph}(u_j)},
}
where $u_j:B_{\frac{1}{8}}^{n}(0)\rightarrow \mathbb{R}$,
$j\in Q\coloneqq\left\{1,\cdots,q\right\}$, are smooth functions, with
$q\geq0$, whose
graphs
\[
\left\{\left(u_j(x),x\right):x\in B^n_{\frac{1}{8}}(0)\right\}
\]
are minimal and without boundary in $\mbR\times B^n_{\frac{1}{8}}(0)$.
Moreover, $u_j\leq u_{j+1}$ for $j=1,2,\cdots ,q-1$, and
    \eq{
    (\frac{1}{4})^{-1}\sup_{B^{n}_{\frac{1}{8}}(0)}\abs{u_j}
    +\sup_{B^{n}_{\frac{1}{8}}(0)}\abs{Du_j}
    +\frac{1}{4}\sup_{B^{n}_{\frac{1}{8}}(0)}\abs{D^2u_j}
    \leq C\ep^\frac{1}{2},
    }
where $C=C(n,\theta,\Lambda)\in(0,\infty)$.
\end{theorem}

\begin{proof}
Set $\mathbf g=\sqrt{1-\nu_1^2}$.
As discussed in the beginning of Section \ref{sec:capillaryFirstVariation}, we assume WLOG that $(V,W)$ is $\mbF_\theta$-stationary for $\theta\in[\frac\pi2,\pi)$.
Repeating the proof of Proposition \ref{Prop:tilt-excess-control-x_1} with a cutoff adapted to the cylinder $(-1,1)\times B_1^n(0)$, and
then using Lemma \ref{Lem:Ahlfors-regularity}, gives
\eq{
\int_{M\cap((-1,1)\times B_1^n(0))}\mathbf g^2\rd\mcH^n
\leq C(n,\theta,\Lambda)\ep^2.
}
After decreasing $\ep'_0=\varepsilon'(n,\theta,\Lambda)$, we can apply Lemma \ref{lem:wall-De-Giorgi} with $R=1$, and obtain
\eq{
\sup_{M\cap((-1/2,1/2)\times B_{1/2}^n(0))}\mathbf g
<\frac12\sin\theta,
}
so that
\eq{
\partial M\cap((-1/2,1/2)\times B_{1/2}^n(0))
=\emptyset.
}
Namely, the hypersurface is a smooth stable minimal embedding with empty boundary in the cylinder.
After rotation, the interior sheeting theorem \cite[Theorem 7]{Bellettini25} applies at scale $R=1/4$, and yields the assertions.
\end{proof}

\section{Compactness and regularity}
\label{sec:mainCompactness}

\subsection{Notations}

Our goal is to prove the regularity of varifolds in $\overline{\mathscr{V}}(\theta,\Lambda)$ and show that they are indeed induced by stable capillary minimal hypersurfaces.
It is convenient to work with the following notations.

\begin{definition}[classical cones and $\theta$-classical cones]
\normalfont
For $\beta \in [0,\pi]$, define the half-hyperplane

\eq{
H_\beta\coloneqq\left\{ (r \sin \beta,x_2,\ldots,x_{n}, -r\cos\beta) \in \overline{\mbH^{n+1}}
: r\ge 0\right\}.
}

We define the class of {\em classical cones} by
\eq{\label{defn:classical-cones}
\mathscr{C}
\coloneqq\left\{
q_0 |H_0| + \sum_{i=1}^m p_i |H_{\theta_i}| + q_\pi |H_\pi| :
\begin{aligned}
&m \ge 1,\quad p_i \in \mathbb Z_{>0},\quad
q_0,q_\pi \in \mathbb{R}_{\geq0}, \\
&\theta_i\in(0,\pi),\quad \theta_i\neq\theta_j\text{ for }i\neq j
\end{aligned}
\right\}.
}
Of particular interest is its subclass, called \emph{$\theta$-classical cones}, defined as
\eq{\label{defn:theta-cones}
\mathscr C_\theta \coloneqq
\left\{
\begin{aligned}
&q_0 |H_0| + p_1 |H_{\pi-\theta}| + p_2 |H_\theta| + q_\pi |H_\pi| : \\
&p_1,p_2 \in \mathbb Z_{\ge 0},\quad p_1+p_2>0,\quad
q_0,q_\pi \in \mathbb{R}_{\geq0}
\end{aligned}
\right\}.
}
\end{definition}

For a $\mathbf C \in \mathscr C$ expressed as
\eq{
\mathbf C = q_0 |H_0| + \sum_{i=1}^m p_i |H_{\theta_i}| + q_\pi |H_\pi|,
}
we define
\eq{
\underline{\theta}_1 & \coloneqq
\min\left\{|\theta_i-\theta_j|:1\leq i<j\leq m\right\}, \\[0.3em]
\underline{\theta}_2 & \coloneqq \min \left\{
|\theta_i - \theta|,\ |\theta_i - (\pi-\theta)| :
\theta_i \neq \theta,\pi-\theta,\ 1 \le i \le m
\right\}, \\[0.3em]
\underline{\theta}_3 & \coloneqq
\min\left\{\theta_i,\pi-\theta_i:1\leq i\leq m\right\}, \\[0.3em]
\underline{\theta} & \coloneqq \min\left\{ \underline{\theta}_1,\underline{\theta}_2,\underline{\theta}_3 \right\}.
}
Here and below, the minimum of an empty set in the definitions of
$\underline\theta_1$ and $\underline\theta_2$ is understood to be $+\infty$.

For $1 \le i \le m$, we define the neighborhoods of $H_{\theta_i}$ by
\eq{
\mathcal{N}_i \coloneqq \bigcup_{|\beta-\theta_i|<\underline{\theta}/3} H_\beta,
}
and
\eq{
\mathcal{N}_0 \coloneqq \bigcup_{\beta<\underline{\theta}/3} H_\beta,
\qquad
\mathcal{N}_\pi \coloneqq \bigcup_{\pi-\beta<\underline{\theta}/3} H_\beta.
}
Note that the definitions of $\underline{\theta}$ and $\mathcal{N}_i$ depend on the choice of $\mathbf{C}$, but this dependence is clear from the context.

For $\tau>0$, we put
\eq{
S_\tau \coloneqq \left\{ (x_1,\ldots,x_{n+1}) : x_1^2 + x_{n+1}^2 < \tau^2 , x_2^{2}+\cdots +x_n^{2}\le 1\right\}.
}
For $y \in \mathbb R^{n-1}$, we define
\eq{
P_y \coloneqq \left\{ (x_1,\ldots,x_{n+1}) \in \overline{\mbH^{n+1}} :
(x_2,\ldots,x_n)=y \right\}.%
}

\subsection{Minimum distance theorem}

\begin{theorem}\label{thm:coneMinDistGraph}
Let $n\geq2, \theta\in(0,\pi), \Lambda\in[1,\infty)$.
Let $\mathbf C \in \mathscr C$.
\begin{enumerate}[label=\textup{(\Roman*)}]
    \item \label{it:thm:miniDist}
If $\mathbf C\in\mathscr C\setminus\mathscr C_\theta$, then there exists
$\varepsilon=\varepsilon(\Lambda,\theta,n,\mathbf C)>0$ such that every
$\overline V\in\mathscr V(\theta,\Lambda)$ satisfies
\eq{
{\rm dist}_{\mcH}\!\left(
\mathrm{spt}\,\|\overline V\|\cap B_2,\;
\mathrm{spt}\,\|\mathbf C\|\cap B_2
\right) \geq \varepsilon .
}
    \item \label{it:thm:graph}If $\mathbf C = q_0 |H_0| + p_1 |H_{\pi-\theta}| + p_2 |H_\theta| + q_\pi |H_\pi|\in \mathscr C_\theta$, then for any $\tau>0$ there exists
$\varepsilon=\varepsilon(\Lambda,\theta,n,\mathbf C,\tau)>0$ such that 
for any $\overline{V}\in \mathscr{V}(\theta,\Lambda)$ with
\eq{
{\rm dist}_{\mcH}\!\left(
\mathrm{spt}\,\|\overline V\|\cap B_2,\;
\mathrm{spt}\,\|\mathbf C\|\cap B_2
\right) < \varepsilon ,
}
the following conclusions hold:
\begin{enumerate}[label=\textup{(\alph*)}]
\item For $i=1,2$, the set $\left(M\cap S_1\cap\mathcal{N}_i\right)\setminus S_\tau$ consists of finitely many
connected graphical components.
After indexing them by their normal, we denote them by
$M_{i,1},\ldots,M_{i,N_i}$, $i=1,2$, where they are graphed over domains in $H_{\pi-\theta}$ for $i=1$ and in $H_\theta$ for $i=2$ (when $\theta=\frac\pi2$, $H_{\pi-\theta}$ and $H_\theta$ coincide, and the index $i$ records only the two different orientations).
\item For $i=0,\pi$, $\left(M\cap S_{1} \cap\mathcal{N}_i\right)\setminus S_\tau = \emptyset$.
\item Let $\nu_{i,j}$ denote the unit normal of $M_{i,j}$.
Then
\eq{
|\nu_{1,j}-\nu_{-\theta}|<\tfrac{\underline{\theta}}{4},
\qquad
|\nu_{2,j}-\nu_\theta|<\tfrac{\underline{\theta}}{4},
}
\item Writing $u_{i,j}$ for the graphing function of $M_{i,j}$, then
\eq{\label{ineq:u_i,j-C^2-estimate}
\|u_{i,j}\|_{C^2(\mathrm{dom}(u_{i,j}))}
\le
C\,
{\rm dist}_{\mcH}\!\left(
\mathrm{spt}\,\|\overline V\|\cap B_2,\;
\mathrm{spt}\,\|\mathbf C\|\cap B_2
\right).
}
\end{enumerate}
\end{enumerate}
\end{theorem}

{
\begin{proof}
In either case, we may assume that there exists a sequence of varifolds $\left\{\overline{V}_k\in\mathscr{V}(\theta,\Lambda)\right\}_{k\in\mbN}$ converging to $\overline{V}$, such that
\eq{\label{assump:dist-to-cone-converges-to-0}
{\rm dist}_{\mcH}\!\left(
\mathrm{spt}\,\|\overline V_k\|\cap B_2,\;
\mathrm{spt}\,\|\mathbf C\|\cap B_2
\right) \to 0 .
}

We shall derive a contradiction in case \ref{it:thm:miniDist}, and establish the conclusions in case \ref{it:thm:graph} for sufficiently large $k$.

Fix $\tau>0$. By \eqref{assump:dist-to-cone-converges-to-0}, we can use the interior tilt and sheeting theorems \cite[Theorems 4 and 7]{Bellettini25}, also Theorem \ref{thm:sheeting-2nd}, to conclude that,
after passing to a subsequence, there exist nonnegative integers
$N_0,N_\pi$, and positive integers $N_i$, $1\leq i\leq m$, such that for all sufficiently large $k$ the following hold:

\begin{enumerate}
\item
For each $i=0,1,\ldots,m,\pi$, the set
$\left(M_k \cap S_{1} \cap\mathcal{N}_i\right) \setminus S_\tau$
consists of exactly $N_i$ connected components, each of which can be written as the graph of a function
over a domain in $H_{\theta_i}$.
We denote these components by
$M_{k,i,1},\ldots,M_{k,i,N_i}$,
with corresponding graph functions
$u_{k,i,1},\ldots,u_{k,i,N_i}$.

\item
For every $i,j$, the $C^2$-norm of $u_{k,i,j}$ converges to $0$, as $k\to\infty$.
In fact, an estimate of the form \eqref{ineq:u_i,j-C^2-estimate} holds.

\item
Let $\nu_{k,i,j}$ denote the unit normal of $M_{k,i,j}$.
Then $\nu_{k,i,j}$ converges uniformly to either $\nu_{\theta_i}$ or $-\nu_{\theta_i}$.
In particular, after possibly passing to a further subsequence, we may assume that for each $i,j$,
\eq{\label{ineq:nu_k,i,j}
|\nu_{k,i,j}-\nu_{\theta_i}|<\tfrac{\underline{\theta}}{4}
\quad\text{or}\quad
|\nu_{k,i,j}+\nu_{\theta_i}|<\tfrac{\underline{\theta}}{4}.
}
\end{enumerate}

\medskip
By Sard's theorem, for almost every $y\in B_1^{n-1}\subset\mathbb R^{n-1}$,
the hypersurface $M_k\cap S_\tau$ intersects $P_y$ transversely.
Consequently,
\eq{
M_k\cap P_y\cap S_\tau=\bigcup_{\gamma\in\Gamma}\gamma,
}
where $\Gamma$ is a finite collection of smooth, properly immersed curves in $P_y\cap S_\tau$ satisfying:

\begin{enumerate}
\item
Each $\gamma\in\Gamma$ is either a closed smooth embedded Jordan curve, or a smooth curve with two endpoints lying in
$(\partial S_\tau\cap P_y)\cup (S_\tau\cap P_y\cap\{x_1=0\})$.

\item
If $\gamma$ has endpoints, then it may have self-intersections only at endpoints lying in
$S_\tau\cap P_y\cap\{x_1=0\}$.

\item
Each endpoint of $\gamma$ lies either on $M_{k,i,j}\cap P_y$ for some $i,j$,
in which case $\gamma$ and $M_{k,i,j}\cap P_y$ together form a smooth curve near the endpoint,
or on $S_\tau\cap P_y\cap\{x_1=0\}$, where $\gamma$ meets the boundary with nonzero contact angle.

\item
If $\gamma_1,\gamma_2\in\Gamma$ are distinct, then $\gamma_1\cap\gamma_2\neq\emptyset$
can occur only at endpoints lying in $S_\tau\cap P_y\cap\{x_1=0\}$.
\label{pf:item:curve-intersection}
\end{enumerate}

We decompose $\Gamma=\Gamma_1\cup\Gamma_2\cup\Gamma_3$, where:
\begin{enumerate}
\item
$\Gamma_1$ consists of curves whose two endpoints lie on some $M_{k,i,j}$;
\item
$\Gamma_2$ consists of curves with exactly one endpoint on some $M_{k,i,j}$;
\item
$\Gamma_3$ consists of curves with no endpoints, or whose endpoints lie entirely on $\{x_1=0\}$.
\end{enumerate}

\medskip
We consider first $\mathbf C\in \mathscr{C}\backslash \mathscr{C}_\theta$, and
observe:

\medskip
\noindent{\em Claim 1.
For almost every $y\in B_1^{n-1}$,
there exists $\gamma\in\Gamma_1\cup\Gamma_2$ such that the angle between the unit normal vector field
$\nu$ at the two endpoints of $\gamma$ is bounded below by $\underline{\theta}/2$.}

\medskip
\noindent
{\bf Case 1: $\Gamma_1\neq\emptyset$.}
Choose $\gamma\in\Gamma_1$ with endpoints
$X_1\in M_{k,i_1,j_1}$ and $X_2\in M_{k,i_2,j_2}$. We consider three subcases. See Figure \ref{fig:curves} for an illustration.
\begin{figure}[ht]
    \centering
	\begingroup
	\def\svgwidth{1\columnwidth}
	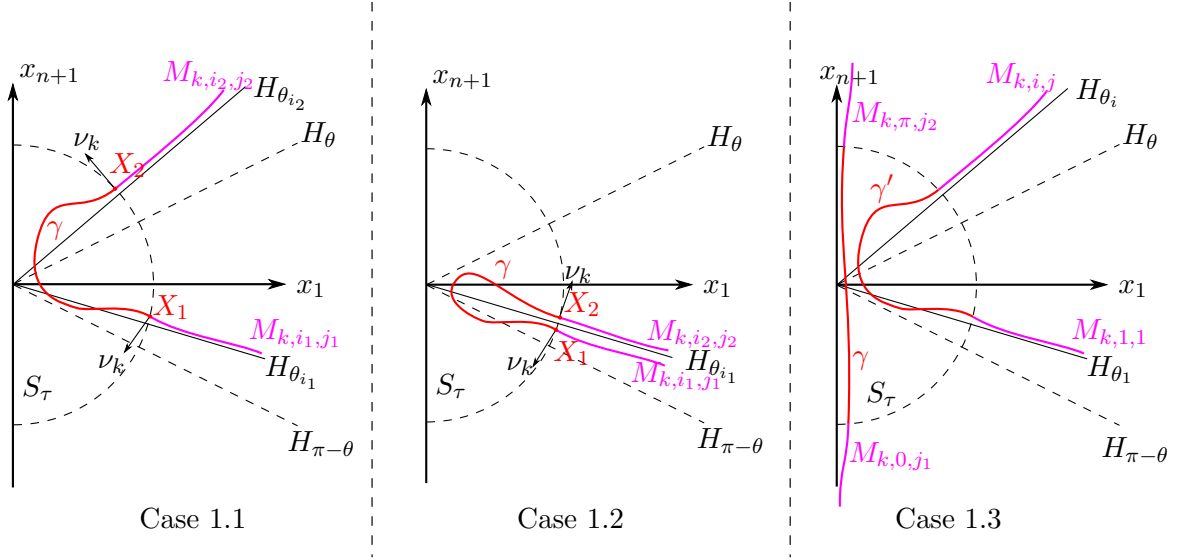
	\endgroup

    \caption{The case when $\gamma\in \Gamma_1$}
    \label{fig:curves}
\end{figure}

\smallskip
\emph{Case 1.1: $i_1\neq i_2$ and at least one of $i_1,i_2$ is not $0$ or $\pi$.}
By construction and the definition of $\underline{\theta}$,
the angle between $\nu_{k,i_1,j_1}(X_1)$ and $\nu_{k,i_2,j_2}(X_2)$ is bounded below by
\eq{
\min\bigl\{|\theta_{i_1}-\theta_{i_2}|,\;\pi-|\theta_{i_1}-\theta_{i_2}|\bigr\}
-\tfrac{\underline{\theta}}{2}
\;\ge\; \tfrac{\underline{\theta}}{2}.
}
Here we adopt the convention $\theta_0=0$ and $\theta_\pi=\pi$.

\smallskip
\emph{Case 1.2: $i_1=i_2$.}
Either the angle between the normals at $X_1$ and $X_2$ is at least $\underline{\theta}/2$,
or it is less than $\underline{\theta}/2$.
The latter case cannot occur:
since $P_y$ intersects $M_k$ transversely,
the projection of the unit normal $\nu_k$ onto $P_y$ defines a nonvanishing normal vector field along $\gamma$,
whose total change along $\gamma$ must be at least $\pi-\underline{\theta}/2$.

\smallskip
\emph{Case 1.3: $i_1=0$ and $i_2=\pi$.}
By the intersection property \eqref{pf:item:curve-intersection},
there exists a curve $\gamma'\in\Gamma_1$ whose endpoints connect nontrivially to some
$M_{k,i,j}$ with $i\notin\{0,\pi\}$.
This reduces to \emph{Case~1.1} or \emph{Case~1.2}.

\medskip
\noindent
{\bf Case 2: $\Gamma_1=\emptyset$, and $\Gamma_2\neq\emptyset$.}

Since $\mathbf C\in \mathscr{C}\backslash \mathscr{C}_\theta$, there exists some $\theta_i\notin\{\theta,\pi-\theta\}$.
Choose $\gamma\in\Gamma_2$ with one endpoint $ X_1 $ on $M_{k,i,j}$.
At the other endpoint $X_2$, $M_k$ meets $\{x_1=0\}$ with contact angle $\theta$.
By the definition of $\underline{\theta}$, the angle between the normals at the two endpoints is bounded below by $\underline{\theta}/2$.
See Figure \ref{fig:curveCase2} for an illustration.

\medskip

\noindent{\bf Case 3: $\Gamma_1=\Gamma_2=\emptyset$.}

By construction, $N_i>0$ for each $i\in\{1,\cdots,m\}$.
Hence, this case reduces to {\em Case 1.3}, which yields a contradiction.

\begin{figure}[ht]
    \centering
    \begin{minipage}[b]{0.45\textwidth}
        \centering
	\begingroup
	\def\svgwidth{0.74\columnwidth}
	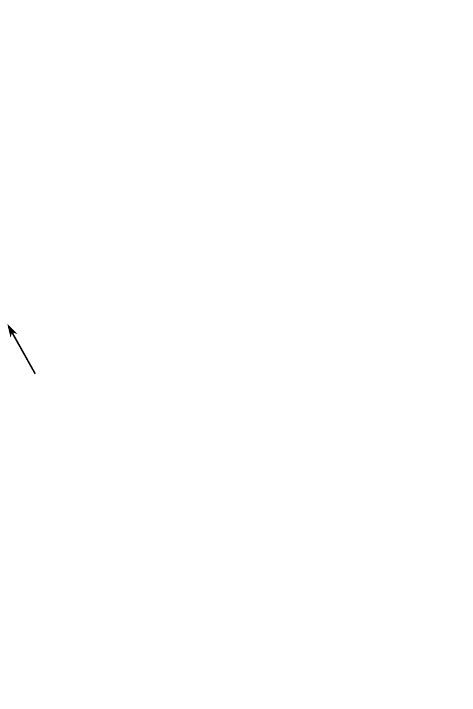
	\endgroup

        \caption{The case when $\gamma\in \Gamma_2$ and $\mathbf{C}$ is not in $\mathscr{C}_\theta$}
        \label{fig:curveCase2}
    \end{minipage}
    \hfill
    \begin{minipage}[b]{0.45\textwidth}
        \centering
	\begingroup
	\def\svgwidth{0.74\columnwidth}
	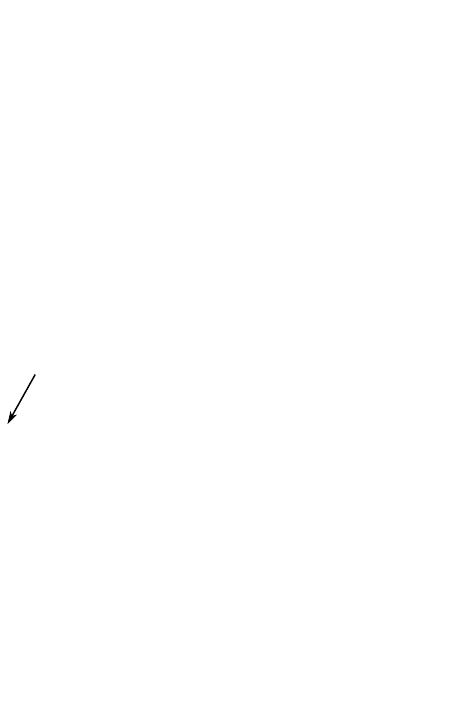
	\endgroup

        \caption{The case when the normal vector is far from $\nu_\theta$}
        \label{fig:curveWrongNormal}
    \end{minipage}
\end{figure}

\medskip

{\em Claim 1} is thus proved, consequently,
\eq{
\frac{\underline{\theta}}{2}
\le
\int_{M_k\cap P_y\cap S_\tau} |A|\,\rd\mcH^1 .
}
Integrating over $y\in B^{n-1}_1$ and applying the coarea formula, we obtain
\eq{\label{ineq:curvature-concentration}
\frac{\underline{\theta}}{2}
\le
C(n)\int_{M_k\cap S_\tau} |A|\,\rd\mcH^n
\le
C(n)\,
\mathcal H^n(M_k\cap S_\tau)^{1/2}
\left(
\int_{M_k\cap B_{3/2}} |A|^2\,\rd\mcH^n
\right)^{1/2}
\le
C(n,\theta,\Lambda)\sqrt{\tau},
}
which is a contradiction for $\tau$ sufficiently small.
This completes the proof of (I).

\medskip

Now we prove (II).
Assume that $\mathbf C\in \mathscr{C}_\theta$ is expressed as
\eq{
\mathbf C
=
q_0|H_0|+q_\pi|H_\pi|
+p_1|H_{\pi-\theta}|+p_2|H_\theta|,\quad q_0,q_\pi\geq0,\ p_1+p_2>0,
}
{and we adopt the convention $\pi-\theta\eqqcolon\theta_1$, $\theta\eqqcolon\theta_2$, to be consistent with \eqref{defn:classical-cones}}.
We first observe that for all sufficiently large $k$,
\eq{
M_k\cap S_1\cap\mathcal{N}_0\setminus S_\tau=\emptyset,
\qquad
M_k\cap S_1\cap\mathcal{N}_\pi\setminus S_\tau=\emptyset .
}
Indeed, if this were false, then by repeating the above argument used to prove {\em Claim 1}, we would again arrive at a contradiction.

It is thus left to show (II)(c).
Here we only show that on each $M_{k,2,j}$ the unit normal vector satisfies (recalling \eqref{ineq:nu_k,i,j})
\eq{
|\nu_{k,2,j}-\nu_\theta|<\frac{\underline{\theta}}{4}.
}
An entirely analogous argument then applies to $M_{k,1,j}$.

Suppose, to the contrary, that there exists some $M_{k,2,j}$ such that
\eq{
|\nu_{k,2,j}+\nu_\theta|\le \frac{\underline{\theta}}{4}.
}

{\em Claim 2. There exists a set of $y\in B_1^{n-1}$ with $\mathcal H^{n-1}$-measure at least $\frac{1}{2}\omega_{n-1}$ such that for each such $y$, one can find a curve $\gamma\in \Gamma_2$ with one endpoint lying on $M_{k,2,j}\cap P_y$.}

Indeed, if this were false, then for a subset of $y\in B_1^{n-1}$ of measure at least $\frac{1}{2}\omega_{n-1}$, every curve $\gamma\in \Gamma$ with an endpoint on $M_{k,2,j}\cap P_y$ must belong to $\Gamma_1$, with the other endpoint lying on some $M_{k,i,j'}$.
By the arguments in \emph{Case~1.1} and \emph{Case~1.2} above, the angle between the unit normal vector field at the two endpoints of such a curve is bounded from below by $\underline{\theta}/2$ for almost every such $y$.
Applying the same co-area estimate as \eqref{ineq:curvature-concentration} then yields a contradiction.
This proves the claim.

Therefore, for a subset of $y\in B_1^{n-1}$ of measure at least $\frac{1}{2}\omega_{n-1}$, there exists $\gamma\in \Gamma_2$ with one endpoint on $M_{k,2,j}\cap P_y$.
At the other endpoint, $M$ meets $\{x_1=0\}$ with contact angle $\theta$.
It follows that the total change of the unit normal vector field $\nu_k$ is at least $\underline{\theta}/2$; see Figure~\ref{fig:curveWrongNormal} for an illustration.
As before, this leads to a contradiction.
The proof is thus completed.
\end{proof}
}

\begin{corollary}\label{Cor:tilt-excess-control-qualitative}
	Let $n\geq2, \theta\in(0,\pi), \Lambda\in[1,\infty)$.%
	Let $\mathbf{C}\in \mathscr{C}_\theta$.
	For any $\delta>0$, there exists $\varepsilon=\varepsilon(n,\theta,\Lambda,\mathbf{C},\delta)>0$ such that for any
	$\overline{V}\in\mathscr{V}(\theta,\Lambda)$,
	if%
	\eq{
		{\rm dist}_{\mcH}\!\left(
		\mathrm{spt}\,\|\overline V\|\cap B_{2}(0),\;
		\mathrm{spt}\,\|\mathbf C\|\cap B_{2}(0)
		\right) <2\varepsilon,
	}
	then
	\eq{
		\int_{M\cap B_1(0)} g_\theta^{2}\, \rd\mcH^n < \delta.
	}
	
\end{corollary}
\begin{proof}
	First, by virtue of Lemma \ref{Lem:Ahlfors-regularity}, we choose $\tau\in(0,1)$ sufficiently small, depending only on $n,\theta,\Lambda,\de$, such that $\mathcal{H}^n(\spt\norm{V}\cap S_\tau)<\frac{\delta}{8}$ for any $V$ associated to $\overline{V}\in\mathscr{V}(\theta,\Lambda)$.
	Then, by
    Theorem \ref{thm:coneMinDistGraph} \ref{it:thm:graph} and the inclusion $B_1(0)\subset S_1$,
    we can choose $\varepsilon>0$ sufficiently small, depending only on the stated quantities, such that
	\eq{
		\int_{\left(M\cap B_1(0)\right)\setminus S_\tau} g_\theta^{2}\, \rd\mcH^n
        < \frac{\delta}{2},%
	}
	where we have used the fact that (see Lemma \ref{Lem:g_theta-controlled-by-model-normals})
	\eq{
		g_\theta^{2}(X)\leq \min\{ |\nu(X)-\nu_\theta|^{2}, |\nu(X)-\nu_{-\theta}|^{2} \}.
	}
\end{proof}

\begin{theorem}[Sheeting theorem, qualitative version]\label{Thm:sheeting-epsilon-delta}
Let $n\geq2, \theta\in[\frac{\pi}{2},\pi)$, $\Lambda\in[1,\infty)$.%
Let $\mfC\in\mathscr{C}_\theta$,
let $\ep_0=\ep_0(n,\theta,\Lambda)$ be the constant from Theorem \ref{thm:sheetingHalf}.
For any $\de\in(0,\ep_0)$,
there exists a positive constant $\ep\in(0,1)$ depending only on $n,\theta,\Lambda,\mfC,\de$, with the following property:
For any $\overline{V}\in\mathscr{V}(\theta,\Lambda)$, with
\eq{
\dist_\mcH\left(\spt\norm{\overline{V}}\cap B_{2}(0),\spt\norm{\mfC}\cap B_{2}(0)\right)<\ep,
}
we have
\eq{
	\overline{M}\cap B_\frac{1}{2}(0)
        =\left(\left(\bigcup_{j\in Q^+}{\mathrm{graph}(u^+_j)}\right)\cup\left(\bigcup_{j\in Q^-}{\mathrm{graph}(u^-_j)}\right)\right)\cap B_\frac{1}{2}(0),
}
	where
    $u^\pm_j:\mathrm{dom}(u^\pm_j)\rightarrow \mathbb{R}$, $j\in Q^\pm\coloneqq\left\{1,\cdots,q^\pm\right\}$ are smooth functions whose graphs 
    $$\left\{\left(x,u_j^\pm(x)\right): x\in \mathrm{dom}(u^\pm_j)\right\},$$ 
    oriented by the unit normal pointing upwards for $u^+_j$ and downwards for $u^-_j$, are minimal and satisfy the capillary boundary condition. If $q^\pm>1$ then $u^\pm_j\leq u^\pm_{j+1}$ for $j=1,2,\cdots ,q^\pm -1$. In particular, for any $j\in Q^\pm$,
    \eq{
    \sup_{\mathrm{dom}(u^\pm_j)}\abs{u_j^\pm\pm\cot\theta x_1}+\sup_{\mathrm{dom}(u^\pm_j)}\abs{Du^\pm_j\pm\cot\theta e_1}+\sup_{\mathrm{dom}(u^\pm_j)}\abs{D^2u^\pm_j}
    \leq C\de^\frac{1}{2},%
    }
where $C=C(n,\theta,\Lambda)\in(0,\infty)$.

\end{theorem}

{
\begin{proof}
Let $\ep_1=\ep_1(n,\theta,\Lambda,\mfC,\de)$ be the constant from Corollary \ref{Cor:tilt-excess-control-qualitative},
then set $\ep=\min\{\ep_0,\ep_1\}$.
Applying Corollary \ref{Cor:tilt-excess-control-qualitative}, we deduce
\eq{
\int_{M\cap B_1(0)}g_\theta^2\rd\mcH^n
<\de.
}
Using Theorem \ref{thm:sheetingHalf} and a covering argument, we conclude the proof.
\end{proof}
}

\subsection{Boundary regularity}

\begin{definition}[weak $\theta$-regular points]\label{defn:theta-regular-point-capillary-vfld}
\normalfont
Let $\overline V\in \overline{\mathscr{V}}(\theta,\Lambda)$.
A point $X\in \mathrm{spt}\,\|\overline V\|\cap B_1(0)$ is called a {\em weak $\theta$-regular point}, denoted as $X\in \mathrm{reg}_\theta\,\overline V$,
if there exists $\rho>0$ such that one of the following holds:
\begin{enumerate}[label=\textup{(\roman*)}]
    \item
    $\mathrm{spt}\,\|\overline V\|\cap B_\rho(X)$ is an orientable, embedded $C^2$-minimal hypersurface without boundary.
    \item
    $X\in\partial \mathbb{H}^{n+1}\cap B_1(0)$, and
    \eq{
        \left(\mathrm{spt}\,\|\overline V\|\cap B_\rho(X)\right)\backslash \partial \mathbb{H}^{n+1}
        =
        \bigcup_{j=1}^N \Sigma_j \cap B_\rho^+(X),
    }
    where each $\Sigma_j$ is an embedded 
    $C^2$-hypersurface such that following conditions hold:
    \begin{enumerate}[label=\textup{(\alph*)}]
        \item $\partial \Sigma_j \cap B_\rho(X)\subset \partial \mathbb{H}^{n+1}$ for each $j$;
        \item there exists a unit normal vector field $\nu_j$ of $\Sigma_j$ such that $\langle \nu_j, e_1 \rangle=\cos \theta$ on $\p\S_j$; 
        \item the interiors of $\Sigma_i$ and $\Sigma_j$ are disjoint in $B_\rho(X)$ for $i\neq j$;
        \item any intersection between distinct components may occur only along $\partial \mathbb{H}^{n+1}$;
        \item%
        if $\Sigma_i\cap\Sigma_j\neq\emptyset$, then their boundaries are
        \begin{enumerate}
            \item [(e1)] either identical, and with the same induced unit normal in $\p\mbH^{n+1}$, which implies that $\Sigma_i$ and $\Sigma_j$ are identical;
            \item [(e2)] or mutually tangent within $\partial\mathbb{H}^{n+1}$, with opposite induced unit normals in $\partial \mathbb{H}^{n+1}$.
        \end{enumerate}
    \end{enumerate}
\end{enumerate}
We denote the {\em weak $\theta$-singular set} by $\mathrm{sing}_\theta\,\overline V := \mathrm{spt}\,\|\overline V\| \backslash \mathrm{reg}_\theta\,\overline V$.
\end{definition}
In particular, 0 is a weak $\theta$-regular point for $\mathbf{C}\in \mathscr{C}_\theta$, {but not a $\theta$-regular point in the sense of Definition \ref{Defn:theta-regular-points-manifold} unless $q_0=q_\pi=0$.}

\begin{theorem}\label{thm:singular-dimension-estimate-capillary-vflds}
Let $n\geq2, \theta\in[\frac{\pi}{2},\pi)$, $\Lambda\in[1,\infty)$.
For any $\overline{V}\in \overline{\mathscr{V}}(\theta,\Lambda)$,
we have $\mathcal{H}^{n-n_\theta+\delta}(\mathrm{sing}_\theta\,\overline{V}\cap B_2(0))=0$ for all $\delta\in (0,1)$, where $n_\theta$ is defined in Theorem \ref{thm:mainCompact}.
    \label{thm:regularityLimit}
\end{theorem}

\begin{proof}
Note that for ${\rm reg}_\theta{\overline{V}}\mid_{\mbH^{n+1}}$ (i.e., Definition \ref{defn:theta-regular-point-capillary-vfld} (i)) agrees with the notion of classical regular points, therefore by Schoen-Simon's (interior) regularity theorem \cite[Theorem 3]{SS81} we have $\mcH^{n-7+\de}\left({\rm sing}_\theta\overline{V}\cap\mbH^{n+1}\right)=0$.
So, it remains to consider the singularities on $\p\mbH^{n+1}$.
Put \eq{
\mathfrak{S}\coloneqq\mathrm{sing}_\theta\overline{V}\cap B_{1}(0)\cap \partial \mathbb{H}^{n+1}.
}

\noindent{\bf Claim.} $\mathcal{H}^{n-n_\theta+\delta}(\mathfrak{S})=0$ for all $\delta\in(0,1)$, and when $n=n_\theta$, $\mathfrak{S}$ is a discrete set.

By Lemma \ref{Lem:Ahlfors-regularity-capillary-vfld}, Proposition
\ref{Prop:bdd-1st-variation-free-bdry-vfld}, and Lemma
\ref{Lem:Monotonicity-formula-doubling}, the set
${\rm VarTan}(\overline V,X_0)$ is nonempty for every $X_0\in\mathfrak S$.
Every $\mfC\in{\rm VarTan}(\overline V,X_0)$ is a cone.
To prove the theorem, we analyze such tangent cones. In particular, we have the following.
    \begin{lemma}\label{lem:tangen-cone-splitting-capillary-vflds}
        \label{lem:coneDimReduction}
        For any $\mathbf{C} \in \mathrm{VarTan}(\overline{V},X_0)$ with $X_0 \in\mathfrak{S}$, we can write $\mathbf{C}=\mathbf{C}'\times \mathbb{R}^{n-p}$ for some $p\ge n_\theta$.
    \end{lemma}
    \begin{proof}
        [Proof of Lemma \ref{lem:coneDimReduction}]
        For any cone $\mathbf{C}$, we write $\mathcal{S}(\mathbf{C})$ (the {\em spine} of $\mathbf{C}$) to be the linear subspace containing all $X\in \partial \mathbb{H}^{n+1}$ such that $\mathbf{C}$ is invariant under the translation along the line spanned by $X$.
        For any $X_0 \in\mathfrak{S}$, we introduce the notion of {\em iterated tangents} of $\overline{V}$ at $X_0$ as follows.
        We say a collection of cones $\left\{ \mathbf{C}_1,\mathbf{C}_2,\cdots ,\mathbf{C}_N \right\}$ is iterated tangents of $\overline{V}$ at $X_0$ if $\mathbf{C}_1$ is a tangent cone of $\overline{V}$ at $X_0$, and $\mathbf{C}_{j+1}$ is a tangent cone of $\mathbf{C}_j$ at $X_j \in \mathrm{sing}_\theta\mfC_j\backslash \mathcal{S}(\mathbf{C}_j)$ for $1\le j\le N-1$.
        
        Note that each tangent cone $\mathbf{C}_j$ is stationary with free boundary in $\mbH^{n+1}$.
        Moreover, we can make sure the iterated tangents satisfy the following properties:
        \begin{enumerate}[(a)]
            \item $\mathrm{sing}_\theta\mfC_j\neq \emptyset$.
            \item $\mathrm{dim}(\mathcal{S}(\mathbf{C}_{j+1}))>\mathrm{dim}(\mathcal{S}(\mathbf{C}_{j}))$ for each $j=1,2,\cdots ,N-1$.
            \item $\mathbf{C}_N=\mathbf{C}'\times \mathbb{R}^{\mathrm{dim}(\mathcal{S}(\mathbf{C}_N))}$ where $\mathrm{sing}\,\mathbf{C}'=\{ 0 \}$.
            \item For each $1\le j\le N$, we can find a sequence of points $\left\{ Y_k \right\}$ with $Y_k\rightarrow X_0$, a sequence of positive real numbers $\left\{ r_k \right\}$ with $r_k\rightarrow 0^+$ as $k\rightarrow \infty$, and a sequence $\left\{\overline{V_k}\in\mathscr{V}(\theta,\Lambda)\right\}_{k\in\mbN}$ such that $(\bseta_{Y_k,r_k})_\#\overline{V_k}$ converges, in the sense of varifolds, to $\mfC_j$.
            This also implies that $\spt\norm{(\bseta_{Y_k,r_k})_\#\overline{V_k}}$ converges in the sense of Hausdorff distance to $\spt\norm{\mfC_j}$, thanks to Lemma \ref{Lem:Ahlfors-regularity-capillary-vfld}.
            Moreover, we note that up to a different $\Lambda'$, depending only on $n,\theta,\Lambda$, we have $(\bseta_{Y_k,r_k})_\#\overline{V_k}\in\mathscr{V}(\theta,\Lambda')$ for each $k\in\mbN$.
            At a weak $\theta$-regular point of $\mathbf C_j$, smooth convergence follows
            from Theorem \ref{Thm:sheeting-epsilon-delta} in the capillary case,
            from Theorem \ref{thm:sheeting-2nd} when the tangent plane is $\p\mbH^{n+1}$, and from the interior sheeting theorem
            \cite[Theorem 7]{Bellettini25} at an interior point.
        \end{enumerate}
        In particular, (b) implies $N$ is a finite number, and (d) implies that the weak $\theta$-regular part of $\mathbf{C}_j$ is stable.

        Now, let us determine the dimension of $\mathbf{C}'$.
        Note that the dimension of $\mathbf{C}_N$ is at least one.

        If the dimension of $\mathbf{C}'$ is one, then either
        $\operatorname{spt}\|\mathbf C_N\|\subset\p\mbH^{n+1}$ or
        $\mathbf C_N\in\mathscr C$. In the first case, stationarity implies
        that $\mathbf C_N$ is a multiple of $\p\mbH^{n+1}$. In the second
        case, Theorem \ref{thm:coneMinDistGraph} and (d) imply that
        $\mathbf C_N\in\mathscr C_\theta$. Thus in either case,
        ${\rm sing}_\theta\mathbf C_N=\emptyset$, contradicting (a).

        If $2 \leq \dim(\mathbf{C}') < n_\theta$, {in view of (d) we know that the stability inequality \eqref{ineq:SS81-(1.17)-smooth-case} holds on ${\rm reg}_\theta\mfC'$.
        Hence} by the classification of stable capillary cone (Theorem \ref{thm:stable-capillary-cone-isolated-singularity}), we will have $\mathrm{sing}_\theta\mfC'=\emptyset$, which again contradicts (a).

        Therefore, we know $\mathbf{C}'$ has dimension at least $n_\theta$.
        Hence, by (b), we know $\dim(\mathcal{S}(\mathbf{C}))\le n-n_\theta$ for any $\mathbf{C}\in \mathrm{VarTan}(\overline{V},X_0)$, and the lemma follows.
    \end{proof}
    
    For $n\ge n_\theta$, using Lemma \ref{lem:coneDimReduction} we can apply Federer's dimension reducing principle (cf., \cite[Appendix A]{Simon83}) to get $\dim_{\mathcal{H}} \left(\mathrm{sing}_\theta\overline{V}\cap B_1(0)\right)\le n-n_\theta$, and when $n<n_\theta$, we can directly apply Lemma \ref{lem:coneDimReduction} to get $\mathrm{sing}_\theta\overline{V}\cap B_1(0)=\emptyset$.
    
    At last, we need to show when $n=n_\theta$, $\mathrm{sing}_\theta\overline{V}\cap B_1(0)$ is discrete. We argue by contradiction, and assume that there exists a sequence of points $\left\{ X_i \right\}_{i\in\mbN}\subset \mathrm{sing}_\theta\overline{V}\cap B_1(0)$, such that $X_i\rightarrow X_0$ for some $X_0 \in \mathrm{sing}_\theta\overline{V}\cap B_1(0)$.
    Up to a subsequence, we can assume $(\bseta_{X_0,\rho_i})_{\#}\overline{V}$ converges to some $\mathbf{C} \in \mathrm{VarTan}(\overline{V},X_0)$, where $\rho_i=|X_i-X_0|$, and assume $Y=\lim_{i\rightarrow \infty} \frac{X_i-X_0}{\rho_i}\neq 0$.
    Since $n=n_\theta$, Lemma \ref{lem:coneDimReduction} shows that $Y$ is a
    weakly $\theta$-regular point of $\mathbf C$.
    The sheeting Theorems (\ref{Thm:sheeting-epsilon-delta}, Theorem
    \ref{thm:sheeting-2nd}, and
    \cite[Theorem 7]{Bellettini25}) then shows that
    $(\bseta_{X_0,\rho_i})_\#\overline V$ is weakly $\theta$-regular near $Y$
    for all sufficiently large $i$.
    This contradicts the fact that
    $(X_i-X_0)/\rho_i\to Y$ and each $X_i$ is weakly $\theta$-singular.
    The {\bf Claim} is thus proved.
    
    Finally, note that $(\bm\eta_{X,\frac{2-|X|}{2}})_\#\overline{V}\in\overline{\mathscr{V}}(\theta,\left( \frac{2}{2-|X|} \right)^n\Lambda)$ for any $X\in B_{2}(0)$, hence the proof of the theorem follows by applying the {\bf Claim} to the above push-forward varifold (up to a different $\tilde\Lambda\coloneqq\left( \frac{2}{2-|X|} \right)^n\Lambda$).
\end{proof}
\begin{theorem}[$\theta$-regularity and compactness]\label{thm:bdry-regularity}
    Let $n\geq2, \theta\in[\frac{\pi}{2},\pi)$, $\Lambda\in[1,\infty)$.
    Suppose $\overline{V}_i \in \mathscr{V}(\theta,\Lambda)$, $i\in\mbN$, and let $M_i, V_i,W_i$ be the corresponding hypersurfaces and varifolds as in Definition \ref{defn:varifold-class}.
    Then, after passing to a subsequence, there exists a stable capillary minimal hypersurface $M$, a varifold $V$ induced by $M$, and a varifold $W$ such that $(V,W)$ is $\mbF_\theta$-stationary in $B_2(0)\cap\overline{\mbH^{n+1}}$ with $\left(\norm{V}-\cos\theta\norm{W}\right)(B_2(0))\leq\Lambda$, and that $V_i,W_i$ converge in the sense of varifolds to $V,W$ respectively.
    Moreover, ${\rm Sing}_\theta V\cap B_2(0)=\emptyset$ if $n< n_\theta$, ${\rm Sing}_\theta V\cap B_2(0)$ is discrete if $n=n_\theta$, and $\mcH^{n-n_\theta+\de}\left({\rm Sing}_\theta V\cap B_2(0)\right)=0$ for any $\de>0$ if $n>n_\theta$, and $M_i$ converges to $M$ smoothly away from ${\rm Sing}_\theta V$.
\end{theorem}

\begin{proof}[Proof of Theorem \ref{thm:bdry-regularity}]
    For each compact subset $K\subset B_2(0)$, we consider cut-off function $\phi_K=1$ on $K$, $=0$ outside $B_2(0)$, with $\abs{D\phi_K}\leq C(K)$ for some positive constant depending on $K$.
    Testing the trace estimate \eqref{eq-divf-capillary} with $\varphi$ therein chosen as $\phi_K$ and $M$ chosen as $M_i$, we see that $\{V_i\}_{i\in\mbN}$, and consequently $\{W_i\}_{i\in\mbN}$ (by Proposition \ref{Prop:bdd-1st-variation-free-bdry-vfld}), have (uniform) locally bounded first variation in $B_2(0)$.
    Applying Allard's integral compactness, we deduce that $V_i$ and $W_i$ subsequentially converge to integral $n$-varifolds $V$ and $W$ in $B_2(0)$, respectively.
    Now, we can define $\overline{V}=V-\cos\theta W$, and we have $\overline{V}_i \rightarrow \overline{V}$.
    Though not needed for this proof, we note that a stronger form of convergence can in fact be established, namely convergence as curvature varifolds with capillary boundary (cf. \cite{WZ25}).
    
    We make the following claim.

    \noindent\textbf{Claim.} ${\rm reg}_\theta \overline{V}\cap B_2(0) \subset {\rm Reg}_\theta V\cap B_2(0)$. 

    Let $X_0\in{\rm reg}_\theta \overline{V}\cap B_2(0)$.
    If the tangent support at $X_0$ is an interior hyperplane, Bellettini's
    interior sheeting theorem \cite[Theorem 6]{Bellettini25} gives
    smooth convergence near $X_0$.
    If it is the supporting wall, the same conclusion follows from Theorem
    \ref{thm:sheeting-2nd}.
    It remains to consider case \textup{(ii)} of Definition
    \ref{defn:theta-regular-point-capillary-vfld}.
    Choose $\rho>0$ so that
    $\overline V\cap B_\rho(X_0)\setminus\partial\mathbb H^{n+1}$ is the stated
    union of embedded capillary minimal hypersurfaces.
    The tangent cone $\boldsymbol C$ at $X_0$ belongs to $\mathscr C_\theta$.
    By Theorem \ref{Thm:sheeting-epsilon-delta}, we can choose
    $\sigma\in(0,\rho)$ such that, for all sufficiently large $i$,
    $B_\sigma^+(X_0)\cap M_i$ can be written as
    \[
        B_\sigma^+(X_0)\cap M_i = \bigcup_{j=1}^{Q_i} \Sigma_{i,j},
    \]
    where $\{\Sigma_{i,j}\}_{j=1}^{Q_i}$ satisfies properties
    \ref{it:reg:onboundary}--\ref{it:defNonemptyIntersection} in Definition
    \ref{Defn:theta-regular-points-manifold}, and
    $B_\sigma^+(X_0)\cap M_i$ converges smoothly to
    $B_\sigma^+(X_0)\cap M$.
    In particular, $B_\sigma^+(X_0)\cap M$ can have the same decomposition as listed in Definition \ref{Defn:theta-regular-points-manifold} (ii), which implies $X_0\in{\rm Reg}_\theta V$.
    Hence, the claim is proved.
    Now, we directly have ${\rm Sing}_\theta V\cap B_2(0)\subset {\rm sing}_\theta \overline{V}\cap B_2(0)$, and the regularity of $V$ follows from Theorem \ref{thm:singular-dimension-estimate-capillary-vflds}.

At last, the smooth convergence is the consequence of the sheeting theorems (Theorem \ref{Thm:sheeting-epsilon-delta}, Theorem
    \ref{thm:sheeting-2nd}, and
    \cite[Theorem 7]{Bellettini25}).
\end{proof}

\section{Bernstein theorem}
\label{sec:BernsteinTheorem}

\begin{theorem}\label{thm:Bernstein-capillary}
    Given $\theta\in(0,\pi)$, let $n_\theta$ be the integer defined in Theorem \ref{thm:mainCompact}.
    Then, for any $2\le n<n_\theta$, if $M$ is a complete, connected, stable capillary minimal hypersurface embedded in $\mbH^{n+1}$ with Euclidean area growth, then $M$ must be flat.
\end{theorem}
\begin{proof}
Let $M\subset \overline{\mbH^{n+1}}$ be as in the statement, and denote $V\coloneqq |M|$.
Reversing the unit normal if necessary and relabeling the resulting angle as
$\theta$, we may assume that $\theta\in[\frac{\pi}{2},\pi)$.
This leaves $M$, its stability, and the value of $n_\theta$ unchanged.
By the relative Jordan--Brouwer separation theorem, $M$ separates
$\overline{\mbH^{n+1}}$ into two connected components. Choose one of
them, denoted by $E_1$, and set
$\Omega_1\coloneqq E_1\cap\partial\mbH^{n+1}$, with the choice of $E_1$
made so that $(V,|\Omega_1|)$ is $\mbF_\theta$-stationary in the sense of
Definition~\ref{defn:stationary-pair}.

By Euclidean area growth, there exists $\Lambda\ge 1$ such that
\eq{
\mcH^n(M\cap B_R)-\cos\theta\,\mcH^n(\Omega_1\cap B_R)\le \Lambda R^n,\quad\forall R>0.
}
Fix any sequence $r_j\to\infty$ as $j\ra\infty$, and define the blow-down sequence $V_j\coloneqq (\bseta_{0,r_j})_\#V$, $W_j\coloneqq (\bseta_{0,r_j})_\#|\Omega_1|$.
Since we have the Euclidean area growth condition, we can apply Theorem~\ref{thm:bdry-regularity} to find varifolds $V_\infty,W_\infty$ such that, after passing to a subsequence, $V_j \to V_\infty$, $W_j\to W_\infty$ as varifolds on $B_2\cap\overline{\mbH^{n+1}}$.
By construction, $V_\infty$ is conical and ${\rm Sing}_\theta V_\infty=\emptyset$.
In particular, the classification of stable capillary cones (see Theorem~\ref{thm:stable-capillary-cone-isolated-singularity}) shows that the associated cone $\mathbf C_\infty\coloneqq V_\infty-\cos\theta\,W_\infty$ belongs to $\mathscr C_\theta$, possibly with multiplicities.
Applying the qualitative sheeting theorem (Theorem~\ref{Thm:sheeting-epsilon-delta}) around $\mathbf C_\infty$, we deduce, after passing to a further subsequence, that the rescaled hypersurfaces $M_j\coloneqq r_j^{-1}M$ converge smoothly, possibly with multiplicity, to the half-hyperplanes in $\spt\|V_\infty\|$ on compact subsets. In particular, for any fixed $x\in M$,
\eq{
\abs{A_{M_j}}\!\left(\frac{x}{r_j}\right)\longrightarrow0\text{ as }j\ra\infty.
}
Using the scaling law of the second fundamental form,
\eq{
\abs{A_{M_j}}\!\left(\frac{x}{r_j}\right)=r_j\,\abs{A_M}(x),
}
we therefore deduce $A_M(x)=0$. By the arbitrariness of $x\in M$, we conclude $M$ is flat.
\end{proof}
This implies the following curvature estimates on Riemannian manifolds.

\begin{corollary}
Given $\theta\in(0,\pi)$, let $n_\theta$ be the integer defined in Theorem \ref{thm:mainCompact}, and let $2\le n<n_\theta$.
Let $(N^{n+1},g)$ be an open, $(n+1)$-dimensional Riemannian manifold with boundary $\p N$,
let $U\subset N$ be an open subset with compact closure, and denote by $\p_{rel}U=\overline{\p U\cap N}$ the relative boundary of $U$ in $N$. 
Let $M$ be a compact, orientable stable capillary minimal hypersurface embedded in $(N^{n+1},g)$ (namely, $\p M\subset\p N$ and $M$ meets $\p N$ along $\p M$ with constant contact angle $\theta$).
If $M\subset U$ with ${\rm dist}_N(M,\p_{rel}U)>0$, and has the area bound $\mcH^n_g(M)\leq\Lambda$ for some $\Lambda>0$, then there exists a constant $C>0$, depending only on $n,(N^{n+1},g),U,\Lambda,\theta$, such that
\eq{
\abs{A}^2(x)
\leq\frac{C}{{\rm dist}^2_N(x,\p_{rel}U)},\quad\forall x\in M.
}
\end{corollary}
\begin{proof}[Sketch of proof]
The proof follows by a straightforward modification of that of \cite[Theorem 1]{GLZ20}.
More precisely, one can argue by contradiction that the curvature estimates fail, then apply a blow-up argument to obtain a non-flat, complete, orientable, stable capillary minimal hypersurface in $\mbH^{n+1}$ (or stable minimal hypersurface without boundary in $\mbR^{n+1}$), satisfying the Euclidean area growth condition. Therefore contradicts either to the classical Bernstein theorem for stable minimal hypersurfaces without boundary, or to Theorem \ref{thm:Bernstein-capillary}.  
\end{proof}

\section{Stable minimal capillary cones}
\label{sec:stableCapillaryCone}

In this section we prove the classification result for stable minimal capillary cones with an isolated singularity (Theorem \ref{thm:stable-capillary-cone-isolated-singularity}).
Note that reversing the unit normal of $M$ replaces $\theta$ by $\pi-\theta$ and $A$ by $-A$, while leaving the stability inequality unchanged. It therefore suffices to assume $\theta\in(0,\frac{\pi}{2}]$ throughout the section.

We begin with the following Simons-type inequality for minimal cones with capillary boundary, (cf. \cite[Equation~(4.4)]{HLW2024deltaStable}):
\eq{\label{ineq:Simons-type}
\abs{A}\De\abs{A}+\abs{A}^4
\geq(s-1)\abs{\na\abs{A}}^2+(3-s)\frac{\abs{A}^2}{r^2},\quad s\leq1+\frac{2}{n-1}.
}
We henceforth fix $s=1+\frac{2}{n-1}$.

For $\alpha\in(0,1]$, multiplying \eqref{ineq:Simons-type} by $\varphi^2\abs{A}^{2\alpha-2}$, integrating over $M$, and integrating by parts, we obtain
\eq{\label{ineq:simonsIntegral}
&\int_{\p M}\varphi^2\abs{A}^{2\alpha-1}\frac{\p\abs{A}}{\p\eta}
+\int_M\abs{A}^{2\alpha+2}\varphi^2
-\left(2\alpha-1+\frac{2}{n-1}\right)\int_M\abs{A}^{2\alpha-2}\abs{\na\abs{A}}^2\varphi^2\\
&-2\int_M\varphi\abs{A}^{2\alpha-1}\langle\na\varphi,\na\abs{A}\rangle
\geq\frac{2(n-2)}{n-1}\int_M\frac{\abs{A}^{2\alpha}}{r^2}\varphi^2.
}

\begin{remark}
\normalfont
The terms involving $\abs{A}^{2\alpha-2}$ may be singular near the zero set of $\abs{A}$. This can be handled by the standard regularization: one replaces $\abs{A}^{2\alpha-2}$ by $(\abs{A}^2+\varepsilon)^{\alpha-1}$ and inserts $(\abs{A}^2+\varepsilon)^{\alpha/2}\varphi$ into the stability inequality, then passes to the limit $\varepsilon\to0$, see \cite[Section~4]{HLW2024deltaStable} for details.
\end{remark}

Multiplying the stability inequality \eqref{ineq:SS81-(1.17)} (with test function $\abs{A}^\alpha\varphi$) by $(q+1)$ for $q>0$, and adding to \eqref{ineq:simonsIntegral}, we obtain
\eq{\label{ineq:combined}
&\int_M(1+q)\abs{A}^{2\alpha}\abs{\na\varphi}^2
-\left(2\alpha-1+\frac{2}{n-1}-\alpha^2(q+1)\right)\int_M\abs{A}^{2\alpha-2}\abs{\na\abs{A}}^2\varphi^2\\
&+2(\alpha(q+1)-1)\int_M\varphi\abs{A}^{2\alpha-1}\left<\na\varphi,\na\abs{A}\right>
-q\int_M\abs{A}^{2\alpha+2}\varphi^2\\
&+\int_{\p M}\varphi^2\abs{A}^{2\alpha-1}\frac{\p\abs{A}}{\p\eta}+(1+q)\cot\theta\int_{\p M}\abs{A}^{2\alpha}A(\eta,\eta)\varphi^2\\
\geq{}&\frac{2(n-2)}{n-1}\int_M\frac{\abs{A}^{2\alpha}}{r^2}\varphi^2.
}

We now analyze the boundary contributions in \eqref{ineq:combined}.
Recall the boundary formula \cite[Lemma C.2]{LZZ24}:
\eq{\label{eq:boundary-formula}
\abs{A}\frac{\p\abs{A}}{\p\eta}
=\cot\theta\left(\sum_{i=1}^n\lambda_i^3-2\abs{A}^2A(\eta,\eta)\right),
}
where $\eta$ is a principal direction of $\p M$, and $\lambda_i$ are the principal curvatures of $M$.

We denote by $\lambda_1=A(\eta,\eta)$ the principal curvature in the $\eta$-direction, and $\lambda_2,\cdots,\lambda_{n}$ the remaining principal curvatures.
Suppose there exists $p_n\geq0$ such that
\eq{\label{ineq:estLam}
\Abs{\lambda_2^3+\cdots+\lambda_n^3+(1-q)(\lambda_2^2+\cdots+\lambda_n^2)(\lambda_2+\cdots+\lambda_n)-q(\lambda_2+\cdots+\lambda_n)^3}\\
\leq p_n\left(\lambda_2^2+\cdots+\lambda_n^2+(\lambda_2+\cdots+\lambda_n)^2\right)^{3/2}.
}
Using \eqref{eq:boundary-formula}, one computes
\eq{
&\abs{A}^{2\alpha-1}\p_\eta\abs{A}+(1+q)\cot\theta\abs{A}^{2\alpha}A(\eta,\eta)\\
=&\cot\theta\left(\abs{A}^{2\alpha-2}\left(\sum_{i=1}^n\lambda_i^3-2\abs{A}^2A(\eta,\eta)\right)+(1+q)\abs{A}^{2\alpha}A(\eta,\eta)\right)\\
=&\cot\theta\left(\left(\abs{A}^{2\alpha-2}\sum^n_{i=1}\lambda^3_i\right)-(1-q)\abs{A}^{2\alpha}A(\eta,\eta)\right)\\
=&\cot\theta\abs{A}^{2\alpha-2}\left(-(1-q)\lambda_1(\lambda_1^2+\cdots+\lambda_n^2)+\sum^n_{i=1}\lambda_i^3\right),
}
where $\lambda_1=-(\lambda_2+\cdots+\lambda_n)$ thanks to the minimality.
Note also
\eq{
(1-q)(\lambda_2+\cdots+\lambda_n)\lambda_1^2+\lambda_1^3
=q\lambda_1^3
=-q(\lambda_2+\cdots+\lambda_n)^3.
}
Thus we can write
\eq{
&\abs{A}^{2\alpha-1}\p_\eta\abs{A}+(1+q)\cot\theta\abs{A}^{2\alpha}A(\eta,\eta)\\
=&\cot\theta\abs{A}^{2\alpha-2}\left((1-q)(\lambda_2+\cdots+\lambda_n)(\lambda_1^2+\cdots+\lambda_n^2)+\sum^n_{i=1}\lambda_i^3\right)\\
=&\cot\theta\abs{A}^{2\alpha-2}\left((1-q)(\lambda_2+\cdots+\lambda_n)(\lambda_2^2+\cdots+\lambda_n^2)+\sum^n_{i=2}\lambda_i^3-q(\lambda_2+\cdots+\lambda_n)^3\right).
}
Combining this with
\eqref{ineq:estLam}, and using $\theta\in(0,\frac{\pi}{2}]$, we can estimate the boundary contribution in \eqref{ineq:combined} as
\eq{
\int_{\p M}\varphi^2\abs{A}^{2\alpha-1}\frac{\p\abs{A}}{\p\eta}+(1+q)\cot\theta\int_{\p M}\abs{A}^{2\alpha}A(\eta,\eta)\varphi^2
\leq p_n\cot\theta\int_{\p M}\abs{A}^{2\alpha+1}\varphi^2.
}
To control this boundary term, we apply \eqref{eq-divf-capillary},
together with the Cauchy-Schwarz inequality, to obtain, for any $\delta\in(0,1)$:
\eq{
p_n\cot\theta\int_{\p M}\abs{A}^{2\alpha+1}\varphi^2
\leq\,&\frac{p_n\cos\theta}{\sin^2\theta}\int_M\left[(2\alpha+1)\abs{A}^{2\alpha}\abs{\na\abs{A}}\varphi^2+2\abs{A}^{2\alpha+1}\varphi\abs{\na\varphi}\right]\\
\leq\,&\left(2\alpha-1+\frac{2}{n-1}-\alpha^2(q+1)\right)\int_M\abs{A}^{2\alpha-2}\abs{\na\abs{A}}^2\varphi^2\\
&+\frac{(2\alpha+1)^2p_n^2\cos^2\theta}{4\left(2\alpha-1+\frac{2}{n-1}-\alpha^2(q+1)\right)\sin^4\theta}\int_M\abs{A}^{2\alpha+2}\varphi^2\\
&+q\delta\int_M\abs{A}^{2\alpha+2}\varphi^2+\frac{p_n^2\cos^2\theta}{q\delta\sin^4\theta}\int_M\abs{A}^{2\alpha}\abs{\na\varphi}^2,
}
provided when $p_n>0$ that
\eq{\label{condi:p_n>0-Young-ineq}
2\alpha-1+\frac{2}{n-1}-\alpha^2(q+1)>0.
}
Substituting back into \eqref{ineq:combined} and absorbing the gradient terms, we arrive at
\eq{\label{ineq:finalInequalitywithTheta}
\int_M\left(1+q+\frac{p_n^2\cos^2\theta}{q\delta\sin^4\theta}\right)\abs{A}^{2\alpha}\abs{\na\varphi}^2
+2(\alpha(q+1)-1)\int_M\varphi\abs{A}^{2\alpha-1}\langle\na\varphi,\na\abs{A}\rangle\\
\geq\frac{2(n-2)}{n-1}\int_M\frac{\abs{A}^{2\alpha}}{r^2}\varphi^2
+\left(q-\frac{(2\alpha+1)^2p_n^2\cos^2\theta}{4\left(2\alpha-1+\frac{2}{n-1}-\alpha^2(q+1)\right)\sin^4\theta}-q\delta\right)\int_M\abs{A}^{2\alpha+2}\varphi^2.
}
In order for the last term on the right-hand side to be non-negative, it suffices to require
\eq{\label{ineq:rangeThetaAlpha}
\frac{(2\alpha+1)^2p_n^2}{4(1-\delta)q\left(2\alpha-1+\frac{2}{n-1}-\alpha^2(q+1)\right)}\cdot\frac{\cos^2\theta}{\sin^4\theta}
\leq1
}
whenever $p_n>0$.

Under condition \eqref{ineq:rangeThetaAlpha}, inequality \eqref{ineq:finalInequalitywithTheta} simplifies to
\eq{\label{ineq:finalSimplified}
\int_M\left(1+q+\frac{4(1-\delta)}{(2\alpha+1)^2\delta}\left(2\alpha-1+\frac{2}{n-1}-\alpha^2(q+1)\right)\right)\abs{A}^{2\alpha}\abs{\na\varphi}^2\\
+2(\alpha(q+1)-1)\int_M\varphi\abs{A}^{2\alpha-1}\langle\na\varphi,\na\abs{A}\rangle
\geq\frac{2(n-2)}{n-1}\int_M\frac{\abs{A}^{2\alpha}}{r^2}\varphi^2.
}

\subsection{Proof of Theorem \ref{thm:stable-capillary-cone-isolated-singularity}: case \texorpdfstring{$n=3$}{n=3}}

For $n=3$, note that at every $X\in\partial M\setminus\{0\}$, the radial direction $\partial_r$ belongs to $T_X\p M$, which is orthogonal to $\eta$, and satisfies $D_{\partial_r}\nu=0$. Hence $\partial_r$ is a principal
direction with zero principal curvature. After relabeling, we may assume $\lambda_3=0$. Therefore, back to \eqref{ineq:estLam} we can take $q=1, p_3=0$, and set $\alpha=1$.
Then \eqref{ineq:finalInequalitywithTheta} gives
\eq{
\int_M 2\abs{A}^2\abs{\na\varphi}^2+2\int_M\varphi\abs{A}\langle\na\varphi,\na\abs{A}\rangle
\geq\int_M\frac{\abs{A}^2}{r^2}\varphi^2.
}
We choose $\varphi=r^{1+\varepsilon}\max\{1,r\}^{-\frac{1}{2}-2\varepsilon}$, and note that $\langle\na\abs{A},\frac{\p}{\p r}\rangle=-\frac{\abs{A}}r$ for a homogeneous cone.
A direct computation then yields, for sufficiently small $\varepsilon>0$,
\begin{align*}
&\int_{ \{ r>1 \}} r^{-1-2\varepsilon}|A|^{2}+\int_{ \{ r< 1 \}} r^{2\varepsilon}|A|^{2}\\
\le{}& 
    \left[ 2\left( \frac{1}{2}-\varepsilon \right)^{2}-2\left( \frac{1}{2}-\varepsilon \right) \right] \int_{ \{ r>1 \}} r^{-1-2\varepsilon}|A|^{2}+\left[ 2(1+\varepsilon)^{2}-2(1+\varepsilon) \right] \int_{ \{ r< 1 \}} r^{2\varepsilon}|A|^{2}.
\end{align*}
For every sufficiently small $\varepsilon>0$, the coeffcients on the RHS are strictly less than $1$, thus we get $\abs{A}\equiv0$, and hence $M$ is a half-hyperplane.

\subsection{Proof of Theorem \ref{thm:stable-capillary-cone-isolated-singularity}: case \texorpdfstring{$n=4, \cos\theta>\frac{1}{3}$}{n=4}}

Let $\Sigma=M\cap\mbS^4$ denote the link of the cone $M$, and let $\Gamma=\p M\setminus\{0\}$. Assume that $\cos\theta>\frac13$.

Applying the stability inequality to test functions of the form $\varphi(r,\omega)=\zeta(r)\psi(\omega)$, where $\psi\in H_0^1(\Sigma)$, then using the sharp radial Hardy constant in
dimension four, we obtain
\eq{\label{eq:n=4-link-stability}
\int_\Sigma\left(\abs{\na^\Sigma\psi}^2
+(1-\abs{A_\Sigma}^2)\psi^2\right)\rd\mcH^3\geq0.
}
The height function $h(X)=\langle X,e_1\rangle$ on $\Sigma$ satisfies
\eq{
-\De_\Sigma h=3h,
\quad h>0\text{ in }{\rm int}(\Sigma),
\quad h=0\text{ on }\p\Sigma.
}
Namely, $h$ is a first Dirichlet eigenfunction, so that
\eq{\label{eq:n=4-link-Poincare}
3\int_\Sigma\psi^2\rd\mcH^3
\leq\int_\Sigma\abs{\na^\Sigma\psi}^2\rd\mcH^3
}
for every $\psi\in H_0^1(\Sigma)$.

On the other hand, note that the zero-homogeneous angle function $v=\langle\nu,e_1\rangle$ satisfies
\eq{\label{eq:angle-v=<nu,e_1>}
\De_\Sigma v+\abs{A_\Sigma}^2v=0\quad\text{in }\Sigma,
\quad v=\cos\theta\quad\text{on }\p\Sigma.
}
Set $w=(\cos\theta-v)_+\in H_0^1(\Sigma)$.
Testing \eqref{eq:angle-v=<nu,e_1>} with
$w$, and then using $w$ in \eqref{eq:n=4-link-stability}, gives
\eq{\label{eq:n=4-angle-estimates}
\int_\Sigma\abs{\na^\Sigma w}^2\rd\mcH^3
=\int_\Sigma\abs{A_\Sigma}^2(w^2-\cos\theta w)\rd\mcH^3,
\quad
\cos\theta\int_\Sigma\abs{A_\Sigma}^2w\rd\mcH^3
\leq\int_\Sigma w^2\rd\mcH^3.
}
Since $v\geq-1$, we have $w(w-\cos\theta)\leq w$ on $\{w>0\}$.
Combining with \eqref{eq:n=4-link-Poincare}, \eqref{eq:n=4-angle-estimates},
\eq{
3\int_\Sigma w^2\rd\mcH^3
\leq\int_\Sigma\abs{\na^\Sigma w}^2\rd\mcH^3
\leq\frac1{\cos\theta}\int_\Sigma w^2\rd\mcH^3.
}
Thus $w\equiv0$, and hence $v\geq \cos\theta$ on $\Sigma$.

Let $H_\Gamma$ be the mean curvature of $\Gamma\subset\p\mbH^5$ with
respect to $\bar\nu$ in \eqref{eq:nu_and_eta}. By the capillary boundary condition and minimality, we have
\eq{
\p_\eta v
=\sin\theta\,A(\eta,\eta),\quad
A(\eta,\eta)
=-\sin\theta\,H_\Gamma\quad\text{ along }\Gamma.
}
Since $v-\cos\theta$ is nonnegative and vanishes on the boundary, its outer
conormal derivative is nonpositive.  Therefore we conclude that
\eq{\label{eq:n=4-positive-boundary-mean-curvature}
\cos\theta>\frac13\quad\Longrightarrow\quad H_\Gamma\geq0.
}
By a recent result of 
Pacati-Tortone-Velichkov \cite[Proposition~2.4, Lemmas~4.1--4.3, and Section~4.3, Case~2]{PTV25}, a $4$-dimensional stable capillary minimal cone in $\mbR^5$ with an isolated singularity and nonnegative boundary mean curvature is flat \footnote{The results in \cite{PTV25} are stated for capillary minimizing cones, but we point out that these results apply to our case since their argument uses only the Simons, Codazzi, and capillary boundary identities, cone homogeneity, and also stability inequality.
}. Consequently, $A\equiv0$ whenever $\cos\theta>\frac13$.

\subsection{Proof of Theorem \ref{thm:stable-capillary-cone-isolated-singularity}: cases \texorpdfstring{$n=4,5,6$}{n=4,5,6}}

For $n=4,5,6$, we choose the test function $\varphi=r^\varepsilon\max\{1,r\}^{1+\alpha-n/2-2\varepsilon}$.
Substituting into \eqref{ineq:finalSimplified}, we have
\begin{align*}
    &\frac{2n-4}{n-1} \left( \int_{ \{ r>1 \}} |A|^{2\alpha}r^{2\alpha-n-2\varepsilon} + \int_{ \{ r< 1 \}} |A|^{2\alpha}r^{2\varepsilon-2} \right)\\
    \le{}&\left(1+q+\frac{4(1-\delta)}{(2\alpha+1)^2\delta} \left(2\alpha-1+\frac{2}{n-1}-\alpha^{2}(q+1)\right)\right)\\
    &\qquad\times\left[ \left( 1+\alpha-\frac{n}{2}-\varepsilon \right)^{2}\int_{ \{ r\ge 1 \}} |A|^{2\alpha}r^{2\alpha-n-2\varepsilon} + \varepsilon^{2}\int_{ \{ r< 1 \}} |A|^{2\alpha}r^{2\varepsilon-2} \right]\\
    &+ 2(\alpha(q+1)-1)\left[ -\left(1+\alpha-\frac{n}{2}-\varepsilon\right)\int_{ \{ r\ge 1 \}} |A|^{2\alpha}r^{2\alpha-n-2\varepsilon} - \varepsilon\int_{ \{ r< 1 \}} |A|^{2\alpha}r^{2\varepsilon-2} \right].
\end{align*}
To obtain $|A|\equiv 0$, we need the above inequality to hold for sufficiently small $\varepsilon>0$, which can be ensured if
\begin{align}
    &\left[1+q+\frac{4(1-\delta)}{(2\alpha+1)^2\delta} \left(2\alpha-1+\frac{2}{n-1}-\alpha^{2}(q+1)\right)\right]\left( 1+\alpha-\frac{n}{2} \right)^{2}\\
    &-2(\alpha(q+1)-1)\left(1+\alpha-\frac{n}{2}\right)< \frac{2n-4}{n-1}.
    \label{eq:constrain}
\end{align}

To determine a valid range of $\theta$, one seeks to maximize
\eq{\label{defn:M-functional}
\mathfrak{M}(n,\alpha,\delta,q,p_n(q))
\coloneqq\frac{4(1-\delta)q}{(2\alpha+1)^2p_n^2}\left(2\alpha-1+\frac{2}{n-1}-\alpha^2(q+1)\right)
}
over parameters $\alpha\in(0,1]$, $\delta\in(0,1)$, $q\in(0,1]$ subject to \eqref{eq:constrain} and $2\alpha-1+\frac{2}{n-1}-\alpha^2(q+1)>0$.
In view of\eqref{ineq:rangeThetaAlpha}, $\theta$ should satisfy
\eq{
\frac{\cos^2\theta}{\sin^4\theta}
\leq\mathfrak{M}(n,\alpha,\delta,q,p_n).
}

It is difficult to determine the exact maximum of $\mathfrak{M}(n,\alpha,\delta,q,p_n)$ analytically. Instead, we obtain an explicit lower bound by selecting suitable parameters $\alpha,\delta,q$.
We also need to determine the constant $p_n$ appearing in \eqref{ineq:estLam}.

As discussed above, $\partial_r$ is a principal
direction with zero principal curvature. After relabeling, we may assume $\lambda_n=0$.
Put $\tilde n\coloneqq n-2$ and $x_i\coloneqq\lambda_{i+1}$ for $i=1,\ldots,\tilde n$.
For $q>0$, we define the function
\eq{
f_{\tilde n,q}(x_1,\cdots,x_{\tilde n})
\coloneqq\frac{P_3+(1-q)P_1P_2-qP_1^3}{(P_2+P_1^2)^{3/2}},
}
where $P_k=x_1^k+\cdots+x_{\tilde n}^k$, for $(x_1,\cdots,x_{\tilde n})\in\mbR^{\tilde n}\setminus\{0\}$.
Thus $\tilde n$ should take values in $\{2,3,4\}$, which corresponds to the case $n=4,5,6$.
With this function, one can then choose $p_n=\sup\abs{f_{\tilde n,q}}$.

\begin{lemma}\label{lem:criticalPoints}
The critical points of $f_{\tilde n,q}$ take at most two distinct values among $x_1,\cdots,x_{\tilde n}$.
\end{lemma}
\begin{proof}
    Since $f_{\tilde n,q}$ is $0$-homogeneous,
    we can restrict to the unit sphere $P_2=1$.
    Then,
\eq{
f_{\tilde n,q}(x_1,\cdots,x_{\tilde n})
=\frac{P_3+(1-q)P_1-qP_1^3}{(1+P_1^2)^{3/2}}.
}
By the Lagrange multiplier condition, any critical point satisfies $\frac{\p f_{\tilde n,q}}{\p x_i}+2\lambda x_i=0$ for all $i$.
A direct computation gives
\eq{
\frac{\p f_{\tilde n,q}}{\p x_i}-\frac{\p f_{\tilde n,q}}{\p x_j}
=\frac{3(x_i^2-x_j^2)}{(1+P_1^2)^{3/2}}+(x_i-x_j)\cdot(\text{terms involving }P_1, P_2, P_3,\text{ independent of }i,j),
}
so that for any pair $i\neq j$:
\eq{
0
=&\frac{\p f_{\tilde n,q}}{\p x_i}-\frac{\p f_{\tilde n,q}}{\p x_j}+2\lambda(x_i-x_j)\\
=&(x_i-x_j)\left(\frac{3(x_i+x_j)}{(1+P_1^2)^{3/2}}+2\lambda+(\text{terms involving }P_1, P_2, P_3,\text{ independent of }i,j)\right).
}
If there were three distinct values among $x_1,\cdots,x_{\tilde n}$, say $x_1,x_2,x_3$, then we have
\eq{
    \frac{3(x_i+x_j)}{(1+P_1^2)^{\frac{3}{2}}}+2\lambda+(\text{terms involving }P_1, P_2, P_3,\text{ independent of }i,j)=0,
}
for any $i\neq j\in \{ 1,2,3 \}$. In particular, this implies
\[
    \frac{3x_1+3x_2}{(1+P_1^2)^{\frac{3}{2}}}=\frac{3x_1+3x_3}{(1+P_1^2)^{\frac{3}{2}}},
\]
which leads to $x_2=x_3$, contradicting the assumption.
\end{proof}

\begin{lemma}\label{lem:explicit-p_n}
With the choices $q_4=1$, $q_5=\frac{6}{11}$, $q_6=\frac{43}{391}$, the following bounds hold:
\eq{
\abs{f_{2,1}}\leq\frac{1}{\sqrt6},\quad
\abs{f_{3,\frac{6}{11}}}\leq\frac{65}{11\sqrt{66}},\quad
\abs{f_{4,\frac{43}{391}}}\leq\frac{25423}{782\sqrt{1173}}.
}
\end{lemma}
\begin{proof}
By Lemma~\ref{lem:criticalPoints}, it suffices to evaluate $f_{\tilde n,q}$ at points with at most two distinct values.
For integers $a,b\geq1$ with $a+b=\tilde n$, define
\eq{
f_{a,b,q}(x,y)
\coloneqq\frac{ax^3+by^3+(1-q)(ax+by)(ax^2+by^2)-q(ax+by)^3}{(ax^2+by^2+(ax+by)^2)^{3/2}}.
}
Computing the derivative of $f_{a,b,q}$ we find, its critical points (except for the case $x=y$) satisfy
\eq{
(2+q)a(a+1)x^2+(3(a+b+1)+2(2+q)ab)xy+(2+q)b(b+1)y^2=0.
}
Fixing $y=1$, then the roots of the above quadratic equation are $x_{a,b,q,1}\leq x_{a,b,q,2}$,
and can be computed explicitly.
Since $f_{\tilde n,q}$ is $0$-homogeneous, we thus have
\eq{
\max\abs{f_{\tilde n,q}}
=\max\left\{\abs{f_{a,b,q}(x_{a,b,q,1},1)},\abs{f_{a,b,q}(x_{a,b,q,2},1)},\abs{f_{a,b,q}(1,1)}, \abs{f_{a,b,q}(1,0)}: a,b\geq1,\,a+b=\tilde n\right\}.
}
A numerical evaluation for $(\tilde n,q)\in\{(2,1),(3,\tfrac6{11}),(4,\tfrac{43}{391})\}$ confirms that the maximum is achieved at $f_{\tilde n,q}(1,s_{\tilde n},\cdots,s_{\tilde n})$, with $s_2=-\tfrac12$, $s_3=-\tfrac27$, $s_4=-\tfrac{11}{60}$, yielding the stated bounds.
\end{proof}

We now specify the parameters for $n=4,5,6$.
Set
\begin{align*}
\alpha_4={}&\frac{14}{33}, \delta_4=\frac{1}{15}, q_4=1,p_4=\frac{1}{\sqrt{6}},\\
\alpha_5={}&\frac{7}{12}, \delta_5=\frac{4}{19}, q_5=\frac{6}{11},p_5=\frac{65}{11\sqrt{66}},\\
\alpha_6={}&\frac{6}{11}, \delta_6=\frac{16}{25}, q_6=\frac{43}{391},p_6=\frac{25423}{782 \sqrt{1173}}.
\end{align*}
One verifies directly that each triple $(\alpha_n,\delta_n,q_n)$ satisfies the constraints \eqref{condi:p_n>0-Young-ineq}, \eqref{eq:constrain}.
Substituting into \eqref{ineq:rangeThetaAlpha}, the condition on $\theta$ becomes:
\eq{
\frac{\cos^2\theta}{\sin^4\theta}
<\begin{cases}
\dfrac{18928}{18605}, & n=4,\\[6pt]
\dfrac{264924}{2713295}, & n=5,\\[6pt]
\dfrac{12002306544}{1858195670875}, & n=6.
\end{cases}
}
From the numerical solutions of these inequalities we deduce, the following ranges for $\theta$ will guarantee the above inequality holds:
\eq{
\theta\in
\begin{cases}
(51.654^\circ,128.346^\circ), & n=4,\\
(73.336^\circ,106.664^\circ), & n=5,\\
(85.420^\circ,94.580^\circ), & n=6.
\end{cases}
}
For $n=4$, the interval
$[\arccos(1/3),\frac{\pi}{2}]$ is contained in
$(51.654^\circ,128.346^\circ)$.  Hence this calculation covers
$\cos\theta\leq\frac13$, while the preceding subsection covers
$\cos\theta>\frac13$.  The orientation reduction proves the
four-dimensional statement for every $\theta\in(0,\pi)$.
This completes the proof of Theorem~\ref{thm:stable-capillary-cone-isolated-singularity}.

\begin{remark}
\normalfont
The Mathematica code for the numerical verifications in this section is available at \url{https://github.com/wgaom/stable-capillary-cone-verification}.
\end{remark}

\begin{proof}[Proof of Corollary~\ref{cor:minimizing-singular-set}]
As discussed in \cite{CEL25}, by the $\varepsilon$-regularity and compactness theorem due to De Philippis-Maggi \cite{DePM15}, the monotonicity formula of Kagaya-Tonegawa \cite{KT17}, and Federer's dimension reduction, it suffices to rule out non-flat minimizing capillary cones with an isolated singularity in dimensions at most four. Note that such cones are of course stable and hence flat by Theorem \ref{thm:stable-capillary-cone-isolated-singularity}, which completes the proof.
\end{proof}

\appendix

\section{Miscellaneous computations}
\begin{lemma}
Let $n\geq2, \theta\in(0,\pi)$, let $g_\theta$ be defined as in \eqref{defn:g_theta}, there holds
\eq{\label{ineq:nabla-g_theta-estimate}
\abs{\na g_\theta}
\leq\abs{A}.
}
\end{lemma}
\begin{proof}
Note that for any $\tau\in T_xM$, we have
\eq{
g_{\theta,k}\tau g_{\theta,k}
=&\sum_{i=2}^n\nu_i\tau\nu_i+k(\nu_1-\cos\theta)\tau\nu_1\\
=&\left<(\nu_2,\ldots,\nu_n,\sqrt{k}(\nu_1-\cos\theta)),(\tau\nu_2,\ldots,\tau\nu_n,\sqrt{k}\tau\nu_1)\right>.
}
By Cauchy-Schwarz,
\eq{
g^2_{\theta,k}\abs{\tau g_{\theta,k}}^2
\leq\underbrace{\left(\sum_{i=2}^n\nu_i^2+k(\nu_1-\cos\theta)^2\right)}_{=g^2_{\theta,k}}\left(\sum_{i=2}^n(\tau\nu_i)^2+k(\tau\nu_1)^2\right).
}
Thus whenever $k\in(0,1]$, we have
\eq{
\abs{\tau g_{\theta,k}}^2
\leq\sum_{i=2}^n(\tau\nu_i)^2+k(\tau\nu_1)^2
\leq\sum_{i=1}^n(\tau\nu_i)^2
\leq\abs{D_\tau\nu}^2,
}
which implies \eqref{ineq:nabla-g_theta-estimate}.

\end{proof}

{
\begin{lemma}\label{Lem:Du+cot-e_1<=g_theta}
Let $n\geq2, \theta\in(0,\pi)$, let $g_\theta$ be defined as in
\eqref{defn:g_theta}, and let $\nu_\theta, \nu_{-\theta}$ be defined as
in \eqref{defn:nu_theta}. Let $u$ be a function defined on
$\mbR^n_+=\{x_1>0\}$.

Suppose that $u$ is locally $C^2$ around a point $x_0\in\mbR^n_+$.
There exists a positive constant $C=C(n,\theta)$ with the following
property:
\begin{enumerate}
    \item If the graph of $u$ is oriented by the upwards pointing unit
    normal locally around $x_0$, and
    $g^2_\theta\leq\eta^2\leq\mfc_\theta$ at $(x_0,u(x_0))$, where
\eq{\label{defn:mfc}
    \mfc_\theta
    \coloneqq
    \begin{cases}
        \mfc_{\theta,1},\quad&\text{ when }n=2,\\
        \mfc_{\theta,\frac{1}{n-2}},\quad&\text{ when }n\geq3,
    \end{cases}
}
and
\eq{
    \mfc_{\theta,k}
    \coloneqq
    \min\left\{
        \frac{k\sin^4\theta}{64},
        \frac{k}{k+1+16\sin^{-2}\theta}
    \right\},
    \quad\forall k\in(0,1],
}
then
\eq{
    \abs{Du(x_0)+\cot\theta e_1}^2
    \leq C\abs{\nu_{\mid_{(x_0,u(x_0))}}-\nu_\theta}^2
    \leq C(g_\theta)^2_{\mid_{(x_0,u(x_0))}}
    \leq C\eta^2.
}
Moreover,
\eq{
    \langle\nu,\nu_\theta\rangle
    \geq\frac{1}{2}.
}

    \item If the graph of $u$ is oriented by the downwards pointing unit
    normal locally around $x_0$, and
    $g^2_\theta\leq\eta^2\leq\mfc_\theta$ at $(x_0,u(x_0))$, where
    $\mfc_\theta$ is given by \eqref{defn:mfc}, then
\eq{
    \abs{Du(x_0)-\cot\theta e_1}^2
    \leq C\abs{\nu_{\mid_{(x_0,u(x_0))}}-\nu_{-\theta}}^2
    \leq C(g_\theta)^2_{\mid_{(x_0,u(x_0))}}
    \leq C\eta^2.
}
Moreover,
\eq{
    \langle\nu,\nu_{-\theta}\rangle
    \geq\frac{1}{2}.
}
\end{enumerate}
In particular, in both cases we have as a by-product
\eq{\label{defn:mfC_theta}
    \abs{Du(x_0)}^2
    \leq\frac{4}{\sin^2\theta}-1
    \eqqcolon\mfC_\theta.
}
\end{lemma}

\begin{proof}
The following computations are carried out at $x_0$, or equivalently at
$(x_0,u(x_0))$. For simplicity, we omit the argument.

We prove the estimates for $g_{\theta,k}$ defined in
\eqref{defn:g_theta,k}. The assertion then follows by taking $k=1$ when
$n=2$, and $k=\frac{1}{n-2}$ when $n\geq3$. In particular,
$k\in(0,1]$.

{\em (1)}.
Write
\eq{
\nu
=(\nu_1,\cdots,\nu_{n+1})
=\left(
\frac{-u_1}{\sqrt{1+\abs{Du}^2}},
\cdots,
\frac{-u_n}{\sqrt{1+\abs{Du}^2}},
\frac{1}{\sqrt{1+\abs{Du}^2}}
\right).
}
Then
\eq{\label{eq:nu-nu_theta-length-square}
    \abs{\nu-\nu_\theta}^2
    =
    (\nu_1-\cos\theta)^2
    +\sum_{i=2}^n\nu_i^2
    +(\nu_{n+1}-\sin\theta)^2.
}

We first bound $\abs{\nu-\nu_\theta}^2$ in terms of
$g_{\theta,k}^2$. Since $g^2_{\theta,k}=\sum_{i=2}^n\nu_i^2+k(\nu_1-\cos\theta)^2$, we have
\eq{\label{ineq:appendix-estiamte-in-terms-of-g^2}
    \sum_{i=2}^n\nu_i^2
    \leq g^2_{\theta,k},
    \quad
    \abs{\nu_1-\cos\theta}^2
    \leq\frac{1}{k}g^2_{\theta,k}.
}

By direct computation,
\eq{
    \nu^2_{n+1}-\sin^2\theta
    =
    -(\nu_1-\cos\theta)(\nu_1+\cos\theta)
    -\sum_{i=2}^n\nu_i^2.
}
Hence, by \eqref{ineq:appendix-estiamte-in-terms-of-g^2},
\eq{
    \abs{\nu_{n+1}^2-\sin^2\theta}
    \leq
    \abs{\nu_1-\cos\theta}
    \left(\abs{\nu_1}+\abs{\cos\theta}\right)
    +\sum_{i=2}^n\nu_i^2
    \leq
    \frac{2}{\sqrt{k}}g_{\theta,k}
    +g_{\theta,k}^2.
}
By $g_{\theta,k}^2\leq\eta^2\leq\mfc_{\theta,k}\leq\frac{k\sin^4\theta}{64}\leq1$,
we have $g_{\theta,k}\leq1$. Since $k\leq1$, we therefore
obtain
\eq{
    \abs{\nu_{n+1}^2-\sin^2\theta}
    \leq\frac{4}{\sqrt{k}}g_{\theta,k}.
}
Since $\nu_{n+1}>0$ in this case, it follows that
\eq{\label{ineq:nu_n+1-sin-theta}
    \abs{\nu_{n+1}-\sin\theta}
    =
    \frac{\abs{\nu_{n+1}^2-\sin^2\theta}}
    {\nu_{n+1}+\sin\theta}
    \leq
    \frac{4}{\sqrt{k}\sin\theta}g_{\theta,k}.
}

Combining \eqref{ineq:appendix-estiamte-in-terms-of-g^2} and
\eqref{ineq:nu_n+1-sin-theta}, we obtain
\eq{\label{ineq:normal-g-estimate-upward}
    \abs{\nu-\nu_\theta}^2
    \leq
    \left(
        \frac{1}{k}
        +1
        +\frac{16}{k\sin^2\theta}
    \right)g^2_{\theta,k}.
}
Using $g_{\theta,k}^2\leq\eta^2\leq\mfc_{\theta,k}\leq\frac{k}{k+1+16\sin^{-2}\theta}$,
we thus deduce
\eq{
    \abs{\nu-\nu_\theta}^2
    \leq
    \left(
        \frac{1}{k}
        +1
        +\frac{16}{k\sin^2\theta}
    \right)
    \frac{k}{k+1+16\sin^{-2}\theta}
    =1.
}
Consequently,
\eq{
    2-2\langle\nu,\nu_\theta\rangle
    =
    \abs{\nu-\nu_\theta}^2
    \leq1,
}
and hence $\langle\nu,\nu_\theta\rangle\geq\frac{1}{2}$.

We next obtain a lower bound for $\nu_{n+1}$. Since $g_{\theta,k}^2\leq\eta^2\leq\mfc_{\theta,k}\leq\frac{k\sin^4\theta}{64}$,
inequality \eqref{ineq:nu_n+1-sin-theta} gives
\eq{
    \abs{\nu_{n+1}-\sin\theta}
    \leq
    \frac{4}{\sqrt{k}\sin\theta}
    \sqrt{\frac{k\sin^4\theta}{64}}
    =
    \frac{\sin\theta}{2}.
}
Therefore,
\eq{\label{ineq:nu-n+1-lower-bound}
    \nu_{n+1}
    \geq\frac{\sin\theta}{2}.
}
It follows that
\eq{
    \abs{Du}^2
    =
    \frac{1-\nu_{n+1}^2}{\nu_{n+1}^2}
    \leq
    \frac{4}{\sin^2\theta}-1.
}

It remains to bound
$\abs{Du+\cot\theta e_1}^2$ in terms of
$\abs{\nu-\nu_\theta}^2$. We compute
\eq{
    Du+\cot\theta e_1
    =
    \left(
        -\frac{\nu_1}{\nu_{n+1}}+\cot\theta,
        -\frac{\nu_2}{\nu_{n+1}},
        \cdots,
        -\frac{\nu_n}{\nu_{n+1}}
    \right),
}
and hence
\eq{\label{ineq:Lem-Du-cote_1-ineq2}
    \abs{Du+\cot\theta e_1}^2
    =
    \frac{1}{\nu_{n+1}^2}
    \left(
        \left(\nu_1-\cot\theta\nu_{n+1}\right)^2
        +\sum_{i=2}^n\nu_i^2
    \right).
}
Moreover,
\eq{\label{ineq:Lem-Du-cote_1-ineq3}
    \left(\nu_1-\cot\theta\nu_{n+1}\right)^2
    =
    \left(
        \nu_1-\cos\theta
        +\cos\theta\frac{\sin\theta-\nu_{n+1}}{\sin\theta}
    \right)^2
    \leq
    2(\nu_1-\cos\theta)^2
    +2\cot^2\theta(\sin\theta-\nu_{n+1})^2.
}
Also,
\eq{\label{ineq:Lem-Du-cote_1-ineq1}
    \max\left\{
        (\nu_1-\cos\theta)^2,
        \sum_{i=2}^n\nu_i^2,
        (\nu_{n+1}-\sin\theta)^2
    \right\}
    \leq
    \abs{\nu-\nu_\theta}^2.
}
Combining the above estimates with
\eqref{ineq:nu-n+1-lower-bound}, we arrive at
\eq{\label{ineq:Du+cot-e_1<=nu-nu_theta}
    \abs{Du+\cot\theta e_1}^2
    \leq
    \frac{4}{\sin^2\theta}
    \left(3+2\cot^2\theta\right)
    \abs{\nu-\nu_\theta}^2.
}
Together with \eqref{ineq:normal-g-estimate-upward}, this proves
assertion \textup{(1)}.

{\em (2)}.
In this case,
\[
\nu=(\nu_1,\cdots,\nu_{n+1})
=
\left(
\frac{u_1}{\sqrt{1+\abs{Du}^2}},
\cdots,
\frac{u_n}{\sqrt{1+\abs{Du}^2}},
\frac{-1}{\sqrt{1+\abs{Du}^2}}
\right).
\]
In particular, $\nu_{n+1}<0$, and
\eq{\label{eq:nu-nu_-theta-length-square}
    \abs{\nu-\nu_{-\theta}}^2
    =
    (\nu_1-\cos\theta)^2
    +\sum_{i=2}^n\nu_i^2
    +(\nu_{n+1}+\sin\theta)^2.
}
The estimates in
\eqref{ineq:appendix-estiamte-in-terms-of-g^2} still hold. As arguing above, we obtain
\eq{
    \abs{\nu_{n+1}+\sin\theta}
    =
    \frac{\abs{\nu_{n+1}^2-\sin^2\theta}}
    {\sin\theta-\nu_{n+1}}
    \leq
    \frac{4}{\sqrt{k}\sin\theta}g_{\theta,k}.
}
Consequently,
\eq{\label{ineq:normal-g-estimate-downward}
    \abs{\nu-\nu_{-\theta}}^2
    \leq
    \left(
        \frac{1}{k}
        +1
        +\frac{16}{k\sin^2\theta}
    \right)g^2_{\theta,k}.
}
The choice $\eta^2\leq\mfc_{\theta,k}\leq\frac{k}{k+1+16\sin^{-2}\theta}$ then gives $\abs{\nu-\nu_{-\theta}}^2\leq1$, and hence $\langle\nu,\nu_{-\theta}\rangle\geq\frac{1}{2}$.

Similarly, since $\eta^2\leq\mfc_{\theta,k}
\leq\frac{k\sin^4\theta}{64}$, we obtain
\eq{\label{ineq:nu_n+1+sin-theta}
    \abs{\nu_{n+1}+\sin\theta}
    \leq\frac{\sin\theta}{2}.
}
Since $\nu_{n+1}<0$, this implies $\abs{\nu_{n+1}}
    \geq\frac{\sin\theta}{2}$,
and therefore
\eq{
    \abs{Du}^2
    =
    \frac{1-\nu_{n+1}^2}{\nu_{n+1}^2}
    \leq
    \frac{4}{\sin^2\theta}-1.
}

Finally,
\eq{
    Du-\cot\theta e_1
    =
    \left(
        -\frac{\nu_1}{\nu_{n+1}}-\cot\theta,
        -\frac{\nu_2}{\nu_{n+1}},
        \cdots,
        -\frac{\nu_n}{\nu_{n+1}}
    \right),
}
and hence
\eq{
    \abs{Du-\cot\theta e_1}^2
    =
    \frac{1}{\nu_{n+1}^2}
    \left(
        \left(\nu_1+\cot\theta\nu_{n+1}\right)^2
        +\sum_{i=2}^n\nu_i^2
    \right).
}
Moreover,
\eq{
    \left(\nu_1+\cot\theta\nu_{n+1}\right)^2
    =
    \left(
        \nu_1-\cos\theta
        +\cos\theta\frac{\sin\theta+\nu_{n+1}}{\sin\theta}
    \right)^2
    \leq
    2(\nu_1-\cos\theta)^2
    +2\cot^2\theta(\sin\theta+\nu_{n+1})^2.
}
It follows that
\eq{
    \abs{Du-\cot\theta e_1}^2
    \leq
    \frac{4}{\sin^2\theta}
    \left(3+2\cot^2\theta\right)
    \abs{\nu-\nu_{-\theta}}^2.
}
Combining this with \eqref{ineq:normal-g-estimate-downward} proves
assertion \textup{(2)} and completes the proof.
\end{proof}
}

{
\begin{lemma}\label{Lem:g_theta-controlled-by-model-normals}
Let $n\geq2$ and $\theta\in(0,\pi)$, and let
$\nu_\theta,\nu_{-\theta}$ be defined as in
\eqref{defn:nu_theta}. For every unit vector
$\nu=(\nu_1,\ldots,\nu_{n+1})\in\mathbb S^n$ and every
$k\in(0,1]$, we have
\eq{
g_{\theta,k}^2
\leq
\min\left\{
    \abs{\nu-\nu_\theta}^2,
    \abs{\nu-\nu_{-\theta}}^2
\right\}.
}

\end{lemma}

\begin{proof}
Since $k\in(0,1]$, we have
\begin{align}
g_{\theta,k}^2
&=
1-\nu_1^2-\nu_{n+1}^2
+k(\nu_1-\cos\theta)^2 \nonumber\\
&\leq
1-\nu_1^2-\nu_{n+1}^2
+(\nu_1-\cos\theta)^2 \nonumber\\
&=
\abs{\nu-\nu_\theta}^2
-(\nu_{n+1}-\sin\theta)^2 \nonumber\\
&\leq
\abs{\nu-\nu_\theta}^2.
\label{ineq:g_theta,k-4}
\end{align}
Similarly,
\begin{align}
g_{\theta,k}^2
&\leq
1-\nu_1^2-\nu_{n+1}^2
+(\nu_1-\cos\theta)^2 \nonumber\\
&=
\abs{\nu-\nu_{-\theta}}^2
-(\nu_{n+1}+\sin\theta)^2 \nonumber\\
&\leq
\abs{\nu-\nu_{-\theta}}^2.
\label{ineq:g_theta,k-3}
\end{align}
Combining \eqref{ineq:g_theta,k-4} and
\eqref{ineq:g_theta,k-3} proves the assertion.

\end{proof}
}

\bibliography{BibTemplate}%
\bibliographystyle{amsalpha}
\end{document}